\documentclass{amsart}
\usepackage{amsfonts,amscd,amssymb,amsmath,amsthm,mathdots,mathtools,upgreek}
\usepackage{graphicx,caption,subcaption,mathrsfs,appendix}
\usepackage{color}
\usepackage{bm}
\usepackage{ytableau}
\usepackage[colorlinks=true,linkcolor=blue]{hyperref}
\hypersetup{pdfstartview=FitW}

\numberwithin{equation}{section}

\usepackage{scalerel,stackengine}
\stackMath
\newcommand\reallywidecheck[1]{%
\savestack{\tmpbox}{\stretchto{%
  \scaleto{%
    \scalerel*[\widthof{\ensuremath{#1}}]{\kern-.6pt\bigwedge\kern-.6pt}%
    {\rule[-\textheight/2]{1ex}{\textheight}}
  }{\textheight}%
}{0.6ex}}%
\stackon[1pt]{#1}{\scalebox{-0.8}{\tmpbox}}%
}

\usepackage{xparse}
\definecolor{airforceblue}{rgb}{0.36, 0.54, 0.66}
\definecolor{amethyst}{rgb}{0.6, 0.4, 0.8}
\definecolor{applegreen}{rgb}{0.55, 0.71, 0.0}

\def\corAB{}
\def\corOZ{}

\definecolor{purple}{rgb}{0.9,0,0.8}

\newcommand{\abbr}[1]{{\sc\lowercase{#1}}}

\begin{document}

\newtheorem{theorem}{Theorem}
\newtheorem{proposition}[theorem]{Proposition}
\newtheorem{conjecture}[theorem]{Conjecture}
\newtheorem{corollary}[theorem]{Corollary}
\newtheorem{lemma}[theorem]{Lemma}
\newtheorem{dfn}{Definition}
\newtheorem{assumption}[theorem]{Assumption}
\newtheorem{claim}[theorem]{Claim}
\theoremstyle{definition}
\newtheorem{remark}[theorem]{Remark}

\numberwithin{theorem}{section}
\numberwithin{dfn}{section}


\newcommand{\B}{\mathbb{B}}
\newcommand{\R}{\mathbb{R}}
\newcommand{\T}{\mathcal{T}}
\newcommand{\C}{\mathbb{C}}
\newcommand{\D}{\mathbb{D}}
\newcommand{\G}{\mathcal{G}}
\newcommand{\Z}{\mathbb{Z}}
\newcommand{\Q}{\mathbb{Q}}
\newcommand{\E}{\mathbb E}
\renewcommand{\P}{\mathbb P}
\newcommand{\N}{\mathbb N}
\newcommand{\gd}{\mathfrak{d}}
\newcommand{\gb}{\mathfrak{b}}
\newcommand{\gL}{\mathfrak{L}}
\newcommand{\vep}{\varepsilon}
\newcommand{\cS}{\mathcal{S}}
\newcommand{\cI}{\mathcal{I}}
\newcommand{\cY}{\mathcal{Y}}
\newcommand{\cB}{\mathcal{B}}
\newcommand{\cM}{\mathcal{M}}
\newcommand{\cN}{\mathcal{N}}
\newcommand{\cA}{\mathcal{A}}
\newcommand{\cX}{\mathcal{X}}
\newcommand{\frakg}{\mathfrak{g}}
\newcommand{\supp}{{\mbox{\rm Supp }}}

\newcommand{\barray}{\begin{eqnarray*}}
\newcommand{\earray}{\end{eqnarray*}}

\newcommand{\dvec}[1]{ \llbracket #1 \rrbracket}

\newcommand{\Def}{:=}


\DeclareDocumentCommand \Pr { o }
{%
\IfNoValueTF {#1}
{\operatorname{Pr}  }
{\operatorname{Pr}\left[ {#1} \right] }%
}
\newcommand{\Prob}{\Pr}
\newcommand{\Exp}{\mathbb{E}}
\newcommand{\expect}{\mathbb{E}}
\newcommand{\1}{\one}
\newcommand{\Pto}{\overset{\mathbb{P}}{\to} }
\newcommand{\weakto}{\Rightarrow}
\newcommand{\lawequals}{\overset{\mathcal{L}}{=}}
\newcommand{\prob}{\Pr}
\newcommand{\pr}{\Pr}
\newcommand{\filt}{\mathscr{F}}
\newcommand{\ohadI}{\mathbbm{1}}
\DeclareDocumentCommand \one { o }
{%
\IfNoValueTF {#1}
{\ohadI }
{\ohadI\left\{ {#1} \right\} }%
}
\newcommand{\sgn}{\operatorname{sgn}}
\newcommand{\Bernoulli}{\operatorname{Bernoulli}}
\newcommand{\Binomial}{\operatorname{Binom}}
\newcommand{\Binom}{\Binomial}
\newcommand{\Poisson}{\operatorname{Poisson}}
\newcommand{\Exponential}{\operatorname{Exp}}

\newcommand{\Var}{\operatorname{Var}}
\newcommand{\Cov}{\operatorname{Cov}}


\newcommand{\Id}{\operatorname{Id}}
\newcommand{\diag}{\operatorname{diag}}
\newcommand{\tr}{\operatorname{tr}}
\newcommand{\proj}{\operatorname{proj}}
\newcommand{\Span}{\operatorname{span}}


\newcommand{\nicefrac}{\frac}
\newcommand{\half}{\frac12}
\DeclareDocumentCommand \JB { O{n} O{\lambda} } {J_{{#1}}({#2})}

\DeclareDocumentCommand \LP { O{\ESD} } {U_{ {#1} }}
\newcommand{\LPL}{ \LP[{\mu_N}] }

\newcommand{\ESD}{ L_N^{\model} }

\newcommand{\model}{\mathcal{M}}
\newcommand{\SVD}{\Sigma}
\newcommand{\LL}{\mathcal{L}}
\newcommand{\PI}{\Pi}
\DeclareDocumentCommand \PG { O{n} }
{
\mathfrak{S}_{{ #1 }}
}
\newcommand{\RS}{\mathcal{C}}

\newcommand{\TODO}[1]{ {\bf TODO: #1} }

\newcommand{\row}{X}
\newcommand{\col}{Y}
\newcommand{\srow}{x}
\newcommand{\scol}{y}
\newcommand{\csrow}{w}
\newcommand{\cscol}{z}
\newcommand{\COMP}[1]{ \check{#1} }
\newcommand{\gy}{\mathfrak y}
\newcommand{\gx}{\mathfrak x}
\newcommand{\gP}{\mathfrak{P}}
\newcommand{\gX}{\mathfrak{X}}
\newcommand{\gY}{\mathfrak{Y}}
\newcommand{\gc}{\mathfrak{c}}
\newcommand{\cL}{\mathcal{L}}
\newcommand{\gz}{\mathfrak{z}}
\newcommand{\gG}{\mathfrak{G}}
\newcommand{\gH}{\mathfrak{H}}
\newcommand{\gs}{\mathfrak{s}}
\newcommand{\wc}{\reallywidecheck}

\title[Outliers of random perturbation of Toeplitz] {Outliers of random perturbations of Toeplitz matrices with finite symbols}
\author[A.\ Basak]{Anirban Basak$^*$}
 \address{$^*$International Centre for Theoretical Sciences
 \newline\indent Tata Institute of Fundamental Research 
 \newline\indent Bangalore 560089, India}
\author[O.\ Zeitouni]{Ofer Zeitouni$^{\mathsection}$}
\address{$^{\mathsection}$Department of Mathematics, Weizmann Institute of Science 
 \newline\indent POB 26, Rehovot 76100, Israel
 \newline \indent and
 \newline\indent Courant Institute, New York University
 \newline \indent 251 Mercer St, New York, NY 10012, USA}


\begin{abstract}
Consider an $N\times N$ Toeplitz matrix $T_N$ with symbol ${\bm a }(\lambda) := \sum_{\ell=-d_2}^{d_1} a_\ell \lambda^\ell$, perturbed by 
an additive noise matrix $N^{-\gamma} E_N$,
where the entries of $E_N$ are centered  
i.i.d.~random variables of unit variance and $\gamma>1/2$. 
It is known that the empirical measure of eigenvalues 
of the perturbed matrix converges weakly, as $N\to\infty$,  to the law of ${\bm a}(U)$, where $U$ is distributed uniformly on $\mathbb{S}^1$.  
In this paper, we consider the outliers, i.e.~eigenvalues that are at a positive ($N$-independent) distance from ${\bm a}(\mathbb{S}^1)$. We prove that there are no outliers outside ${\rm spec} \, T({\bm a})$, the spectrum of the limiting Toeplitz operator, with probability approaching one, as $N \to \infty$. \corAB{In contrast,} in ${\rm spec}\, T({\bm a})\setminus {\bm a}({\mathbb S}^1)$ the process of outliers converges to the point process described by the zero set of certain random \corAB{analytic} functions. The limiting random \corAB{analytic} functions can be expressed as linear combinations of the determinants of finite sub-matrices of an infinite dimensional matrix, whose entries are i.i.d.~having the same law as that of $E_N$. The coefficients in the linear combination depend on the roots of the polynomial $P_{z, {\bm a}}(\lambda):= ({\bm a}(\lambda) -z)\lambda^{d_2}=0$ and semi-standard Young Tableaux with shapes determined by the number of roots of $P_{z,{\bm a}}(\lambda)=0$ that are greater than one in moduli.
\end{abstract}
\date{\today}
\maketitle

\section{Introduction}
Let ${\bm a} : \C \mapsto \C$ be a Laurent polynomial. That is, 
\begin{equation}\label{eq:laurent-poly}
{\bm a }(\lambda) := \sum_{\ell=-d_2}^{d_1} a_\ell \lambda^\ell, \qquad \lambda \in \C,
\end{equation}
for some $d_1, d_2 \in \N$ and some sequence of 
complex numbers $\{a_\ell\}_{\ell=-d_2}^{d_1}$, so that $d_1>0$\footnote{We remark that if one is interested in the case where $d_1=0$ but $d_2>0$,
one may simply consider, when computing spectra, the transpose of $T_N$ or of $M_N=T_N+\Delta_N$. For this reason, the restriction to $d_1>0$ does not reduce 
generality.}.
Define
$T({\bm a}): \C^\N \mapsto \C^\N$ to be the Toeplitz operator with symbol ${\bm a}$, that is the operator  given by
\[
(T({\bm a}) \underline x)_i := \sum_{\ell=-d_2}^{d_1} a_\ell x_{\ell+i}, \qquad \text{ for } i \in \N, \quad \text{ where } \underline x := (x_1, x_2, \ldots) \in \C^\N,
\]
and we set $x_i =0$ for non-positive integer values of $i$. 
For $N \in \N$, we denote by $T_N({\bm a})$ the natural $N$-dimensional truncation of the infinite dimensional Toeplitz operator $T({\bm a})$. As a matrix of dimension $N \times N$, we have 
(for $N>\max(d_1,d_2)$) that
\[
T_N:=T_N({\bm a}):=\begin{bmatrix}
a_{0} & a_{1} & a_2 & \cdots & \cdots & 0\\
a_{-1}& a_{0} & a_1 &\ddots &  & \vdots\\
a_{-2} &a_{-1}& \ddots & \ddots & \ddots & \vdots \\
\vdots & \ddots & \ddots & \ddots & a_1 & a_2\\
\vdots & & \ddots & a_{-1} & a_0 & a_1\\
0 & \cdots & \cdots & a_{-2} &a_{-1} & a_{0}
\end{bmatrix}.
\]
In general, $T_N$ is not a normal matrix, and thus its spectrum can be sensitive
to small perturbations. In this paper, we will be interested in the spectrum of 
$M_N:=T_N+\Delta_N$, where $\Delta_N$ is a ``vanishing'' random perturbation, and especially in
outliers, i.e.~eigenvalues that are at positive distance from the limiting spectrum.

Let $L_N$ denote the empirical measure of eigenvalues $\{\lambda_i\}_{i=1}^N$ of $M_N$, i.e.
$L_N:=N^{-1}\sum_{i=1}^N \delta_{\lambda_i}$, where $\delta_x$ is the Dirac measure at $x$.
It has been shown in \cite{BPZ-non-triang} that under a fairly general condition on the (polynomially vanishing)
noise matrix $\Delta_N$, $L_N$ converges (weakly, in probability) to
${\bm a}_* {\rm Unif}(\mathbb{S}^1)$, 
where ${\rm Unif}(\mathbb{S}^1)$ denotes the Haar measure on $\mathbb{S}^1:=\{z \in \C: |z| =1\}$.
That is, the limit is the law of
${\bm a}(U)$ where $U \sim {\rm Unif}(\mathbb{S}^1)$ (see also \cite{SV3} for the case of Gaussian noise).
However, simulations (see Figure \ref{fig:1}) suggest
that although the bulk of the eigenvalues approach ${\bm a}(\mathbb{S}^1)$, as $N \to \infty$, there are a few eigenvalues of $M_N$
that wander around outside a small neighborhood of 
${\bm a}(\mathbb{S}^1)$. Following standard terminology, we call them outliers. The goal of this paper is to characterize these outliers.

\begin{figure}[h]
\begin{center}
  \includegraphics[width=2.2in]{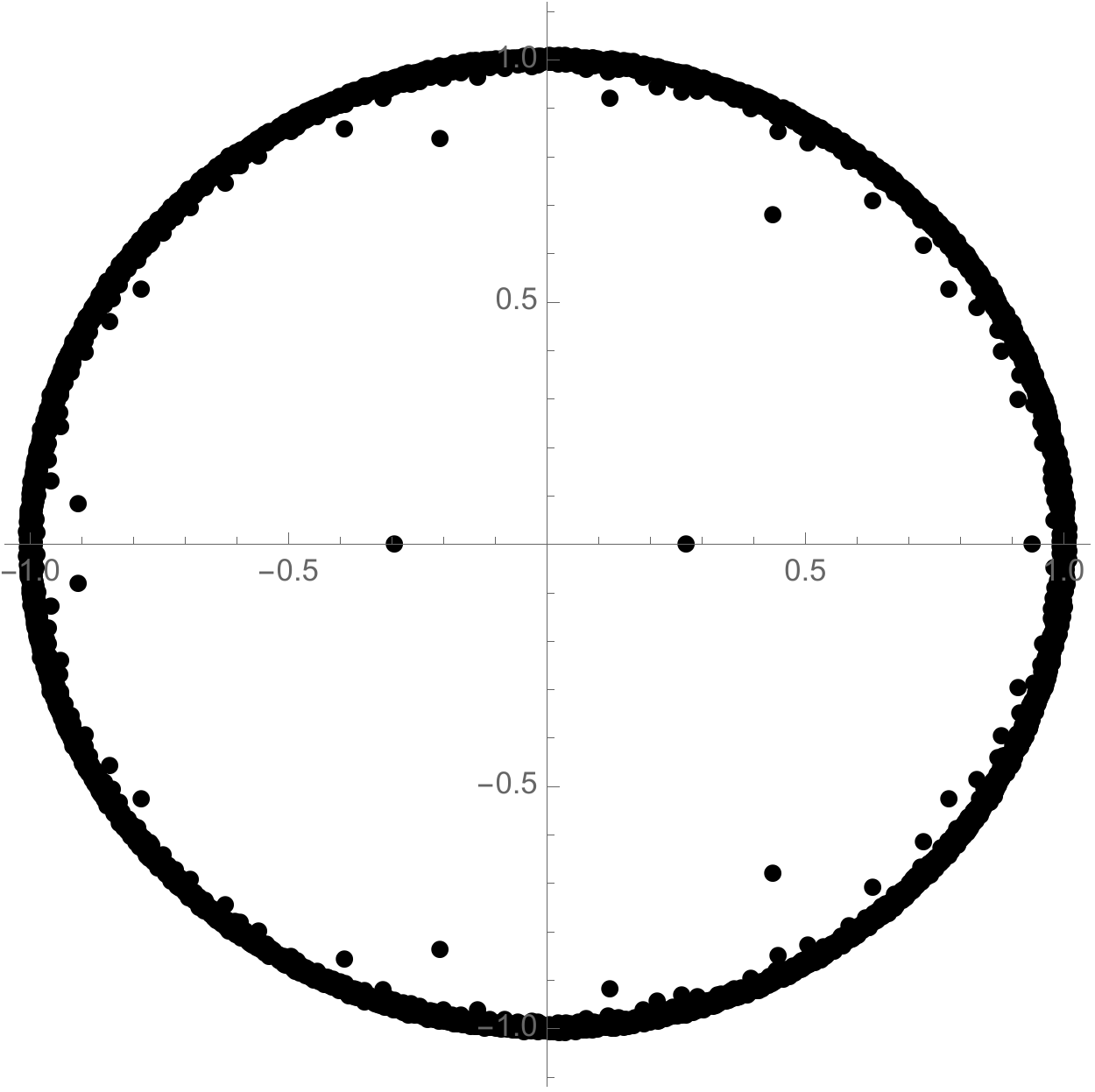}
  \qquad
  \includegraphics[width=2.2in]{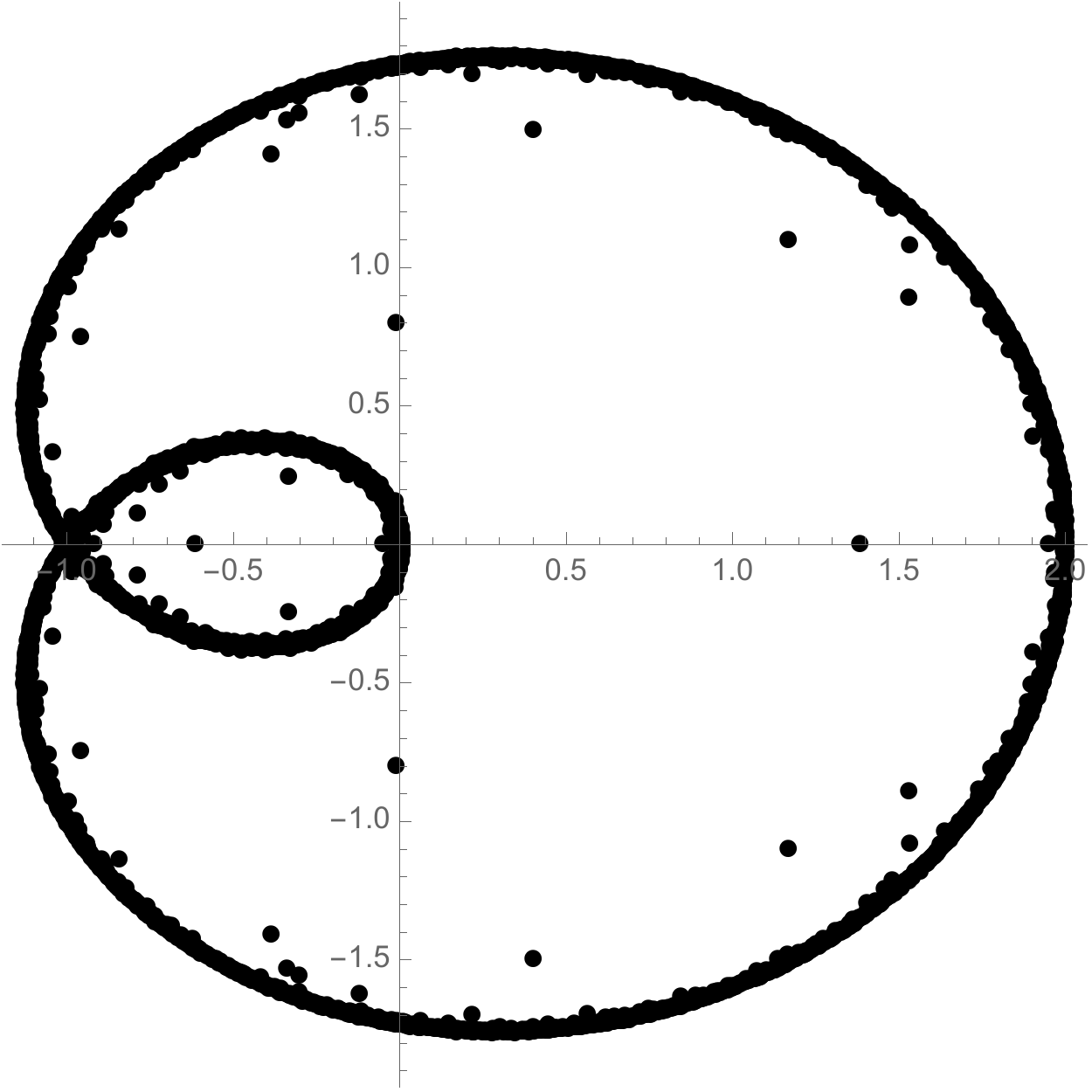}
  \includegraphics[width=2.2in]{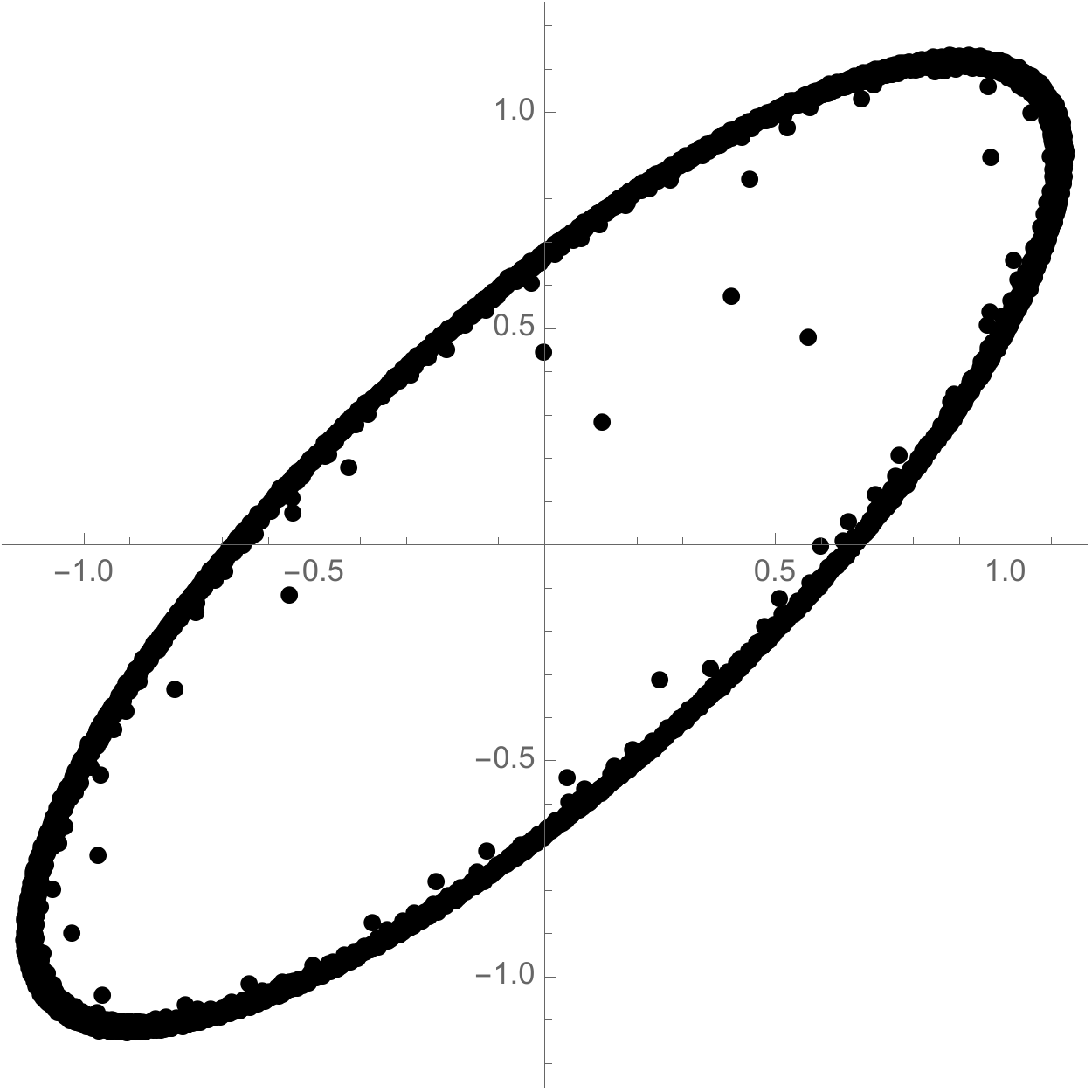}
\caption{
 The eigenvalues of $T_N({\bm a})+N^{-\gamma} E_N$, with $N=4000$, $\gamma=0.75$, $E_N$ a real Ginibre matrix, and various symbols ${\bm a}$. On the top
 left,
 ${\bm a}(\lambda)=\lambda$. On the top right, ${\bm a}(\lambda)=\lambda+\lambda^2$. On the
 bottom, ${\bm a}(\lambda)=\lambda^{-1} + 0.5 i \lambda$.
}
\label{fig:1}
\end{center}
\end{figure}

In  simulations depicted in 
Figure \ref{fig:1}, all eigenvalues of 
$M_N=T_N({\bm a}) + N^{-\gamma} E_N$, where $\gamma =0.75$ and $E_N$ is a standard real Ginibre matrix, are inside the unit disk, the lima\c con, and the ellipse when the symbol is ${\bm a}(\lambda)=\lambda, \lambda+ \lambda^2$, and $\lambda^{-1} + 0.5 i \lambda$, respectively. It follows from standard results on the spectrum of Toeplitz operators, e.g.~\cite[Corollary 1.12]{bottcher-finite-band} that these regions are precisely ${\rm spec}\ T({\bm a})$, the spectrum of the Toeplitz operator 
$T({\bm a})$ acting on $L^2(\mathbb{N})$.
Thus, Figure \ref{fig:1} suggests that there are no outliers outside ${\rm spec} \ T({\bm a})$. In our first result,  Theorem \ref{thm:no-outlier} below, we confirm this and prove the universality of this phenomenon for any finitely banded Toeplitz matrix, $\gamma >\frac12$, and under a minimal assumption on the entries of the noise matrix. 

We introduce the following standard notation: for $\D \subset \C$ and $\vep >0$, let $\D^\vep$ be the $\vep$-fattening of $\D$. That is, $\D^\vep:=\{z \in \C: {\rm dist} (z, \D) \le \vep\}$, where ${\rm dist}(z, \D):= \inf_{z' \in \D} |z-z'|$. Further  denote $\D^{-\vep}:= ((\D^c)^\vep)^c$. 

\begin{theorem}\label{thm:no-outlier}
  Let ${\bm a}$ be a Laurent polynomial. {Let}
  $T({\bm a})$ {denote}
the Toeplitz operator on $L^2(\mathbb{N})$ with symbol ${\bm a}$,
 and let $T_N({\bm a})$  be its natural $N$-dimensional projection. 
Assume $\Delta_N= N^{-\gamma} E_N$ for some $\gamma >\frac12$, where the entries of $E_N$ are independent (real or complex-valued) with zero mean and unit variance.
 Further, let $L_N$ 
denote the empirical measure of eigenvalues of $T_N({\bm a})+\Delta_N$. Fix $\vep>0$. Then,
\begin{equation}\label{eq:no-outlier}
\P\left( L_N\left(\left( \left({\rm spec} \ T({\bm a})\right)^c\right)^{-\vep}\right) =0\right) \to 1 \qquad \text{ as } N \to \infty.
\end{equation}
\end{theorem}
{In the terminology of  \cite{SV3},
$\C \setminus {\rm spec} \, T({\bm a})$ is a 
{\em zone of spectral stability} for $T_N({\bm a})$.}
The following remarks discuss some generalizations and extensions of Theorem \ref{thm:no-outlier}.

\begin{remark}
For clarity of presentation, in Theorem \ref{thm:no-outlier} we assume the entries of $E_N$ to have a unit variance. The 
proof shows that the same conclusion continues to hold under the assumption that the entries of $E_N$ are jointly independent (possibly having $N$-dependent distributions), and have zero mean and uniformly bounded second moment, i.e.
\begin{equation}\label{eq:gen-cond}
\E [E_N(i,j)] =0 \ \forall i,j  \in \{1,2,\ldots, N\} \qquad \text{ and } \qquad \sup_{N \in \N} \, \max_{i,j=1}^N \E[E_N(i,j)^2] < \infty,
\end{equation}
where $E_N(i,j)$ denotes the $(i,j)$-th entry of $E_N$.

We emphasize that under the general assumption \eqref{eq:gen-cond} on the entries of $E_N$,
one may not have the convergence of the empirical measure of the eigenvalues of $T_N({\bm a})+N^{-\gamma} E_N$ to ${\bm a}(U)$. Theorem \ref{thm:no-outlier} shows that even under such perturbations there are no eigenvalues in the complement of ${\rm spec} \, T({\bm a})$. 
\end{remark}

\begin{remark}\label{rmk:poly-vanish}
  The proof of Theorem \ref{thm:no-outlier} shows that {its conclusion}
  continues to hold if $\Delta_N = \mathfrak{a}_N E_N$, with
  any sequence $\{\mathfrak{a}_N\}_{N \in \N}$ such that $\mathfrak{a}_N = o(N^{-1/2})$ (the standard notation $a_N = o(b_N)$ means that $\lim_{N \to \infty} a_N/b_N=0$). For {conciseness, we only consider $\mathfrak{a}_N=
  N^{-\gamma}$.}
\end{remark}

\begin{remark}\label{rmk:varying-eps}
Theorem \ref{thm:no-outlier} shows that, with probability approaching one, all eigenvalues of the random perturbation of $T_N({\bm a})$ are contained in an $\vep$-fattening of the spectrum of 
the infinite dimensional Toeplitz operator $T({\bm a})$. Here we have chosen to work with a fixed parameter $\vep>0$. With some additional efforts it can be shown that in \eqref{eq:no-outlier} one can allow $\vep =\vep_N$ to decay to zero slowly with $N$. We do not pursue this direction.
\end{remark}

\begin{remark}\label{rmk:spec-rad}
The ideas used to prove Theorem \ref{thm:no-outlier} also show that 
{the sequence $\{\varrho_N\}_{N \in \N}$ is} tight, where $\varrho_N$ is the spectral radius (maximum modulus eigenvalue) of $E_N/\sqrt{N}$, with $E_N$ as in Theorem \ref{thm:no-outlier}. See Proposition \ref{prop-specrad}. It has been conjectured in \cite{bordenave} that the spectral radius of a matrix with i.i.d.~entries of zero mean and variance $1/N$ converges to one, in probability. Thus Proposition \ref{prop-specrad} proves a weaker form of this conjecture.  
\end{remark}

\begin{remark}
\label{rmk:resolvent}
\corOZ{In \cite[\corAB{Proposition 3.13}]{SV3}, the authors show that \corAB{the resolvent of $T_N({\bm a})$ remains bounded in compact subsets of $\left({\rm spec} \ T({\bm a})\right)^c$}.
As noted in \cite{SV3}, 
this implies Theorem \ref{thm:no-outlier} in the Gaussian case because
in that case, the operator norm of $N^{-\gamma} E_N$ is bounded with high probability. For more general perturbations possessing only \corAB{four} or less moments, the operator norm of $N^{-\gamma} E_N$ is in general not bounded, for some $\gamma\in (1/2,1)$, and a similar argument fails.}
\end{remark}


We turn to the identification of 
the limiting law of the random point process consisting
of the outliers of $M_N$,
which by
Theorem  \ref{thm:no-outlier} are contained in  ${\rm spec} \ T({\bm a})$. Before stating the 
results, we review 
standard
definitions of random point processes and their notion of convergence, taken from
\cite{DV}.

For $\D \subset \C$ we let $\cB(\D)$  denote the Borel $\sigma$-algebra on it. Recall that a Radon measure on $(\D, \cB(\D))$ is a measure that is finite for all Borel compact subsets of $\D$. 
\begin{dfn}\label{dfn:conv}
A random point process $\zeta$ on an open  set $\D \subset \C$ is a probability measure on the space of all non-negative integer valued Radon measures on $(\D, \cB(\D))$. Given a sequence of random point processes $\{\zeta_n\}_{n \in \N}$ on $(\D, \cB(\D))$, we say that $\zeta_n$ converges weakly to a (possibly random) point process $\zeta$ on the same space, and write $\zeta_n \Rightarrow \zeta$, if for all compactly supported bounded real-valued continuous functions $f$,  
\[
\int f(z) d\zeta_n(z) \to \int f(z) d\zeta(z), \qquad \text{ in distribution, as } n \to \infty,
\]
when viewed as real-valued Borel measurable random variables.
\end{dfn}

Next we proceed to describe the limit. We will see below that the limit is given by the zero set a random analytic function, where the description of the limiting random analytic function differs across various regions in the complex plane. This necessitates the following definition.

\begin{dfn}[Description of regions]\label{dfn:cS-gd}
For any 
Laurent polynomial ${\bm a}$ as in \eqref{eq:laurent-poly}, set $P_{z,{\bm a}}(\lambda):=\lambda^{d_2} ({\bm a}(\lambda)-z)$. Writing $d:=d_1+d_2$,  let $\{-\lambda_\ell(z)\}_{\ell=1}^d$  be the roots of the equation $P_{z,{\bm a}}(\lambda)=0$ arranged in an non-increasing order of their moduli.
For $\gd$ an integer such that $-d_2 \le \gd  \le d_1$, set
\[
\cS_\gd:=\{ z \in \C\setminus {\bm a}(\mathbb{S}^1): d_0(z) = d_1 - \gd, \text{ where } d_0(z) \text{ such that } |\lambda_{d_0(z)}(z)| > 1 > |\lambda_{d_0(z)+1}(z)|\},
\]
where for convenience we set $\lambda_{d+1}(z) = 0$ and $\lambda_0(z)=\infty$ for all $z \in \C$.  
\end{dfn}

Note that for $z \in \C \setminus {\bm a}(\mathbb{ S}^1)$ all roots of the polynomial $P_{z,{\bm a}}(\lambda)=0$ have moduli different from one. Therefore for such values of $z$, $d_0(z)$ is well defined, and hence so is $\cS_\gd$. 

By construction, $\cup_{\gd=-d_2}^{\gd_1} \cS_\gd = \C\setminus {\bm a}(\mathbb{S}^1)$. 
Since, by \cite{BPZ-non-triang}, the bulk of the eigenvalues of $M_N$
approaches ${\bm a}(\mathbb{S}^1)$ in the large $N$ limit, to study the outliers we only need to analyze the roots of $\det(T_N({\bm a}) + \Delta_N - z\Id_N) =0$ that are in $\cup_{\gd =-d_2}^{d_1} \cS_\gd$. 

Before describing the limiting random field let us mention some relevant properties of the regions $\{\cS_\gd\}_{\gd=-d_2}^{d_1}$.  As  ${\bm a}(\cdot)$ is a Laurent polynomial satisfying \eqref{eq:laurent-poly} it is straightforward to check that for $z \in \cS_\gd$ we have ${\rm wind}_z \, ({\bm a}) =\gd$, where ${\rm wind}_z ({\bm a})$ denotes the winding number about $z$ of the closed curve induced by the map $\lambda \mapsto {\bm a}(\lambda)$ for $\lambda \in \mathbb{S}^1$. Thus $\{\cS_\gd\}_{\gd=-d_2}^{d_1}$ splits the complement of ${\bm a}(\mathbb{S}^1)$ according to the winding number. Moreover, as will be seen later, the description of the law of the limiting random point process differs across the regions $\{\cS_\gd\}_{\gd=-d_2}^{d_1}$. 

Furthermore, from  \cite[Corollary 1.12]{bottcher-finite-band} we have that 
\[
{\rm spec} \, T({\bm a}) = {\bm a}(\mathbb{S}^1) \cup \{ z \in \C \setminus {\bm a}(\mathbb{S}^1): {\rm wind}_z ({\bm a}) \ne 0\}.
\]
It was noted above that ${\rm wind}_z \, ({\bm a}) =0$ for $z \in \cS_0$. So 
\begin{equation}\label{eq:cS_0-spec}
\cS_0 = ({\rm spec} \, T({\bm a}))^c. 
\end{equation}
Hence in light of Theorem \ref{thm:no-outlier} we conclude that
to find the limiting law of the outliers it suffices to analyze the 
eigenvalues of $T_N({\bm a})+\Delta_N$ that are in $\cup_{\gd \ne 0} \cS_\gd$. 

Finally, we note that from the continuity of the roots of $P_{z,{\bm a}}(\lambda)=0$ in $z$ in the symmetric product topology (see \cite[Appendix 5, Theorem 4A]{whitney}) it follows that the regions $\{\cS_\gd\}_{\gd=-d_2}^{d_1}$ are open, and hence by Definition \ref{dfn:conv} the random point processes on those regions are well defined. 


\begin{remark}
We highlight that one or more of the regions $\{\cS_\gd\}_{\gd=-d_2}^{d_1}$ may be empty. For example, when ${\bm a}(\lambda)= \lambda^{-1} + 0.5 i \lambda$ the product of the moduli of the roots of $P_{z,{\bm a}}(\lambda)=0$ is $1/0.5 =2$. So both roots of $P_{z,{\bm a}}(\lambda)=0$ cannot be less than one in moduli. Therefore $\cS_1=\emptyset$ in this case. It can be checked that under this same set-up $\cS_0$ and $\cS_{-1}$ are the outside and the inside of the ellipse, respectively, in the bottom panel of Figure \ref{fig:1}.
Furthermore, if ${\bm a}(\lambda)= a_1 \lambda + a_{-1} \lambda^{-1}$ with $|a_1| = |a_{-1}|$ then both $\cS_1$ and $\cS_{-1}$ are empty.
\end{remark}

Fix an integer $\gd \ne 0$ such that $-d_2 \le \gd \le d_1$. As mentioned above, the limiting random point process in $\cS_\gd$ will given by the zero set of a
 random analytic function $\gP_\gd^\infty(\cdot)$ to be defined on $\cS_\gd$. The function $\gP_\gd^\infty(\cdot)$ can be written as a linear combination of determinants of $|\gd| \times |\gd|$ sub-matrices of the noise matrix, where the coefficients depend on the roots of the polynomial $P_{z, {\bm a}}(\lambda)=0$ and semistandard Young Tableaux of some given shapes with certain restrictions on its entries. {We recall}
  the definition of semistandard Young Tableaux  \cite[Section 7.10]{Stanley}.

\begin{dfn}[Semistandard Young Tableaux]
Fix $k \in \N$. A partition  ${\bm \mu}:=(\mu_1, \mu_2,\ldots,\mu_k)$ with $k$ parts  is a collection of non-increasing non-negative integers
$\mu_1\ge \mu_2\ge \cdots \ge \mu_k \ge 0$. Given a partition ${\bm \mu}$, a semistandard Young Tableaux of shape ${\bm \mu}$ is an array $\gx:=(\gx_{i,j})$ of positive integers of shape ${\bm \mu}$ (i.e.~$1 \le i \le k$ and $1 \le j \le \mu_i$) that is weakly increasing (i.e.~non-decreasing) in every row and strictly increasing {in}
every column. 
\end{dfn}

The limiting random field depends on the following subset of the set of all semistandard Young Tableaux, for which we have not found a standard terminology in the literature.

\begin{dfn}[Field Tableaux]\label{dfn:young-restrict}
Let ${\bm \mu}=(\mu_1,\mu_2,\ldots,\mu_k)$ be a partition with $k \ge d_2$ and $\mu_k >0$. For an integer $\gd \ne 0$ such that $-d_2 \le \gd \le d_1$, let
 $\gL({\bm \mu},\gd)$ denote the collection of  all semistandard Young Tableaux $\gx$ of shape ${\bm \mu}$
  that are strictly increasing along the southwest diagonals and satisfies the assumption $\gx_{i,1}=i$ for all $i \in [d_2]:=\{1,2,\ldots,d_2\}$. See Figure \ref{fig:2} for a pictorial illustration.
%
\end{dfn}


Equipped with the relevant notion of semistandard Young Tableaux we now turn to define the coefficients that appear when the limiting random analytic function is expressed as a linear sum of determinants of $|\gd| \times |\gd|$ sub-matrices of the noise matrix. 

\begin{dfn}[Field notation]\label{dfn:random fieldnot}
 Fix $\gd \ne 0$ an integer such that $-d_2 \le \gd \le d_1$. Set $d_0:=d_1 -\gd$. For any finite set $\gX:=\{x_1< x_2< \cdots <x_\ell\}$,  
 we define $\widehat \sgn(\gX)$ to be the sign of the permutation which places all elements of $\gX$ before those
 in $\{x_1, x_1+1,x_1+2,\ldots, x_\ell\}\setminus \gX$ but preserves the order of the elements within the two sets.
 
 {\bf The  case $\gd >0$.} Denote $\gL_1(\gd):=\gL({\bm \mu}_1, \gd)$ and $\gL_2(\gd):=\gL({\bm \mu}_2,\gd)$, where 
 \[
 {\bm \mu}_1:= (d,d-1,\ldots,d_0+1)
\quad \text{and}\quad 
 {\bm \mu}_2:= (d-d_0,d-d_0-1,\ldots,1),\]
 see  Definition \ref{dfn:young-restrict}.
 Given any $\gx \in \gL_1(\gd)$ and $\gy \in \gL_2(\gd)$, define 
 \[
\widehat \gX:=\widehat \gX(\gx, \gy):= \{\gx_{i,1}, i \in [\gd+d_2]\setminus [d_2]\}, \qquad \widehat \gY:=\widehat \gY(\gx, \gy):=\{\gy_{i,1}, i \in [\gd+d_2]\setminus[d_2]\}.
\]
Next define
\[
\gc(\gx,\gy):= \gc(\gx,\gy,z):=\prod_{i=1}^{d_0}\lambda_i(z)^{-c_i(\gx,\gy)} 
\cdot \prod_{i=d_0+1}^d \lambda_i(z)^{c_i(\gx,\gy)},
\]
where
\begin{equation}\label{eq:gc_i-gd-ge-0}
\gc_i(\gx,\gy):= \left\{\begin{array}{ll} \sum\limits_{j=1}^{d-d_0} (\gx_{j,i+1} - \gx_{j,i}+1) & \mbox{ for }i \in [d_0]\\ \gy_{1,d+1-i}+\gx_{1,i}-1 \corAB{-d_2-\gd}\\
\qquad \qquad \qquad  +\sum\limits_{j=2}^{i-d_0}(\gy_{j,d+1-i} -\gy_{j-1,d+2-i})+\sum\limits_{j=2}^{d+1-i}(\gx_{j,i} -\gx_{j-1,i+1})& \mbox{ for } i \in [d]\setminus [d_0]
\end{array}\right. ,
\end{equation}
and $\{\lambda_i(z)\}_{i=1}^d$ are as in Definition \ref{dfn:cS-gd}.

{\bf The case $\gd <0$.} Define $\gL_1(\gd):=\gL_1({\bm \mu}_1, \gd)$ and $\gL_2(\gd):=\gL_2({\bm \mu}_1, \gd)$, where now 
 \[
{\bm \mu}_1:= (\underbrace{d+1,d+1,\ldots,d+1}_{-\gd}, d, d-1,\ldots,d_0+1)  \text{ and }  {\bm \mu}_2:=(\underbrace{d+1,d+1,\ldots,d+1}_{-\gd}, d_2+\gd, d_2+\gd-1,\ldots,1),
 \]
 respectively. In the special case $\gd = -d_2$ the definitions of ${\bm \mu}_1$ and ${\bm \mu}_2$ simplify to
 \[
 {\bm \mu}_1 = {\bm \mu}_2 = (\underbrace{d+1,d+1,\ldots,d+1}_{d_2}).
 \]
 Given any $\gx \in \gL_1(\gd)$ and $\gy \in \gL_2(\gd)$ further denote
\[
\widehat \gX:=\widehat \gX(\gx, \gy):= \{\gy_{i,d+1}  \}_{i \in [-\gd]}, \qquad \widehat \gY:=\widehat \gY(\gx, \gy):=  \{\gx_{i,d+1} \}_{i \in [-\gd]}  
\]
and
\[
\gc(\gx,\gy):= \gc(\gx,\gy,z):= \prod_{i=1}^{d_0} \lambda_i(z)^{-\gc_i(\gx,\gy)}
 \cdot \prod_{i=d_0+1}^d \lambda_i(z)^{c_i(\gx,\gy)},
\]
where now
\[
c_i(\gx,\gy):= \left\{\begin{array}{ll} \sum\limits_{j=1}^{d_2} (\gx_{j,i+1}-\gx_{j,i}+1)+ \sum\limits_{j=1}^{-\gd} (\gy_{j,d+2-i} -\gy_{j,d+1-i}+1)   & \mbox{ for } i \in [d_0]\\  \gy_{1,d+1-i} +\gx_{1,i} -1\corAB{-d_2 +\gd} \\
\qquad \qquad \qquad + \sum\limits_{j=2}^{i-d_1}(\gy_{j,d+1-i} -\gy_{j-1,d+2-i})+\sum\limits_{j=2}^{d+1-\gd-i}(\gx_{j,i} -\gx_{j-1,i+1})  & \mbox{for } i \in [d]\setminus [d_0]
\end{array} \right. .
\]

For all values of $\gd\neq 0$,  set $\gz(\gx, \gy):= \widehat\sgn(\widehat \gX) \cdot \widehat\sgn(\widehat \gY)$.
\end{dfn}

 \noindent  Figure \ref{fig:2} gives a pictorial illustration of the definition.
  \begin{figure}[h]
 \captionsetup{singlelinecheck=false}
\centering	
\ytableausetup{centertableaux}
\begin{ytableau}
      \none  & 1  & 1 & 1 &2 & 4&6 \\
  \none  & 2 & 2  &3  &6 & 7 \\
  \none & 3 & 5  &9  & 9\\
  \none &  6 &10 & 12
\end{ytableau}
\qquad\qquad\qquad\qquad\qquad
\begin{ytableau}
       \none & 1& 1& 1&9 \\
  \none  & 2&2 &10\\
  \none & 3&11\\
  \none & 12
\end{ytableau}

\vskip10pt

\begin{ytableau}
           \none  & 1  & 1 & 1 &6 & 7&8 &8 \\
  \none  & 2&2 &10 &10 &10 &11\\
  \none & 3&12&12 & 12 & 12\\
  \end{ytableau}
\qquad\qquad\qquad\qquad\qquad
\begin{ytableau}
      \none  & 1  & 1 & 1 &8 & 8&11 &12 \\
  \none  & 2 & 2  \\
  \none & 3 
\end{ytableau}
\caption[]{Examples of $\gx \in \gL_1(\gd)$ (\texttt{left column}) and $\gy \in \gL_2(\gd)$ (\texttt{right column}) with $\max_i\gc_i(\gx,\gy) \le L$, where $\gL_1(\gd)$, $\gL_2(\gd)$, and $\gc_i(\gx,\gy)$ for $i \in [d]$ are as in Definition \ref{dfn:random fieldnot}. $d_1=d_2=3$, $L=\corAB{15}$. \texttt{Top} and \texttt{bottom rows} are $\gd =1$ and $\gd=-1$, respectively.
For the \texttt{top row}
\begin{align*}
\gc(\gx,\gy,z)=\lambda_1(z)^{-10}\cdot \lambda_2(z)^{-11} \cdot \lambda_3(z)^{\corAB{12}} \cdot \lambda_4(z)^{\corAB{3}} \cdot \lambda_5(z)^{\corAB{3}} \cdot \lambda_6(z)^{\corAB{5}},
\end{align*}
and for the \texttt{bottom row}
\begin{align*}
\gc(\gx,\gy,z)=\lambda_1(z)^{-14}\cdot \lambda_2(z)^{-15} \cdot \lambda_3(z)^{-9} \cdot \lambda_4(z)^{-12} \cdot \lambda_5(z)^{\corAB{7}} \cdot \lambda_6(z)^{\corAB{9}}.
\end{align*}
}
\label{fig:2}
\end{figure}
 
Having defined all necessary ingredients we now introduce the limiting random analytic function $\gP_\gd^\infty(\cdot)$.

 \begin{dfn}[Description of the random fields]\label{dfn:random field}
 Let $E_\infty$ denote  a semi-infinite array of i.i.d.~random variables $\{e_{i,j}\}_{i,j \in \N}$ with zero mean and unit variance. 
For $\gX, \gY \subset \N$, let $E_\infty[\gX;\gY]$ denote  the sub-matrix of $E_\infty$ induced by the rows and the columns indexed by 
$\gX$ and $\gY$, respectively.  With notation for $\gc(\gx,\gy)$,$\gz(\gx,\gy)$, $\widehat \gX$ and $\widehat \gY$ as in Definition \ref{dfn:random fieldnot},
we set,  for $z \in \cS_\gd$ and $L \in \N \cup\{\infty\}$,
\begin{equation}\label{eq:random-fld}
\gP_\gd^L(z) := \sum_{\gx \in \gL_1(\gd)} \sum_{\gy \in \gL_2(\gd)} \gc(\gx, \gy) \cdot {\bf 1}_{\{\max_i \gc_i(\gx, \gy) \le L\}}\cdot (-1)^{\gz(\gx,\gy)}\det(E_\infty[\widehat \gX; \widehat \gY]).
\end{equation}
 \end{dfn}

It may not be apriori obvious from Definition \ref{dfn:random field} that $\gP_\gd^\infty(\cdot)$ is well defined, 
as \eqref{eq:random-fld} is an infinite sum. 
{Lemma \ref{lem:limit} below will establish that it can be expressed as the local uniform limit of the random analytic functions $\{\gP_\gd^L(\cdot)\}_{L \in \N}$ and thus it is indeed a well defined random analytic function. In addition, under an
appropriate anti-concentration property of the entries of 
$E_N$ and $E_\infty$, the random analytic function $\gP_\gd^\infty(\cdot)$ is not identically zero on a set of probability one, and thus
the random point process induced by the zero set of it 
is a valid random Radon measure.} 

To describe the required anti-concentration property, we recall
  L\'{e}vy's concentration function, defined for any
(possibly complex-valued) random variable $X$ by
\[
\cL(X,\vep):= \sup_{w \in \C} \P(|X-w| \le \vep).
\]
Equipped with the above definition we now state the additional
assumption on the entries of $E_N$ and $E_\infty$.
\begin{assumption}[Assumption on the entries of the noise matrix]\label{ass:entry}
Assume that the entries of $E_N$ and $E_\infty$ are {either real-valued or complex-valued} i.i.d.~random variables with zero mean and unit variance, so that, for some absolute constants $\upeta  \in (0,{2}]$ and $C<\infty$,
\begin{equation}\label{eq:levy-bd}
{\cL(e_{1,1},\vep) \le C \vep^{\upeta}},
\end{equation}
 for all sufficiently small $\vep>0$, where $e_{1,1}$ is the first diagonal entry of $E_N$.
%
\end{assumption}

Note that any  random variable having a bounded density with respect to the Lebesgue measure {on the real line}, and the complex plane satisfies the bound \eqref{eq:levy-bd} with {$\upeta=1$ and $2$, respectively}. This in particular includes the standard real and complex Gaussian random variable. 

{Recall that a sequence of complex-valued functions $\{\mathfrak{f}_L\}_{L \in \N}$, defined on some open set $\D \subset \C$, is said to converge locally uniformly to a function $\mathfrak{f}: \D \mapsto \C$, if given any $z \in \D$ there exists some open ball $\D_z \subset \D$ containing $z$ such that $\mathfrak{f}_L$ converges to $\mathfrak{f}$ uniformly on $\D_z$, as $L \to \infty$. 
We now have the following.}

 \begin{lemma}[Description of the limit]\label{lem:limit}
 Let $\gd$, $\{\gP_\gd^L\}_{L \in \N \cup\{\infty\}}$, and $E_\infty$ be as in Definition \ref{dfn:random field}. We let $\zeta_\infty^\gd$ be the random point process induced by the zero set of the random field $\{\gP_\gd^\infty(z)\}_{z \in \cS_\gd}$. That is,
\begin{equation}
\label{eq:zeta-gd-L}
\zeta_\infty^\gd:= \sum_{z \in \cS_\gd: \gP_\gd^\infty(z)=0} \delta_z.
\end{equation}
Then we have the following:
 \begin{enumerate}
 
 \item[(i)] The functions $\{\gP_\gd^L\}_{L \in \N}$ are random analytic functions.
 
 \item[(ii)] The random function $\gP_\gd^\infty$ is well defined 
   on a set of probability one. Furthermore, $\gP_\gd^L$ converges locally uniformly to $\gP_\gd^\infty$, almost surely, as $L \to \infty$, and hence $\gP_\gd^\infty$ is a random analytic function.
 
 \item[(iii)] Under the additional 
     Assumption \ref{ass:entry}, \corOZ{the random function $\gP_\gd^\infty$ is almost surely  non constant.}
 \end{enumerate} 
\end{lemma}
Recalling Definition \ref{dfn:conv} we note that the notion of convergence of random point processes defined on $\cS_\gd$ is tested against continuous functions supported on compact subsets of $\cS_\gd$. Therefore when discussing convergence it is enough to consider continuous functions on sets $\cS_\gd^{-\vep}$ for arbitrary $\vep>0$.

\begin{remark} \label{rem-nonuniv}
We emphasize that the random point process $\zeta_\infty^\gd$, although free of the parameter $\gamma$, is not universal. That is, in general its law depends on the law of the entries of the matrix $E_\infty$.
\end{remark}

The main result of this paper shows that given an integer $\gd \ne 0$ such that $-d_2 \le \gd \le d_1$ and $\gamma >\frac12$,
the random point process induced by the eigenvalues of $T_N({\bm a})+N^{-\gamma} E_N$ that are in $\cS_\gd$ converges weakly to the random point process $\zeta_\infty^\gd$ induced by the zero set of the random analytic function $\gP_\gd^\infty$. 
%


\begin{theorem}\label{thm:outlier-law}
Let ${\bm a}$ be a Laurent polynomial as in \eqref{eq:laurent-poly}. 
{Let  $T({\bm a})$ denote}
the Toeplitz operator on $L^2(\mathbb{N})$ with symbol ${\bm a}$,
 and let $T_N({\bm a})$  be its natural $N$-dimensional projection. 
Assume $\Delta_N= N^{-\gamma} E_N$ for some $\gamma >\frac12$, where the entries of $E_N$ are 
i.i.d.~satisfying Assumption \ref{ass:entry}. Furthermore assume that the 
entries of $E_\infty$ {in Definition \ref{dfn:random field}
are i.i.d. of}  the same law 
as that of {the entries of}
$E_N$.
{For any
  integer $\gd \ne 0$ such that $-d_2 \le \gd \le d_1$,
  let $\Xi_N^\gd$ denote
  the random point process induced by the eigenvalues of $T_N({\bm a})+\Delta_N$ that are in $\cS_\gd$.} That is,
\[
\Xi_N^\gd:= \sum_{z \in \cS_\gd:\, \det(T_N({\bm a})+\Delta_N -z \Id_N)=0} \delta_z.
\]
{Then, for such $\gd$,}
$\Xi_N^\gd$ converges weakly, as $N\to\infty$, to
the \corAB{random} point process $\zeta_\infty^\gd$ {from Lemma \ref{lem:limit}.}
\end{theorem}
\noindent
Explicit expressions for the fields in the
statement of Theorem \ref{thm:outlier-law} for two of the 
symbols depicted in Figure \ref{fig:1} appear in Section \ref{sec-background} below,
see \eqref{eq:hgaf} and Remark \ref{rem-Lim}.

At a first glance, it may seem counter intuitive for the limit to be expressed as the zero set of certain random analytic function of the form \eqref{eq:random-fld}. To see that it is in fact natural, we note that the determinant of $T_N({\bm a}(z)) +N^{-\gamma} E_N$ can be expressed as 
a linear combination of products of determinants of 
sub-matrices of $T_N({\bm a}(z))$ and of $E_N$, and further
  that the determinants of (some) sub-matrices of a finitely 
  banded Toeplitz matrix can be expressed as certain 
  skew-Schur polynomials in $\{\lambda_i(z)\}_{i=1}^d$ (see \cite{Al, BD}),
  where these polynomials are defined as a sum of monomials \corAB{with the sum taken} over (skew) semistandard Young Tableaux of some given shapes. {This leads to
    \eqref{eq:random-fld}.} 

\begin{remark}
As before, when discussing convergence it is enough to consider functions supported on the sets 
$\cS_\gd^{-\vep}$ for arbitrary $\vep>0$. Similar to Remark \ref{rmk:varying-eps}, one can allow    in Theorem \ref{thm:outlier-law} $\vep=\vep_N$ to go to zero, as $N \to \infty$, sufficiently slowly, and consider functions supported on $\cS_\gd^{-\vep}$ as test functions. We do not work out the details 
here.
\end{remark}

\subsection{Background, related results, and extensions}
\label{sec-background}
The fact that the spectrum of non-normal matrices and operators
is not stable with respect
to perturbations is well known, see e.g.~\cite{trefethen} for a
comprehensive account and \cite{DH}
for a recent study. Extensive work has been done concerning worst case perturbations, which are captured through the notion of \textit{pseudospectrum}. However, beyond some specific examples the pseudospectrum of non-normal operators are not well characterized. Hence, in recent times there have been growing interests in studying the spectrum of non-normal operators and matrices under small \textit{typical} perturbations. See the references in \cite[Section 1]{SV}. We also refer to \cite[Section 1.3]{BPZ-twisted} for a discussion about the relation between the pseudospectrum and the spectrum under typical perturbation, and an extensive reference list. We add that early examples of the spectrum obtained by noisily perturbing Toeplitz matrices with finite symbols appeared in \cite{Trefethen1991}.


As mentioned above, the convergence of the empirical measure of eigenvalues for randomly perturbed finite-symbol Toeplitz matrices has now been established in
great generality, see the recent articles \cite{BPZ-non-triang,SV3} and references therein. Our focus in this paper concerns the study of \textit{outliers}. In Theorem \ref{thm:no-outlier} we identify the region where no outliers are present {(in the terminology of  Sj\"{o}strand and Vogel \cite{SV3}, this 
is the
{\em zone of spectral stability}).}
 Then, in Theorem \ref{thm:outlier-law} we find the limit of the random processes induced by the outliers in the interior of the complement of the region identified in Theorem \ref{thm:no-outlier}. 

For the Jordan matrix, i.e.~the Toeplitz matrix with symbol ${\bm a}(\lambda)=\lambda$, \cite[Theorem 2]{DH} shows that there are no outliers outside the unit disc (centered at zero) in the complex plane, with high probability. \corOZ{In the general Toeplitz case, Theorem \ref{thm:no-outlier} follows
(for Gaussian perturbations)   
from the resolvent estimates in \cite[Proposition 3.13]{SV3}, see Remark \ref{rmk:resolvent}.}

{Some bounds on the number of outliers inside ${\rm spec} \, T({\bm a})$
are available in the literature. In the notation} of the current 
 paper, for the Jordan matrix perturbed by additive complex Gaussian noise, with
$\gamma>3$, a logarithmic in $N$ bound 
for the number of outliers {appears in}
\cite{DH}. Similar results (with worse error bounds) are given in \cite{SV} for non-triangular tridiagonal Toeplitz matrices (i.e.~the symbol is ${\bm a}(\lambda)= a_1 \lambda + a_{-1} \lambda^{-1}$), and in \cite{SV3} for  general Toeplitz matrices with finite symbol.

Sharper results concerning 
outliers for the Jordan matrix and the non-triangular tridiagonal Toeplitz matrix (under complex Gaussian perturbation),
are presented in \cite{SV1,SV2}. In both these cases, a sharp $O(1)$ control on the number of outliers in the regions $\cS_\gd$ with $\gd \ne 0$ {is}
provided. In the language of the current paper, the authors compute the {\em first intensity measure} of the limiting field 
$\zeta_\infty^\gd$, that is, they compute the function $\rho_\gd(\D)=\E[\zeta_\infty^\gd(\D)]$ for subsets $\D \subset \cS_\gd$. 

{For 
the Jordan matrix,} it has been shown in \cite[Theorem 1.1]{SV2} that $\rho_1(\cdot)$ has a density with respect to the two dimensional Lebesgue measure,
given by 
\[
\rho_1(dz):= \frac{2}{\pi (1-|z|^2)^2} {\bf 1}_{\{|z| < 1\}} dz.
\]
Due to {the}
Edelman-Kostlan formula (see \cite[Theorem 3.1]{EK}),
$\rho_1(\cdot)$ is the first intensity measure of the random point process induced by the zero set of the hyperbolic Gaussian analytic function (see \cite[Chapter 2.3]{HKPV} for a definition), given by
\begin{equation}\label{eq:hgaf}
\mathfrak{F}(z):= \sum_{k=0}^\infty z^k \mathfrak{g}_k \sqrt{k+1},
\end{equation}
where $\{\mathfrak{g}_k\}_{k \in \N}$ are i.i.d.~standard complex Gaussian random variables. 
We now explain how to recover this result from Theorem 
  \ref{thm:outlier-law}: in the case of the Jordan matrix, ${\bm a}(\lambda)=
  \lambda$ and then $\lambda_1(z)=z$,  $\gd=1$ and $d_0=0$. 
  Substituting in Definition \ref{dfn:random fieldnot}, one finds that
  \[
\gc(\gx,\gy)=  
z^{c_1(\gx,\gy)}, \quad 
 c_1(\gx,\gy)=\gy_{1,1}+\gx_{1,1}-\corAB{2},
 \]
 where $\gx_{1,1},\gy_{1,1}$ are arbitrary \corAB{positive} integers. In particular, there are 
 precisely $k+1$ choices of such integers that give $c_1(\gx,\gy)=k$.
 \corAB{Since the entries of $E_\infty$ are i.i.d.~Gaussian}, it follows from
 \eqref{eq:random-fld} that $\gP_\gd^\infty(z)$ coincides with 
 \eqref{eq:hgaf}. Together with the Edelman-Kostlan formula, this
 recovers \cite[Theorem 1.1]{SV2}. Note that the same expression \eqref{eq:hgaf}
 also holds
 for real Gaussian noise, with the limiting field having now real Gaussian coefficients.

 \begin{remark}
   \label{rem-Lim}
   Theorem \ref{thm:outlier-law} allows one to describe
 all limiting outlier processes for the other cases depicted in 
 Figure \ref{fig:1}. We give one more example, for the symbol
 ${\bm a}(\lambda)=\lambda+\lambda^2$ (the ``lima\c{c}on''),
 with Gaussian noise matrix $E_N$ consisting of i.i.d. centered entries.
 There, we have that $\lambda_i(z)=(\pm\sqrt{1+4z}-1)/2$, 
 arranged so that $|\lambda_1(z)|\geq |\lambda_2(z)|$.
 In the region $\cS_1$, which corresponds to the region 
 enscribed by the lima\c{c}on curve with winding number $1$, 
 the limiting field is 
 \begin{equation}
   \label{eq-lim1} \sum_{k=0}^\infty  \lambda_2(z)^k \mathfrak{g}_k \sqrt{k+1},
 \end{equation}
 where $\mathfrak{g}_k$ are i.i.d. centered Gaussian,
 compare with \eqref{eq:hgaf}. On the other hand, in the regions
 $\cS_2$ (which corresponds to the region of winding number 2, i.e. the smaller 
 region in Figure \ref{fig:1}), the limiting field admits a more complicated
 description, as follows. Let $\{g_{ij}\}$ be i.i.d.
 Gaussians, and for $i,j,k,l$ integers satisfying $i<j$ and  $k<l$,
 introduce 
 the random variables 
 $$X_{ijkl}:=g_{ik}g_{jl}-g_{il} g_{jk}$$ and the
 functions 
 $$W_{ijkl}(z):= \sum_{s,t: i\leq s < j, k\leq t <l} 
 \left( \frac{\lambda_2(z)}{\lambda_1(z)} \right)^{s-t}.$$
 Then,
 \begin{equation}
   \label{eq-coronaday2}
   \gP_{\gd}^\infty (z)=\sum_{1\leq i<j, 1\leq k<l} \lambda_1(z)^{i+j-3}\lambda_2(z)^{k+l-3}
   (-1)^{(j-i-1)(l-k-1)} W_{ijkl}(z) X_{ijkl}.
 \end{equation}
 In particular, the random coefficients in \eqref{eq-coronaday2}
 are in general not Gaussian, and 
 there are terms in the sum in \eqref{eq-coronaday2} with non-trivial
 correlations.
\end{remark}

  Sj\"{o}strand and Vogel in  \cite{SV1} compute $\rho_\gd$ 
  for the non-triangular tridiagonal Toeplitz matrix. 
  Again by the Edelman-Kostlan formula, they identify $\rho_\gd$ 
  with the first intensity measure of the random point 
  process induced by the zero set of some Gaussian analytic 
  function with some covariance kernel $\mathbb{K}_\gd(\cdot, \cdot)$.  
Our Theorem \ref{thm:outlier-law},
when applied to non-triangular tridiagonal Toeplitz matrix, again
shows that under complex Gaussian perturbation the 
limiting random fields are the zero sets of Gaussian analytic functions, 
and a computation (which we omit) shows that its covariance
kernel is given by $\mathbb{K}_\gd(\cdot,\cdot)$.
Thus, Theorem \ref{thm:outlier-law} again recovers the results 
of \cite{SV1}.  

We also mention the relevant work \cite{NV}, where local statistics
for the
eigenvalues of
random perturbations of certain pseudo-differential operators are computed 
and related to local statistics of the zeros of random Gaussian analytic 
functions. 

Based on \cite{SV1, SV2,NV} one may be tempted to predict that for general finitely banded Toeplitz matrices the limiting random field is the zero set of some Gaussian analytic function. Theorem \ref{thm:outlier-law} shows that, even under complex Gaussian perturbations, the limit may not be the zero set of Gaussian analytic functions, e.g.~consider ${\bm a}(\lambda)=\lambda+\lambda^2$ and the limit of the random point process induced by the outlier eigenvalues in $\cS_2$. Furthermore, even in the framework of \cite{SV1, SV2}, under general perturbation, as already mentioned in Remark \ref{rem-nonuniv}, the limit turns out to be non-universal.  

{The work of \'{S}niady \cite{Sn} considers situations where the
additive noise is Gaussian of standard deviation $\sigma N^{-1/2}$, and deals
with the limit where first $N\to\infty$ and then $\sigma\to 0$. 
Some of the subsequent work, reviewed e.g. \cite[Section 1.4]{BPZ-twisted}, 
can
be seen as an attempt to modify the order of limits. In this direction and 
concerning outliers,
Bordenave and Capitaine \cite{BC} study outliers of 
deformed i.i.d.~random matrices.} Namely, for a sequence of deterministic matrices $\{A_N\}_{N \in \N}$ they study the outlier eigenvalues of $\model_N^\sigma:=A_N + \frac{\sigma}{\sqrt{N}} E_N$, where the entries of $E_N$ are i.i.d.~complex-valued random variables satisfying some assumptions on its moments, and $\sigma >0$ is a parameter. When $A_N$ is the Jordan matrix, in \cite[Corollary 1.10]{BC} it is shown that for any $\sigma>0$ the random point process induced by the outlier eigenvalues of $\model_N^\sigma$ converges to the zero set of a Gaussian analytic function with some covariance kernel $\mathcal{K}_\sigma(\cdot, \cdot)$. They also noted that, as $\sigma \to 0$, the kernel $\mathcal{K}_\sigma(\cdot, \cdot)$ admits a non-trivial limit and the limiting kernel turns out to be the covariance kernel of the hyperbolic Gaussian analytic function given by \eqref{eq:hgaf}. 
It is striking to see that for the complex Gaussian perturbation of the Jordan matrix the same limit appears in these two rather different frameworks:~in \cite{BC} $N \to \infty$ is followed by $\sigma \downarrow 0$, whereas in this paper $\sigma^{-1}$ and $N$ are sent to infinity together with $\sigma = N^{-\updelta}$ for some $\updelta >0$. However, it should also be noted that, unlike \cite{BC}, here the limit is non-universal. Based on this observation, we predict that the same phenomenon should continue to hold for general finitely banded Toeplitz matrices.


Next, we discuss {possible}
extensions of our results. 
A first obvious direction is to consider 
in Theorem \ref{thm:outlier-law} the case of $E_N$ without the density assumption. {Many steps of the proof go through, except for anti-concentration results of the type discussed
in Section \ref{sec:anti-concentration}. As will be explained in Section \ref{sec:prelim} below, in Section \ref{sec:anti-concentration} we derive anti-concentration bounds for linear combinations of determinants of sub-matrices of $E_N$. To obtain such a bound we use that there is at least one term in the linear sum with a large coefficient.} 

We conjecture that it should be possible to dispense of any density assumption on the entries of the noise matrix and the conclusion of Theorem 
\ref{thm:outlier-law} should continue to hold under minimal assumptions on the entries $E_N$, e.g.~i.i.d.~with zero mean and unit variance. At the level of convergence of empirical measures,
this has been verified, first in \cite{W} and then in \cite{BPZ-non-triang}. The non-universality of the limit process for outliers, see Remark \ref{rem-nonuniv},
complicates however the task of proving this.

The next section outlines the proofs of Theorems \ref{thm:no-outlier} and \ref{thm:outlier-law}.


%
%

\subsection{Outline of the proof}\label{sec:prelim}

We remind the reader that the bulk of the eigenvalues of $T_N({\bm a}) +\Delta_N$ approach the curve ${\bm a}(\mathbb{S}^1)$, as $N \to \infty$. Thus, to study the outlier eigenvalues we need to analyze the set $\{z \in \cup_{\gd=-d_2}^{d_1}\cS_\gd: \det (T_N({\bm a}(z)) +\Delta_N)=0\}$, where for brevity, hereafter we denote ${\bm a}(z)(\cdot):={\bm a}(\cdot) - z$ and recall the definition of $\cS_\gd$ from Definition \ref{dfn:cS-gd}. 

To this end, a key observation is that for $z \in \cS_\gd$ the dominant term in the expansion of $\det(T_N({\bm a}(z))+\Delta_N)$ is $P_{|\gd|}(z)$, where for $k \in [N]$, $P_k(z)$ is the homogeneous polynomial of degree $k$ in the entries of the noise matrix $\Delta_N$ in the expansion of the determinant (see \eqref{eq:P_k-z} for a precise formulation, and \eqref{eq;spec-rad-new} for 
a decomposition of the determinant in terms of these polynomials). It suggests that, the roots of $\det(T_N({\bm a}(z))+\Delta_N)=0$ that are in $\cS_\gd$ should be close to those of $P_{|\gd|}(z)=0$. This, in turn, indicates that the limit of the random point process induced by the roots of $\det(T_N({\bm a}(z))+\Delta_N)=0$ that are in $\cS_\gd$ should be the same for the equation $P_{|\gd|}(z)=0$. {The proof} then boils down to identifying the limit induced by the roots of $P_{|\gd|}(z)=0$ that are in $\cS_\gd$. The goal of this paper is to make these heuristics precise, leading to
 the conclusions of Theorems \ref{thm:no-outlier} and \ref{thm:outlier-law}.  
 
The heuristics described above can be mathematically formulated as below. We fix $\vep >0$, and consider the region $\cS_\gd^{-\vep}$.
From \cite[Lemma A.1]{BPZ-non-triang} it follows that the determinant of $T_N({\bm a}(z))+\Delta_N$ can be written as a sum of $P_k(z)$, where $k$ runs from zero to $N$. 
 From \cite[Lemma A.3]{BPZ-non-triang}, after some preprocessing, it follows that $P_{|\gd|}(z)$ is a polynomial in $\{\lambda_i(z)\}_{i=1}^d$ such that it is of degree $N$ in each variable (see \eqref{eq:P-k-decompose} below), where we remind the reader that $\{-\lambda_i(z)\}_{i=1}^d$ are the roots of the polynomial $P_{z, {\bm a}}(\lambda)=0$, see Definition \ref{dfn:cS-gd}. {Since for} $z \in \cS_\gd$ we have $|\lambda_1(z)|  \ge \cdots \ge |\lambda_{d_1-\gd}(z)| > 1 > |\lambda_{d_1-\gd+1}(z)| \ge \cdots \ge |\lambda_d(z)|$, it is natural to believe that for large $L \in \N$ the roots of $P_{|\gd|}(z)=0$ and of $P_{|\gd|}^L(z)=0$ should be close to each other, where $P_{|\gd|}^L(z)$ is obtained from $P_{|\gd|}(z)$ by removing terms having exponents of $\{\lambda_i(z)\}_{i=1}^{d_1-\gd}$ and $\{\lambda_i(z)\}_{i=d_1-\gd +1}^{d}$ that are less than $N - O(L)$, and greater than $O(L)$, respectively. 
 Indeed, we show in Section  
 \ref{sec:identify-non-dom-term} that the errors made by replacing
 the determinant of  $T_N({\bm a}(z))+\Delta_N$ (which is an analytic function)
 by $P_{|\gd|}^L(z)$ are small (in the sense that the supremum
 of the modulus of the difference over $\cS_\gd^{-\vep}$, properly
 normalized, has small second moment when $N\to\infty$ followed by 
 $L\to\infty$). We also prove that $z\mapsto P_{|\gd|}^L(z)$ are analytic
 functions in $\cS_{\gd}^{-\vep}$, see
 Lemma \ref{lem:analyticity}.

 The advantage of working with (the normalized form of) $P_{|\gd|}^L(z)$
 is that, for fixed $L$, it has a law independent of $N$. This fact is a 
 consequence of a combinatorial analysis, presented in Section
 \ref{sec:limit-tight}.  The upshot is that  $P_{|\gd|}^L(z)$, properly
 normalized,
 can be replaced by the $N$-independent analytic 
 fields $\gP_\gd^L$, and these fields 
 in turn possess a well defined analytic 
 limit $\gP_\gd^\infty$ as $L\to\infty$.

 In order to pass from convergence of the fields to convergence of the process of
 zeros, we employ a general criteria formulated in
 \cite{Sh}. Namely, we 
 need to check that the limit field $\gP_\gd^\infty$ is non degenerate,
 i.e.~does not vanish identically. Since $\gP_\gd^\infty$ was obtained
 as a limit, it suffices to check that the pre-limit possesses good
 enough anti-concentration properties at a fixed 
 point $z\in \cS_\gd^{-\vep}$. The pre-limit for which we 
 prove anti-concentration is $\widehat{P}_\gd(z)$, see Corollary
 \ref{cor:anti-conc-complex-dom}; the latter  builds on an anti-concentration
 result for polynomials in independent variables with maximal degree one in each variable, see Proposition \ref{prop:anti-conc-complex}.

The proof of Theorem \ref{thm:no-outlier} follows a simpler 
line of argument. Indeed, we show that now the the term with
$\gd=0$ is dominant, now for all $z \in \cS_0^{-\vep}$, 
on a set of probability $1-o(1)$. Thus the task reduces 
{to}
showing that $P_0(z)$ does not have any root in $\cS_0^{-\vep}$. Turning to do the same, using an operator norm bound on the noise matrix an $N$-dependent region $\mathcal{D}_N$ can be identified to not have any eigenvalue of $T_N({\bm a})+\Delta_N$, with high probability. Hence, one only needs to find a uniform lower bound on {the modulus of}
$P_0(z)= \det T_N({\bm a}(z))$ for $z \in \cS_0^{-\vep}\setminus \mathcal{D}_N$.

Evaluating the determinant of a finitely banded Toeplitz matrix has a long and impressive history. If the roots of $P_{z,{\bm a}}(\cdot)=0$ are distinct then the determinant of $T_N({\bm a}(z))$ is given by Widom's formula (see \cite[Theorem 2.8]{bottcher-finite-band} and \cite{Bax-Sch}), whereas in the case of double roots there is an analogous result, known as Trench's 
formula, see \cite[Theorem 2.10]{bottcher-finite-band} and \cite{trench} for
a proof. Recently, Bump and Diaconis \cite{BD} noted that, irrespective of whether $P_{z,{\bm a}}(\cdot)=0$ has double roots or not, the determinant of a finitely banded Toeplitz matrix can be expressed as a ratio of certain Schur polynomials in the roots of $P_{z, {\bm a}}(\lambda)=0$. Since we are interested in finding a uniform lower bound on \corAB{the modulus of} the determinant we work with the formulation of Bump and Diaconis, from which the desired uniform lower bound follows. This finishes the outline of the proof of Theorem \ref{thm:no-outlier}.

\subsection*{Outline of the rest of the paper} In Section \ref{sec:identify-non-dom-term}, using the second moment method, we find upper bounds on the non-dominant terms. Section
 \ref{sec:limit-tight} presents the combinatorial
 analysis leading
 to controls of the fields $\gP_\gd^L$ and $\gP_\gd^\infty$.
 Section \ref{sec:anti-concentration} is devoted to 
 deriving the 
 general anti-concentration bounds alluded to
 above, which is then applied to yield the same for the dominant term. 
 Section \ref{sec:dom-term} is devoted to the proof of
  Theorem \ref{thm:outlier-law}, while
  Section \ref{sec:proof-main-1} is devoted to the proof of
Theorem \ref{thm:no-outlier}.
Finally, as mentioned in Remark \ref{rmk:spec-rad}, 
extending the ideas of proof of Theorem \ref{thm:no-outlier}, in Appendix \ref{sec:spec-rad} we prove that the spectral radii of $\{N^{-1/2}E_N\}_{N \in \N}$ are tight. 

\subsection*{Acknowledgements} Research of AB is partially supported by a grant from Infosys Foundation, an Infosys--ICTS Excellence grant, and a Start-up Research Grant (SRG/2019/001376) and MATRICS grant (MTR/2019/001105) from Science and Engineering Research Board of Govt.~of India. OZ is partially supported by Israel Science Foundation grant 147/15 and funding from the European Research Council (ERC) under the European Unions Horizon 2020 research and innovation program (grant agreement number 692452). We thank Mireille Capitaine for her interest and for discussing \cite{BC} with us, and thank Martin Vogel for  Remark \ref{rmk:resolvent} and other useful comments. We are grateful to the 
anonymous referee for her/his suggestions that led to a 
shortening of our original proof of Theorem \ref{thm:outlier-law},
and 
also to a weakening of its hypotheses. We also thank the referee for
several other useful comments.

\section{Identification of
dominant term and tightness of the scaled determinants}\label{sec:identify-non-dom-term}
 In this section we show that the determinant of $T_N({\bm a}(z)) +\Delta_N $, when scaled appropriately, is uniformly tight, where we remind the reader that ${\bm a}(z)(\cdot):= {\bm a}(\cdot) -z$ and $\Delta_N = N^{-\gamma} E_N$. This will be shown by deriving uniform upper bounds on the non-dominant terms of the scaled determinant, as well as the same for the dominant term. For later use, we will also derive bounds on the second moment of the tail of the dominant term.

 Before proceeding further we
  introduce relevant definitions. For $k \in [N]$ set
 \begin{equation}\label{eq:P_k-z}
P_k(z):= \sum_{\substack{X, Y \subset [N]\\ |X|=|Y|=k}} (-1)^{\sgn(\sigma_{X}) \sgn(\sigma_{Y})} \det (T_N({\bm a}(z))[{X}^c; {Y}^c]) N^{-\gamma k} \det (E_N[X; Y]),
\end{equation}
where we recall that
 $E_N[X; Y]$ denotes the sub-matrix of $E_N$ induced by the rows in $X$ and columns in $Y$, ${X}^c:=[N]\setminus X$, ${Y}^c:= [N]\setminus Y$, and for $Z \in \{X, Y\}$, $\sigma_Z$ is the permutation on $[N]$ which places all the elements of $Z$ before all the elements of ${Z}^c$, but preserves the order of the elements within the two sets. Additionally denote $P_0(z):= \det (T_N({\bm a}(z)))$. 
 
 From \cite[Lemma A.1]{BPZ-non-triang} it follows that 
\begin{equation}\label{eq;spec-rad-new}
\det(T_N({\bm a}(z))+N^{-\gamma} E_N) = \sum_{k=0}^N P_k(z).
\end{equation}
Below we will {show  that for $z \in \cS_\gd$,} the dominant term in the expansion \eqref{eq;spec-rad-new} is $P_{|\gd|}(z)$. In Section \ref{sec:anti-concentration},
 it will be further argued that  {for $z \in \cS_\gd^{-\vep}$,} $|P_{|\gd|}(z)|$ is of the order {$\mathfrak{K}(z)$, where
  \begin{equation}
  \label{fancyK}
 \mathfrak{K}(z):=  \mathfrak{K}(z, \gd):= a_{d_1}^N \cdot N^{-\gamma |\gd|} \cdot \prod_{i=1}^{d_1-\gd} \lambda_i(z)^{N+d_2}. 
 \end{equation}
 (By convention, we set an empty product to equal $1$.) Thus, for uniform tightness, we scale the determinants by  this factor, setting}
\begin{equation}\label{eq:hatP-k}
\widehat P_k(z):= P_k(z)/\mathfrak{K}(z), \qquad k=1,2,\ldots, N,
\end{equation}
and
\begin{equation}\label{eq:hat-det}
\widehat \det_N(z):= \det(T_N({\bm a}(z))+N^{-\gamma} E_N)/\mathfrak{K}(z).
\end{equation} 

We will show that, for any compact set $K \subset \cS_\gd^{-\vep}$, the sequence of random variables $\{\|\widehat\det_N (z)\|_K\}_{N \in \N}$ is uniformly tight. It follows from \cite[Lemma 2.6]{Sh} that if $\sup_N \E [|\widehat \det_N(z)|^p]$ is locally integrable for some $p >0$, then the sequence $\{\|\widehat\det_N (z)\|_K\}_{N \in \N}$ is indeed uniformly tight. Therefore it suffices to derive the local integrability of $\sup_N \E [|\widehat \det_N(z)|^p]$. In this section we derive this local integrability with $p=2$. The following are the main results of this section. 

The first two results, {whose proofs are postponed,} derive a bound on the second moments of the supremum (in $z$) of the non-dominant terms in the expansion of $\widehat \det_N(z)$. 

 \begin{lemma}\label{lem:rouche-multi-gr-d_0}
Fix $\vep >0$, $\gamma > \frac12$, and an integer $\gd$ such that $ -d_2 \le \gd \le d_1$. Let $T_N({\bm a})$ be an $N \times N$ Toeplitz matrix with symbol ${\bm a}$, where ${\bm a}$ as in \eqref{eq:laurent-poly}. Assume that the entries of $E_N$ are independent with zero mean and unit variance. Then, there exists $\eta_0 >0$ such that 
\begin{equation}\label{eq:gr-d-0}
\sup_{k = |\gd|+1}^N  \E \left[ \sup_{z \in \cS_\gd^{-\vep}}  |\widehat P_k(z)|^2 \right] \cdot N^{\eta_0(k- |\gd|)}  \le \sup_{k = |\gd|+1}^N  \E \left[ \sup_{z \in \cS_\gd^{-\vep}}  |\widehat P_k(z)|^2 \cdot \prod_{i=1}^{d_1-\gd} |\lambda_i(z)|^{2(k+d_2)} \right] \cdot N^{\eta_0(k- |\gd|)} \le  1,
 \end{equation}
 for all large $N$. 
\end{lemma}

\begin{lemma}\label{lem:rouche-multi-le-d_0}
Under the same set-up as in Lemma \ref{lem:rouche-multi-gr-d_0}, there exists an $\vep_\star \in (0,1)$, depending only on {${\bm a}$ and $\vep>0$}, such that
\[
\sup_{k=0}^{|\gd|-1}\E\left[\sup_{z \in \cS_\gd^{-\vep}} |\widehat P_k(z)|^2\right] \le (1-\vep_\star)^{N},
 \]
 for all large $N$. 
\end{lemma}

Building on Lemmas \ref{lem:rouche-multi-gr-d_0}-\ref{lem:rouche-multi-le-d_0} we derive the following bound on the second moment of the non-dominant terms. 

\begin{corollary}\label{cor:non-dominant}
Under the same set-up as in Lemma \ref{lem:rouche-multi-gr-d_0}, for large $N$ it holds that
\[
  \E \left[ \sup_{z \in \cS_\gd^{-\vep}} \left|\widehat \det_N(z) -\widehat P_{|\gd|}(z)\right|^2 \right] \le N^{-\frac{\eta_0}{2}}.
\]
\end{corollary}

The next lemma, {whose proof is deferred, gives}  an upper bound on the second moment of the dominant term in the expansion of $\widehat \det_N(z)$. 

\begin{lemma}\label{lem:rouche-multi-eq-d_0}
Under the same set-up as in Lemma \ref{lem:rouche-multi-gr-d_0}, there exists a constant $C_0$ such that
\[
\sup_N  \E \left[\sup_{z \in \cS_\gd^{-\vep}}  |\widehat P_{|\gd|}(z)|^2 \right] \le C_0. 
\]
\end{lemma}

The key to the proof of the above  results is a
representation of
$P_k(z)$ as linear combinations of products 
of determinants of certain bidiagonal matrices with coefficients that are determinants of sub-matrices of $E_N$. Toward this end, we borrow ideas from \cite{BPZ-non-triang}. 

If $T_N({\bm a}(z))$ is an upper triangular matrix then it is obvious that 
\begin{equation}
T_N({\bm a}(z)) =a_d \cdot  \prod_{i=1}^{d} (J_N+\lambda_i (z)\Id_N), \notag
\end{equation}
where $J_N$ is the nilpotent matrix given by $(J_N)_{i,j} = {\bf 1}_{j=i+1}$ for $i, j \in [N]$. Then the desired representation is simply a consequence of Cauchy-Binet theorem. For a general Toeplitz matrix the above product representation does not hold. It was noted in \cite{BPZ-non-triang} that $T_N({\bm a(z)})$ can be viewed as a certain sub-matrix of an upper triangular finitely banded Toeplitz matrix with a slightly larger dimension. {(This is related to the Grushin problem discussed by  Sj\"{o}strand and Vogel, see e.g. \cite{SV3}, in that
  one replaces the study in dimension $N$ with a slightly larger dimension.
However, the details of the replacement, as well as the goals, are different.)}
Therefore one can essentially repeat the same product representation and apply the Cauchy-Binet theorem. 

To use efficiently this idea,  we introduce the following definition.

\begin{dfn}[Toeplitz with a shifted symbol]\label{dfn:toep-shifted}
Let $T_N({\bm a})$ be a Toeplitz matrix with finite symbol ${\bm a}(\lambda)=\sum_{\ell=-d_2}^{d_1} a_\ell \lambda^\ell$. For $M>d$, $z \in \C$  and
$\bar d_1, \bar d_2 \in \mathbb{N}$ such that $\bar d_1 + \bar d_2=d_1+d_2=d$,
let
$T_M({\bm a}, z; \bar d_1)$ denote the $M \times M$ Toeplitz matrix with the first row and column 
\[
(a'_{d_1 -\bar d_1}, a'_{d_1-\bar d_1+1}, \ldots, a'_{d_1}, 0, \ldots, 0) \qquad \text{ and } \qquad (a'_{d_1 - \bar d_1}, a'_{d_1 - \bar d_1-1}, \ldots, a'_{-d_2}, 0, \ldots, 0)^{\sf T},\] 
respectively, where $a'_j := a_j - z \delta_{j,0}$, $j=-d_2,-d_2+1,\ldots, d_1$. 
\end{dfn}

From Definition \ref{dfn:toep-shifted},
it follows that
\[
T_N({\bm a}(z))=T_N({\bm a}, z;d_1) = T_{N+d_2}( {\bm a}, z; d) [[N]; [N+d_2]\setminus [d_2]],
\]
since
\[
T_{N+d_2}({\bm a},z;d_1)=\begin{bmatrix}
a_{-d_2} &\cdots &a_0-z& \cdots & a_{d_1}&0&\cdots& 0\\
0& a_{-d_2} & & a_0-z & &\ddots& & \vdots\\
\vdots & \ddots &\ddots &&\ddots&&\ddots & \vdots\\
\vdots & & \ddots & \ddots &&\ddots& & a_{d_1}\\
\vdots & &  & \ddots & \ddots & &\ddots & \vdots\\
\vdots && &  &\ddots& \ddots & & a_0 -z\\
\vdots & & & &&  \ddots & \ddots & \vdots \\
 0 & \cdots & \cdots & \cdots&\cdots &\cdots & 0 & a_{-d_2}
\end{bmatrix}.
\]
Note that $T_{N+d_2}({\bm a}, z;d)$ is an upper triangular Toeplitz matrix. As $\{-\lambda_\ell(z)\}_{\ell=1}^d$ are the roots of the equation $P_{z, {\bm a} }(\lambda)=0$ we obtain that
\[
T_{N+d_2}({\bm a}, z;d) = \sum_{\ell=0}^d (a_{\ell-d_2} - z \delta_{\ell, d_2}) J_{N+d_2}^\ell = a_{d_1} \prod_{\ell=1}^d (J_{N+d_2}+ \lambda_\ell(z) \Id_{N+d_2}).
\]
 Hence, recalling the definition of $\{P_k(z)\}_{k=1}^N$ from  \eqref{eq:P_k-z}, applying the
 Cauchy-Binet theorem, and writing $S+ \ell:= \{x+\ell, x \in S\}$ for any set of integers $S$ and an integer $\ell$, we obtain that 
\begin{eqnarray}\label{eq:det-decompose-1}
P_k(z) 
&= &\sum_{\substack{X, Y \subset [N]\\ |X|=|Y|=k}} (-1)^{\sgn(\sigma_{X}) \sgn(\sigma_{Y})} \det (T_{N+d_2}({\bm a}, z; d)[{X}^c; {Y}^c+d_2]) \cdot N^{-\gamma k} \cdot \det (E_N[X; Y])
\nonumber\\
& =&  \sum_{\substack{X, Y \subset [N]\\|X|=|Y|=k}}\sum_{i=2}^{d-1} \sum_{\substack{X_i \subset [N+d_2]\\ |X_i|=k+d_2}} (-1)^{\sgn(\sigma_{X}) \sgn(\sigma_{Y})}a_{d_1}^{N-k} \cdot \prod_{i=1}^d \det\left((J_{N+d_2} + \lambda_i(z)\Id_{N+d_2})[\COMP{X}_i; \COMP{X}_{i+1}]\right)\nonumber\\
  &&\qquad\cdot N^{-\gamma k} \cdot \det (E_N[X; Y]),
\end{eqnarray}
where 
\begin{equation}\label{eq:X-Y-constraint}
X_1:= X_1(X):=X \cup [N+d_2]\setminus [N], \qquad X_{d+1}:=X_{d+1}(Y):= (Y+d_2) \cup [d_2],
\end{equation}
and $\COMP{Z}:=[N+d_2]\setminus Z$ for any set $Z \subset [N+d_2]$. We emphasize the notational difference between $\COMP{Z}$ and $Z^c$. The former will be used to write the complement of $Z$ when viewed as a subset of $[N+d_2]$, where for the latter $Z$ will be viewed as a subset of $[N]$. 

The \abbr{RHS} of \eqref{eq:det-decompose-1} gives the desired representation of $P_k(z)$. To prove Lemmas \ref{lem:rouche-multi-gr-d_0}-\ref{lem:rouche-multi-le-d_0} we require some preprocessing of the \abbr{RHS} of \eqref{eq:det-decompose-1}. To obtain a tractable expression we express the sums in \eqref{eq:det-decompose-1} over $\{X_i\}_{i=1}^{d+1}$ as an iterated sums, see \eqref{eq:P-k-decompose} below. The inner sum will be over the choices of $\{X_i\}_{i=1}^{d+1}$ such that the product of the determinants of the bi-diagonal matrices is constant and the outer sum will be over all possible values of the product of the determinants.  

We now describe this decomposition. From \eqref{eq:det-decompose-1}-\eqref{eq:X-Y-constraint} we have that $|X_i| =k+d_2$, for $i \in [d+1]$. Therefore, we write 
\begin{equation}
\label{eq-Lnew}
X_i := \left\{ x_{i,1} < x_{i,2} < \cdots < x_{i,k+d_2} \right\}, \quad \cX_k:=(X_1, X_2, \ldots, X_{d+1}).
\end{equation}
 Using  \cite[Lemma A.3]{BPZ-non-triang} we note that
\begin{eqnarray}\label{eq:det-bidiagonal}
&&\det( (J_{N+d_2} + \lambda_i(z) \Id_{N+d_2}) [\COMP{X}_i; \COMP{X}_{i+1}]) =\\
&& \lambda_i(z)^{x_{i+1,1}-1} \cdot \left(\prod_{\ell=2}^{k+d_2} \lambda_i(z)^{x_{i+1,\ell}-x_{i,\ell}-1} \right) \cdot \lambda_i(z)^{N+d_2 - x_{i,k+d_2}} 
\cdot {\bf 1} \left\{ x_{i+1,\ell} \le x_{i,\ell} < x_{i+1, \ell+1}, \ell \in [k+d_2]\right\},
\nonumber
\end{eqnarray}
where we have set $x_{i+1,k+d_2+1}=\infty$ for convenience. In light of \eqref{eq:det-bidiagonal}, for any ${\bm \ell}:=(\ell_1,\ell_2,\ldots,\ell_d)$ with $0 \le \ell_i \le N+d_2$ for 
$i \in [d]$, and $k\in[N]$, we define
\begin{multline*}
L_{{\bm \ell},k}:= \{ \cX_k:  \, 1 \le x_{i+1,1} \le x_{i,1} < x_{i+1,2} \le x_{i,2} < \cdots < x_{i+1,k+d_2} \le x_{i,k+d_2} \le N+d_2,\\
\text{ and }  \, x_{i+1,1}+ \sum_{j=2}^{k+d_2}(x_{i+1,j}-x_{i,j-1})+ (N+d_2-x_{i,k+d_2}) =\ell_i+k+d_2, \text{ for all } i=1,2,\ldots,d\}.
\end{multline*}
Note that \eqref{eq:det-bidiagonal} implies that the summand in \eqref{eq:det-decompose-1} is non-zero only when $\cX_k \in L_{{\bm \ell}, k}$ for some ${\bm \ell}$ and in that case
\begin{equation}\label{eq:prod-decomp}
\prod_{i=1}^d \det\left((J_{N+d_2} + \lambda_i(z)\Id_{N+d_2})[\COMP{X}_i; \COMP{X}_{i+1}]\right) = \prod_{i=1}^d \lambda_i(z)^{\ell_i}.
\end{equation}
Recall that in Lemmas \ref{lem:rouche-multi-gr-d_0} and \ref{lem:rouche-multi-le-d_0} we aim to show that for $z \in \cS_\gd^{-\vep}$ and $k \ne |\gd|$, $|P_k(z)|$ is small compared 
to 
 $\mathfrak{K}(z)$ of \eqref{fancyK}.
 Thus it would be convenient to pull out this factor from the \abbr{RHS} of \eqref{eq:prod-decomp}. So, using the observation that 
\[
x_{i+1,1}+  \sum_{j=2}^{k+d_2}(x_{i+1,j}-x_{i,j-1}) +  \sum_{j=1}^{k+d_2}(x_{i,j}-x_{i+1,j}) + (N+d_2- x_{i,k+d_2}) = N+d_2,
 \]
we have the following equivalent representation of $L_{{\bm \ell}, k}$:
  \begin{multline}\label{eq:L-ell-k}
L_{{\bm \ell},k}:= \{ \cX_k:  \, 1 \le x_{i+1,1} \le x_{i,1} < x_{i+1,2} \le x_{i,2} < \cdots < x_{i+1,k+d_2} \le x_{i,k+d_2} \le N+d_2,\\
 \, \sum_{j=1}^{k+d_2}(x_{i,j}-x_{i+1,j}+1)   =\hat \ell_i; \ i=1, 2,\ldots,d_0,\\
\text{ and }  \, x_{i+1,1}+ \sum_{j=2}^{k+d_2}(x_{i+1,j}-x_{i,j-1})+ (N+d_2-x_{i,k+d_2}) =\hat \ell_i+k+d_2; \ i=d_0+1,d_0+2,\ldots,d\},
\end{multline}
where
\begin{equation}\label{eq:hat-ell}
\hat \ell_i := \left\{\begin{array}{ll} \ell_i & \mbox{ if } i > d_0,\\
N+d_2- \ell_i & \mbox{ if } i \le d_0.
\end{array}
\right.
\end{equation}
and $d_0=d_1-\gd$. Since $x_{i+1,j} \le x_{i,j}$ for any $i=1,2,\ldots, d$, and $j=1, 2, \ldots, k+d_2$, we note from \eqref{eq:L-ell-k} that
\begin{equation}\label{eq:hat-ell-lbd}
\hat \ell_i \ge k+d_2, \qquad \text{ for } i=1,2,\ldots, d_0. 
\end{equation}
This will be used below
in the proof of Lemma \ref{lem:rouche-multi-gr-d_0}. Equipped with the above notation we find that, for any $\cX_k \in L_{{\bm \ell}, k}$, 
\begin{equation}\label{eq:det-prod-decompose}
\prod_{i=1}^d \det\left((J_{N+d_2} + \lambda_i(z)\Id_{N+d_2})[\COMP{X}_i; \COMP{X}_{i+1}]\right) = \prod_{i=1}^{d_0} \lambda_i(z)^{N+d_2} \cdot \prod_{i=1}^{d_0} \lambda_i(z)^{-\hat \ell_i} \prod_{i=d_0+1}^d \cdot \lambda_i(z)^{\hat \ell_i}.
\end{equation}
Furthermore, the restriction \eqref{eq:X-Y-constraint} and the fact that the outer sum in \eqref{eq:det-decompose-1} is over $X, Y \subset [N]$ implies that the summand in \eqref{eq:det-decompose-1} vanishes unless $\cX_k \in \gL_{{\bm \ell},k}$, where
\begin{equation}\label{eq:gL-ell-k}
\gL_{{\bm \ell}, k} := \{\cX_k \in L_{{\bm \ell}, k}: x_{1,k+j}=N+j; \ j  \in [d_2] \quad \text{ and } \quad x_{d+1,j}=j; \ j \in [d_2]\}.
\end{equation}
Therefore, from \eqref{eq:det-decompose-1} and \eqref{eq:det-prod-decompose} we deduce that 
\begin{multline}\label{eq:P-k-decompose}
P_k(z)= a_{d_1}^{N-k} \cdot \prod_{i=1}^{d_0} \lambda_i(z)^{N+d_2}\cdot N^{-\gamma k} \sum_{{\bm \ell}} \prod_{i=1}^{d_0} \lambda_i(z)^{-\hat \ell_i} \cdot \prod_{i=d_0+1}^{d} \lambda_i(z)^{\hat \ell_i} \\
\cdot \sum_{\cX_k \in \gL_{{\bm \ell},k}} (-1)^{\sgn(\sigma_{\mathbb{X}}) \sgn(\sigma_{\mathbb{Y}})}\det(E_N[\mathbb{X}; \mathbb{Y}]),
\end{multline}
with
\begin{equation}\label{eq:mathbb-X}
 \mathbb{X}:=\mathbb{X}(X_1):= X_1 \cap [N] \qquad \text{ and } \qquad \mathbb{Y}:= \mathbb{Y}(X_{d+1}) := (X_{d+1} - d_2) \cap [N],
\end{equation}
where \eqref{eq:mathbb-X} is a consequence of \eqref{eq:X-Y-constraint}. 
We introduce {further}  notation. Set
\begin{equation}\label{eq:Q-ell-k}
Q_{{\bm \ell}, k}(z) := \prod_{i=1}^{d_0} \lambda_i(z)^{-\hat \ell_i}  \prod_{i=d_0+1}^{d} \lambda_i(z)^{\hat \ell_i} \cdot \widehat Q_{{\bm \ell},k},
\end{equation}
where
\begin{equation}\label{eq:hat-Q-ell-k}
\widehat Q_{{\bm \ell}, k}:=\sum_{ \cX_k \in \gL_{{\bm \ell},k}}  (-1)^{\sgn(\sigma_{\mathbb{X}}) \sgn(\sigma_{\mathbb{Y}})}\det(E_N[\mathbb{X}; \mathbb{Y}])
\end{equation}
does not depend on $z$. Recalling the definition of $\widehat P_k(z)$ (see \eqref{eq:hatP-k}), we have from \eqref{eq:P-k-decompose}  that
\begin{equation}\label{eq:P-k-decompose-0}
\widehat P_k(z)= \corAB{a_{d_1}^{-k}}N^{-\gamma (k- |\gd|)} \sum_{{\bm \ell}}  Q_{{\bm \ell}, k}(z) = \corAB{a_{d_1}^{-k}}N^{-\gamma (k-|\gd|)} \sum_{{\bm \ell}} \prod_{i=1}^{d_0} \lambda_i(z)^{-\hat \ell_i}  \prod_{i=d_0+1}^{d} \lambda_i(z)^{\hat \ell_i} \cdot \widehat Q_{{\bm \ell}, k}
\end{equation}
Having obtained a tractable expression in \eqref{eq:P-k-decompose-0} we now proceed to apply the second moment method to prove Lemmas \ref{lem:rouche-multi-gr-d_0} and \ref{lem:rouche-multi-le-d_0}. So, we next estimate  the variance of $\widehat Q_{{\bm \ell}, k}$. Using the facts that the entries of $E_N$ are independent with zero mean and unit variance it is straightforward to see that
\[
\E \left[\det(E_N[\mathbb{X}_*; \mathbb{Y}_*] \cdot \overline{\det(E_N[\mathbb{X}'; \mathbb{Y}'] )}\right] = \left\{\begin{array}{ll}
 k! & \mbox{if } \mathbb{X}_*= \mathbb{X}' \text{ and } \mathbb{Y}_*=\mathbb{Y}'\\
0 & \mbox{ otherwise},
\end{array}
\right.
\]
for any collection of subsets $\mathbb{X}_*, \mathbb{Y}_*, \mathbb{X}', \mathbb{Y}' \subset [N]$, each of cardinality $k$. Hence, we deduce that 
\begin{align}\label{eq:var-bound-multi}
\E[|\widehat Q_{{\bm \ell}, k}|^2]=\Var (\widehat Q_{{\bm \ell}, k})  & \, = k! \cdot \mathfrak{N}_{{\bm \ell}, k},
\end{align}
where
\begin{equation}\label{eq:N-ell-k}
\mathfrak{N}_{{\bm \ell}, k}:=\left| \left\{\cX_k =(X_1, X_2,\ldots, X_{d+1}), \cX_k'=(X_1', X_2', \ldots, X_{d+1}') \in \gL_{{\bm \ell}, k}: X_1=X_1', X_{d+1}=X_{d+1}' \right\}\right|.
\end{equation}
Thus an estimate on the variance of $\widehat Q_{{\bm \ell}, k}$ requires a bound on  $\mathfrak{N}_{{\bm \ell}, k}$. This is done in the lemma below. The proof is postponed to later in the section.
\begin{lemma}\label{lem:combinatorial-1}
Fix $k \in \N$, an integer $\gd$ such that $-d_2 \le \gd \le d_1$, and ${\bm \ell}=(\ell_1,\ell_2,\ldots,\ell_d)$ such that $0 \le \ell_i \le N+d_2$ for all $i \in [d]$.
For  $|\gd| \le k \le N$ and  ${\mathfrak N}_{{\bm \ell}, k}$  as in \eqref{eq:N-ell-k}, we have
\begin{align}\label{eq:N-l-k-bd}
\mathfrak{N}_{{\bm \ell}, k} & \, \le    \binom{N+d_2}{k-|\gd|} \cdot \prod_{i=1}^{d_0} \binom{\hat \ell_i-1}{k+d_2-1}^2 \cdot \prod_{i=d_0+1}^{d} \binom{\hat \ell_i+k+d_2}{k+d_2}^2.
\end{align}
\end{lemma}
 One final ingredient needed for the proof
of  {Lemmas \ref{lem:rouche-multi-gr-d_0}, \ref{lem:rouche-multi-le-d_0}, and \ref{lem:rouche-multi-eq-d_0},} is a uniform separation of the moduli of the roots $\{-\lambda_i(z)\}_{i=1}^d$ from one, for all $z \in \cS_\gd^{-\vep}$. This is formulated in the lemma below. 
\begin{lemma}\label{lem:root-cont-0}
Fix $\vep \in (0,1)$ and an integer $\gd$ such that $-d_2 \le \gd \le d_1$. Denote $d_0:=d_1-\gd$. Then
\begin{equation}\label{eq:lambda-bd}
\sup_{z \in \cS_\gd^{-\vep}} \max\left\{\max_{i=1}^{d_0}\{|\lambda_i(z)|^{-1}\} , \max_{i=d_0+1}^d\{|\lambda_i(z)|\}\right\} \le 1-\vep_0,
\end{equation}
for some sufficiently small $\vep_0>0$, depending only on $\vep$ and the symbol ${\bm a}$.
\end{lemma}
\begin{proof}
Recalling the definition of $\cS_\gd$ (see Definition \ref{dfn:cS-gd}) we have that $\cS_\gd \subset ({\bm a }(\mathbb{S}^1))^c$. This implies that 
\begin{equation}\label{eq:disjoint}
\cS_\gd^{-\vep} \cap ({\bm a}({\mathbb S}^1))^\vep = \emptyset.
\end{equation} 
On the other hand, if \eqref{eq:lambda-bd} is violated for some $z \in \cS_\gd^{-\vep}$ then there exists a root $\lambda_0(z)$ of the equation $P_{z, {\bm a}}(\lambda)=0$ such that 
\[
||\lambda_0(z)| -1 | \le 2 \vep_0,
\]
whenever $\vep_0 < \frac12$. Therefore, denoting
\[
z_0:= \sum_{i=-d_2}^{d_2} a_i \cdot \frac{\lambda_0(z)^i}{|\lambda_0(z)|^i} \in {\bm a}(\mathbb{S}^1).
\]
By the triangle inequality it follows that 
\[
|z -z_0| \le \sum_{i=-d_2}^{d_1} |a_i| \cdot |\lambda_0(z)|^i \cdot \left| 1 - \frac{1}{|\lambda_0(z)|^i}\right| =  \sum_{i=-d_2}^{d_1} |a_i| \cdot \left| |\lambda_0(z)|^i- 1 \right| \le d 2^{d-1} \max_{i=-d_2}^{d_1} |a_i| \cdot ||\lambda_0(z)| -1|.
\]
Now upon choosing $\vep_0$ sufficiently small we note that the above implies that $z \in ({\bm a}({\mathbb S}^1))^\vep$. This yields a contradiction to \eqref{eq:disjoint}, thereby proving the claim \eqref{eq:lambda-bd}. 
\end{proof}

We now proceed to the proof of Lemma \ref{lem:rouche-multi-gr-d_0}.

\begin{proof}[Proof of Lemma \ref{lem:rouche-multi-gr-d_0} {(assuming Lemma \ref{lem:combinatorial-1})}] 
Note that, as $|\lambda_i(z)| \ge 1$ on $\cS_\gd^{-\vep}$ for $i=1,2,\ldots, d_0$, the left most inequality in \eqref{eq:gr-d-0} is immediate. So we only need to prove the right inequality. To this end, fix $k>|\gd|$ and 
introduce the notation
\begin{equation}\label{eq:kappa}
\kappa({\bm \ell}, z) := a_{d_1}^{-k} N^{-\gamma (k- |\gd|)} \cdot \prod_{i=1}^{d_0} \lambda_i(z)^{-\hat \ell_i} \cdot \prod_{i=d_0+1}^d \lambda_i(z)^{\hat \ell_i} \cdot \prod_{i=1}^{d_0} \lambda_i(z)^{k+d_2}
\end{equation}
and
\begin{equation}\label{eq:kappa-tilde}
\widetilde \kappa({\bm \ell}):= \sup_{z \in \cS_\gd^{-\vep}} |\kappa({\bm \ell}, z)|.
\end{equation}
We now have that
\begin{multline}\label{eq:P-k-cs}
\E\left[ \sup_{z \in \cS_\gd^{-\vep}} |\widehat P_{k}(z)|^2 
\cdot \prod_{i=1}^{d_0} \lambda_i(z)^{2(k+d_2)}\right] = \E\left[ \sup_{z \in \cS_\gd^{-\vep}} \left|\sum_{\bm \ell} \kappa({\bm \ell}, z) \cdot \widehat Q_{{\bm \ell}, k}\right|^2\right] \\
\le  \E\left[ \sum_{{\bm \ell}, {\bm \ell}'} \sup_{z \in \cS_\gd^{-\vep}} \left|\kappa({\bm \ell}, z) \cdot \widehat Q_{{\bm \ell}, k} \cdot  \kappa({\bm \ell}', z)\cdot \widehat  Q_{{\bm \ell}', k}\right|\right] \\
\le  \sum_{{\bm \ell}, {\bm \ell}'} \widetilde \kappa({\bm \ell}) \cdot (\E|\widehat Q_{{\bm \ell}, k}|^2)^{1/2} \cdot \widetilde\kappa({\bm \ell}') \cdot (\E|\widehat  Q_{{\bm \ell}', k}|^2)^{1/2} = \left(\sum_{\bm \ell} \widetilde \kappa({\bm \ell}) \cdot (\E|\widehat Q_{{\bm \ell}, k}|^2)^{1/2}\right)^2,
\end{multline}
where the second inequality is a consequence of the Cauchy-Schwarz inequality. 
It follows from Lemmas \ref{lem:combinatorial-1} and \ref{lem:root-cont-0}, and \eqref{eq:var-bound-multi}, that
\begin{align}\label{eq:indv-term}
& \sup_{z \in \cS_\gd^{-\vep}} |\kappa({\bm \ell}, z)| \cdot (\E|\widehat Q_{{\bm \ell}, k}|^2)^{1/2} \\
 \le & \,  |a_{d_1}|^{-k} \cdot N^{-\gamma (k- |\gd|)} \prod_{i=1}^{d} (1-\vep_0)^{\tilde { \ell}_i} \cdot \sqrt{k! \cdot \binom{N+d_2}{k-|\gd|}} \cdot \prod_{i=1}^{d_0} \binom{\hat \ell_i-1}{k+d_2-1} \cdot \prod_{i=d_0+1}^{d} \binom{\hat \ell_i+k+d_2}{k+d_2}\notag\\
\le & \,   |a_{d_1}|^{-k} \cdot e^{d_2} \cdot N^{-(\gamma -1/2)(k- |\gd|)} \prod_{i=1}^{d} (1-\vep_0)^{\tilde { \ell}_i} \cdot k^{|\gd|/2} \cdot \prod_{i=1}^{d_0} \binom{\hat \ell_i-1}{k+d_2-1} \cdot \prod_{i=d_0+1}^{d} \binom{\hat \ell_i+k+d_2}{k+d_2}, \notag
\end{align}
where 
\[
\tilde \ell_i:= \left\{ \begin{array}{ll}
\hat \ell_i - (k+d_2) & \mbox{for } i=1,2,\ldots, d_0,\\
\hat \ell_i & \mbox{ for } i=d_0+1, d_0+2, \ldots, d.
\end{array}
\right.
\]
So, to find a bound on $\E[\sup_{z \in \cS_\gd^{-\vep}} |\widehat P_k(z)|^2\cdot \prod_{i=1}^{d_0} \lambda_i(z)^{2(k+d_2)}]$,
we need to sum the \abbr{RHS} of \eqref{eq:indv-term} over the range of ${\bm \ell}$. To control this sum effectively we consider the cases of $k$ small and large separately.
First we consider the case when $k$ is small. 

\noindent
{\bf Case $1$:} Let $k \le N^{1/\log \log N}$. In this case to evaluate the sum of the \abbr{RHS} of \eqref{eq:indv-term} over ${\bm \ell}$ we split the range of ${\bm \ell}$ into two further sub-cases. Let us consider the case when $\max_i \ell_i$ is small. Set
\[
{\sf R}_1:=\{{\bm \ell}: \hat \ell_i \le N^{\frac{2}{\log \log N}}, i=1,2,\ldots,d\}. 
\]
We find from the above that 
as by \eqref{eq:det-prod-decompose} $\tilde \ell_i$'s are non-negative for all $i=1,2,\ldots,d$, there exists $\eta_0 >0$ such that  
\begin{align}
\left[\sum_{{\bm \ell} \in {\sf R}_1} \widetilde \kappa({\bm \ell}) \cdot (\E|\widehat Q_{{\bm \ell}, k}|^2)^{1/2}\right] & \le \left[ |{\sf R}_1| \cdot \max_{{\bm \ell} \in {\sf R}_1}  \left\{\widetilde \kappa({\bm \ell}) \cdot (\E|\widehat Q_{{\bm \ell}, k}|^2)^{1/2}\right\} \right]\notag\\
 & \le  |a_{d_1}|^{-k} \cdot e^{d_2} \cdot N^{-(\gamma -1/2)(k- |\gd|)} \cdot  k^{|\gd|/2} \cdot N^{\frac{5(k+d_2)d}{\log \log N}} \le N^{-\eta_0(k -|\gd|)}, \label{eq:kell-ss}
\end{align}
for all large $N$. 

Thus, it remains to consider the case when $\max_i \ell_i$ is large. For any such ${\bm \ell}$ we divide the range as follows: we define
\begin{align*}
\wc{\sf R}_{1,i}&:= \Big\{(k,{\bm \ell}): \hat \ell_{(j)} \le N^{\frac{2}{\log \log N}}, j \in [i-1];  \quad  \text{ and }\quad  N^{\frac{2}{\log \log N}}\le \hat \ell_{(j)} \le N+d_2, j \in [d]\setminus [i-1]\Big\},
\end{align*}
for any $i \in [d]$, where $\hat \ell_{(1)} \le \hat \ell_{(2)} \le \cdots \le \hat \ell_{(d)}$ is a non-increasing rearrangement of $\{\hat \ell_i\}_{i=1}^d$. 

Now, if ${\bm \ell} \in \wc{\sf R}_{1,i}$, and as $k \le N^{1/\log \log N}$, we note that $\hat \ell_{(j)} \ge k^2 \vee N^{\frac{2}{\log \log N}}$, for any $j=i,i+1,\ldots,d$. Therefore, for any such $j$ we have that
\[
(1-\vep_0)^{\tilde \ell_{(j)}} \max\left\{\binom{\hat \ell_{(j)} -1}{k+d_2-1}, \binom{\hat \ell_{(j)} +k+d_2}{k+d_2}\right\} \le (1-\vep_0)^{\hat \ell_{(j)}/2},
\]
for all large $N$, where we set 
\begin{equation}
  \label{eq-coronaday}
  \tilde \ell_{(j)}:= \tilde \ell_{j'} \; \mbox{\rm
    if $\hat \ell_{(j)}= \hat \ell_{j'}$ for some $j' \in\{ 1,2,\ldots, d\}$.}
  \end{equation}
  Hence, using the fact that
\[
\sum_{\ell \ge L} (1-\vep_0)^\ell = (1-\vep_0)^L \cdot \vep_0^{-1}
\]
we observe that
\begin{multline}\label{eq:var-sum-bd-1}
\sum_{{\bm \ell} \in \wc{\sf R}_{1,i}} \prod_{j=1}^{d} (1-\vep_0)^{ \tilde\ell_j}  \prod_{j=1}^{d_0} \binom{\hat \ell_j-1}{k+d_2-1} \cdot \prod_{j=d_0+1}^{d-1} \binom{\hat \ell_j+k+d_2}{k+d_2} \\
\le \sum_{{\bm \ell} \in \wc{\sf R}_{1,i}}  \prod_{j=1}^{i-1} \max\left\{\binom{\hat \ell_{(j)}-1}{k+d_2-1} , \binom{\hat \ell_{(j)}+k+d_2}{k+d_2}\right\} \cdot \left[\prod_{j=i}^{d} (1-\vep_0)^{ \hat\ell_{(j)}/2} \right]\\
 \le N^{\frac{(5k+d_2) i}{\log \log N}} \cdot \vep_0^{-(d-i+1)} (1-\vep_0)^{(d-i+1)N^{\frac{1}{\log \log N}}} \le N^{\frac{(5k+d_2) d}{\log \log N}} \cdot \vep_0^{-d} (1-\vep_0)^{N^{\frac{1}{\log \log N}}}. 
\end{multline}
Therefore, using \eqref{eq:indv-term} we deduce that
\begin{equation}\label{eq:kell-sl}
 \left[\sum_{{\bm \ell} \in \wc{\sf R}_{1,i}} \widetilde \kappa({\bm \ell})\cdot (\E|\widehat Q_{{\bm \ell}, k}|^2)^{1/2}\right] \le N^{-\eta_0(k - |\gd|)} \cdot (1-\vep_0)^{N^{\frac{1}{\log \log N}}},
\end{equation}
for all large $N$ and any $i \in [d]$. Now summing \eqref{eq:kell-ss}, and \eqref{eq:kell-sl} for $i \in [d]$, we deduce from \eqref{eq:P-k-cs} that 
\[
\sup_{k=|\gd|+1}^{N^{\frac{1}{\log \log N}}} \E \left[\sup_{z \in \cS_\gd^{-\vep}} |\widehat P_k(z)|^2 \cdot \prod_{i=1}^{d_0} \lambda_i(z)^{2(k+d_2)} \right] \cdot N^{\eta_0(k - |\gd|)} \le 1,
\]
 for all large $N$. 
 
 It remains to consider the case when $k$ is large. 

\noindent
{\bf Case $2$:} Fix $k \ge N^{1/\log \log N}$. Similar to the above we split the range of ${\bm \ell}$. First we consider the case when $\max_i \ell_i$ is small compared to $k$. Set  
\[
{\sf R}_2:=\{{\bm \ell}: \hat \ell_i \le k (\log N)^2 \wedge (N+d_2), i=1,2,\ldots,d\}. 
\]
Recalling the inequality $\binom{n}{m} \le \left(\frac{en}{m}\right)^m$ we note that
\[
\max\left\{\binom{\hat \ell_j -1}{k+d_2-1}, \binom{\hat \ell_j +k+d_2}{k+d_2}\right\}\le (2e (\log N)^2)^{(k+d_2)}, 
\]
for any ${\bm \ell} \in {\sf R}_2$. Thus proceeding 
as before we deduce that
\begin{multline}\label{eq:kell-ls}
  \left[\sum_{{\bm \ell} \in {\sf R}_2} \widetilde \kappa({\bm \ell}) \cdot (\E|\widehat Q_{{\bm \ell}, k}|^2)^{1/2}\right]  \le \left[ |{\sf R}_2| \cdot \max_{{\bm \ell} \in {\sf R}_2}  \left\{\widetilde \kappa({\bm \ell}) \cdot (\E|\widehat Q_{{\bm \ell}, k}|^2)^{1/2}\right\} \right]\\
  \!\!\!\!\!
  \le  |a_{d_1}|^{-k} \cdot e^{d_2} \cdot N^{-(\gamma -1/2)(k- |\gd|)} \cdot  k^{|\gd|/2} \cdot (2e (\log N)^2)^{(k+d_2)d} \cdot (k(\log N)^2)^d 
  \le N^{-\eta_0(k -|\gd|)},
\end{multline}
for all large $N$ (notation as in \eqref{eq-coronaday}). 

We finally consider the case when $\max_i \ell_i$ is large compared to $k$. For $i \in [d]$, set 
\begin{multline*}
\wc{\sf R}_{2,i}:=\Big\{{\bm \ell}:  \hat \ell_{(j)} \le k(\log N)^2 \wedge (N+d_2), \, j \in [i-1]; 
k (\log N)^2  \le \hat \ell_{(j)} \le N+d_2, j \in [d]\setminus [i-1]\Big\}.
\end{multline*}
If ${\bm \ell} \in \wc{\sf R}_{2,i}$ then $ k(\log N)^2 \le \hat \ell_{(j)} \le N+d_2$, for $j=i,i+1,\ldots,d$, which in turn implies that 
\[
\max_{j=i}^d\left\{\left(\binom{\hat \ell_{(j)} -1}{k+d_2-1} \vee \binom{\hat \ell_{(j)} +k+d_2}{k+d_2}\right) \cdot (1-\vep_0)^{\tilde \ell_{(j)}/2}\right\}\le 1,
\]
for all large $N$. Therefore for any ${\bm \ell} \in \wc{\sf R}_{2,i}$ we obtain that
\begin{align*}
 \prod_{j=1}^{d} (1-\vep_0)^{ \tilde\ell_j}  \prod_{j=1}^{d_0} \binom{\hat \ell_j -1}{k+d_2-1} \cdot \prod_{j=d_0+1}^d \binom{\hat \ell_j +k+d_2}{k+d_2}  & \le 
 \prod_{j=i}^{d} (1-\vep_0)^{ \tilde \ell_{(j)}/2} \cdot (2e(\log N)^2)^{(k+d_2)(i-1)}\\
 & \le (1-\vep_0)^{N^{\frac{1}{2\log \log N}}} \cdot (2e(\log N)^2)^{d(k+d_2)},
\end{align*}
where the last step is a consequence of the fact that for any $j=i, i+1, \ldots, d$,
\[N^{1/\log \log N} \le k \le \hat \ell_{(j)}/(\log N)^2 \le \hat \ell_{(j)} - (k+d_2) \le \tilde \ell_{(j)}.\]  Therefore, by \eqref{eq:indv-term}, 
\begin{multline}\label{eq:kell-ll}
 \left[\sum_{{\bm \ell} \in \wc{\sf R}_{2,i}} |\widetilde \kappa({\bm \ell}) \cdot (\E|\widehat Q_{{\bm \ell}, k}|^2)^{1/2}\right]  \\ 
 \le |a_{d_1}|^{-k} \cdot e^{d_2} \cdot N^{-(\gamma -1/2)(k- |\gd|)}  \cdot k^{|\gd|/2} \cdot (1-\vep_0)^{N^{\frac{1}{2\log \log N}}} \cdot (2e(\log N)^2)^{d(k+d_2)} \cdot |\wc{\sf R}_{2,i}| \\
 \le (1-\vep_0)^{N^{\frac{1}{2\log \log N}}} \cdot N^{-\eta_0(k - |\gd|)},
\end{multline}
for all large $N$. Hence, summing \eqref{eq:kell-ll} for $i \in [d]$, and \eqref{eq:kell-ls} we derive from \eqref{eq:P-k-cs} that 
\[
\sup_{k=N^{\frac{1}{\log \log N}}}^N \E \left[\sup_{z \in \cS_\gd^{-\vep}} |\widehat P_k(z)|^2 \cdot \prod_{i=1}^{d_0} \lambda_i(z)^{2(k+d_2)} \right] \cdot N^{\eta_0(k- |\gd|)} \le 1,
\]
 for all large $N$. This completes the proof of the lemma. 
\end{proof}

Next we proceed to prove {Lemma \ref{lem:rouche-multi-eq-d_0}, that is, to prove} 
 an upper bound on the second moment of the {modulus of the } dominant term $\widehat P_{|\gd|}(z)$. To derive the uniform tightness of the limiting random field it will be also useful to find a bound on the second moment of the tail of $\widehat P_{|\gd|}(z)$, i.e.~a sum over the terms appearing in the \abbr{RHS} of \eqref{eq:P-k-decompose-0} involving at least one large negative or positive exponent of $\{\lambda_i\}_{i=1}^{d_0}$ and $\{\lambda_i\}_{i=d_0+1}^d$, respectively. To this end, we introduce the following set of notation.

Fix $z \in \cS_\gd$. For any $L \in \N$ we define
\[
  |\widehat P|_{|\gd|}^L(z):= {|a_{d_1}|^{-|\gd|}}
\sum_{{\bm \ell}: \max_i \hat \ell_i \le L} \prod_{i=1}^{d_0} |\lambda_i(z)|^{-\hat \ell_i}  \prod_{i=d_0+1}^{d} |\lambda_i(z)|^{\hat \ell_i} \cdot |\widehat Q_{{\bm \ell}, |\gd|}|,
\]
where we recall the definition of $\widehat Q_{{\bm \ell}, |\gd|}$ from \eqref{eq:hat-Q-ell-k}. 
Next, for $L_1, L_2 \in \N$, such that $L_1 \le L_2$, we set
\[
|\widehat P|_{|\gd|}^{\wc{L}_1, L_2}(z):= |\widehat P|_{|\gd|}^{L_2}(z) - |\widehat P|_{|\gd|}^{L_1}(z). 
\]
Note that for $L, L_1, L_2 \le N+d_2$, with $L_1 \le L_2$, the random functions $|\widehat P|_{|\gd|}^L(z)$ and  $|\widehat P|_{|\gd|}^{\wc{L}_1, L_2}(z)$ are well defined. The next lemma derives bounds on the second moments of $|\widehat P|_{|\gd|}^L(z)$ and  $|\widehat P|_{|\gd|}^{\wc{L}_1, L_2}(z)$. 

\begin{lemma}\label{lem:dominant-tail}
Under the same set-up as in Lemma \ref{lem:rouche-multi-gr-d_0}, we have the following:

\begin{enumerate}
\item[(i)] For any $L \in \N$, there exists a constant $\widehat C_L < \infty$ such that
\[
\sup_N  \E \left[ \sup_{z \in \cS_\gd^{-\vep}} \left||\widehat P|_{|\gd|}^L(z)\right|^2\right] \le \widehat C_L.
\]

\item[(ii)] There exists a constant $\widehat c >0$, such that for any $L_1, L_2 \in \N$ with $L_1 \le L_2$,
\[
\sup_N\E \left[\sup_{z \in \cS_\gd^{-\vep}}  \left||\widehat P|_{|\gd|}^{\wc{L}_1, L_2}(z)\right|^2\right] \le \exp(-\widehat c L_1). 
\]
\end{enumerate}
\end{lemma}

For later use, for any $L \in \N$ denote
\begin{equation}\label{eq:dominant-L}
  \widehat P_{|\gd|}^L(z) := {a_{d_1}^{-|\gd|}} \sum_{{\bm \ell}: \, \max \hat \ell_i \le L} \prod_{i=1}^{d_1-\gd} \lambda_i(z)^{-\hat \ell_i} \cdot \prod_{i=d_1-\gd+1}^{d} \lambda_i(z)^{\hat \ell_i} 
\cdot \sum_{\cX_{|\gd|} \in \gL_{{\bm \ell},|\gd|}} (-1)^{\sgn(\sigma_{\mathbb{X}}) \sgn(\sigma_{\mathbb{Y}})}\det(E_N[\mathbb{X}; \mathbb{Y}]),
\end{equation}
and
\begin{equation}\label{eq:dominant-L-bar}
\widehat P_{|\gd|}^{\wc{L}}(z):=  \widehat P_{|\gd|}(z) - \widehat P_{|\gd|}^L(z).
\end{equation}
Setting $L_2 =N+d_2$,  by the triangle inequality it follows that
\[
|\widehat P_{|\gd|}^{\wc{ L}}(z)| \le |\widehat P|_{|\gd|}^{\wc{L}, N+d_2}(z).
\]
Thus, as a consequence of Lemma \ref{lem:dominant-tail}(ii) we obtain the following corollary.

\begin{corollary}\label{cor:dominant-tail}
Under the same set-up as in Lemma \ref{lem:rouche-multi-gr-d_0}, for any $L \in \N$,
\[
\sup_N \E\left[ \sup_{z \in \cS_\gd^{-\vep}} \left|\widehat P_{|\gd|}^{\wc{ L}}(z)\right|^2 \right] \le  \exp(-\widehat c L),
\]
where $\widehat c$ is as in Lemma \ref{lem:dominant-tail}. 
\end{corollary}


\begin{proof}[Proof of Lemma \ref{lem:dominant-tail} {(assuming  Lemma \ref{lem:combinatorial-1})}]
We begin by noting that, for any $L \in \N$,  
\begin{equation}\label{eq:dom-non-tail}
\E\left[ \sup_{z \in \cS_\gd^{-\vep}}  \left||\widehat P_{|\gd|}^L|(z)\right|^2\right] \le  \left(\sum_{\bm \ell: \max_i \hat \ell_i \le L} \widehat \kappa({\bm \ell})\cdot (\E|\widehat Q_{{\bm \ell}, |\gd|}|^2)^{1/2}\right)^2,
\end{equation}
where 
\[
\widehat \kappa({\bm \ell}):= \sup_{z \in \cS_\gd^{-\vep}} |\kappa^\star({\bm \ell}, z)|
\]
and 
\[
\kappa^\star({\bm \ell}, z) := a_{d_1}^{-k} 
\cdot \prod_{i=1}^{d_0} \lambda_i(z)^{-\hat \ell_i} \cdot \prod_{i=d_0+1}^d \lambda_i(z)^{\hat \ell_i}.
\] 
The proof of \eqref{eq:dom-non-tail} is analogous to that of \eqref{eq:P-k-cs}. Proceeding as in \eqref{eq:indv-term} we find that
\begin{multline*}
 \sup_{z \in \cS_\gd^{-\vep}} |\kappa^\star({\bm \ell}, z)| \cdot (\E|\widehat Q_{{\bm \ell}, |\gd|}|^2)^{1/2} 
 \le   |a_{d_1}|^{-|\gd|} \cdot |\gd|^{|\gd|/2}\cdot \prod_{i=1}^{d} (1-\vep_0)^{{\hat { \ell}_i}}  \cdot \prod_{i=1}^{d_0} \binom{\hat \ell_i-1}{|\gd|+d_2-1} \cdot \prod_{i=d_0+1}^{d} \binom{\hat \ell_i+|\gd|+d_2}{|\gd|+d_2}.
\end{multline*}
Therefore
\begin{multline}
  \sum_{{\bm \ell}: \max_i \hat \ell_i \le L} \widehat \kappa({\bm \ell}) \cdot (\E|\widehat Q_{{\bm \ell}, {|\gd|}}|^2)^{1/2} \\
\le  |a_{d_1}|^{-|\gd|} \cdot |\gd|^{|\gd|/2}  \sum_{{\bm \ell}: \max_i \hat \ell_i \le L}  \prod_{i=1}^{d} (1-\vep_0)^{\hat { \ell}_i} \cdot \prod_{i=1}^{d_0} \hat \ell_i^{2d}  \cdot \prod_{i=d_0+1}^{d} (\hat \ell_i+2d)^{2d}\le \widehat C_L^{1/2},
\end{multline}
for some constant $C_L < \infty$. This together with \eqref{eq:dom-non-tail} yields part (i) of the lemma. To prove the second part we proceed similarly as in \eqref{eq:dom-non-tail} to deduce that
\begin{equation}\label{eq:dom-tail}
\E\left[ \sup_{z \in \cS_\gd^{-\vep}}  \left||\widehat P_{|\gd|}^{\wc{L}_1, L_2}|(z)\right|^2\right] \le  \left(\sum_{\bm \ell: \max_i \hat \ell_i >  L_1} \widehat \kappa({\bm \ell}) \cdot (\E|\widehat Q_{{\bm \ell}, k}|^2)^{1/2}\right)^2,
\end{equation}
Since
\[
\sum_{\ell \ge 0} (\ell+2d)^{2d} \cdot (1-\vep_0)^\ell  < \infty \qquad \text{ and } \sum_{\ell \ge 0}  (\ell+2d)^{2d} \cdot (1-\vep_0)^\ell  \le (1-\widetilde \vep)^L,
\]
for any $L \in \N$ and some $\widetilde \vep >0$, we use \eqref{eq:indv-term} again to deduce that
\[
\sum_{{\bm \ell}: \max_i \hat \ell_i > L} \widehat \kappa({\bm \ell}) \cdot (\E|\widehat Q_{{\bm \ell}, k}|^2)^{1/2} \le \exp( - (\widehat c /2)\cdot L_1). 
\]
Plugging this bound in \eqref{eq:dom-tail} completes the proof of the lemma.  
\end{proof}
The proof of Lemma \ref{lem:rouche-multi-eq-d_0} is now immediate.
\begin{proof}[Proof of Lemma \ref{lem:rouche-multi-eq-d_0}]
As
\[
|\widehat P_{|\gd|}(z)| \le |\widehat P|_{|\gd|}^L(z) + |\widehat P_{|\gd|}^{\wc{ L}}(z)|,
\]
for any $L \in \N$, the conclusion of Lemma \ref{lem:rouche-multi-eq-d_0} is immediate from Lemma \ref{lem:dominant-tail}(i) and Corollary \ref{cor:dominant-tail}.
\end{proof}
Next we provide the proof of Lemma \ref{lem:combinatorial-1} which has been used in the proofs of Lemmas \ref{lem:rouche-multi-gr-d_0} and \ref{lem:dominant-tail}. Recall that Lemma \ref{lem:combinatorial-1} yields a bound on $\mathfrak{N}_{{\bm \ell}, k}$ for any $k= |\gd|, |\gd|+1, \ldots, N$.

\begin{proof}[Proof of Lemma \ref{lem:combinatorial-1}]
To prove \eqref{eq:N-l-k-bd} we need to consider the cases $\gd \ge 0$ and $\gd <0$ separately. First let us consider the case $\gd \ge 0$.

To this end, denote
\begin{equation}\label{eq:delta-dfn}
\delta_{i,j}:= \delta_{i,j}(\cX_k):=\left\{ \begin{array}{ll}
x_{i,j} -x_{i+1,j} & \mbox{ for } i \in [d_0] \text{ and } j \in [k+d_2]\\
x_{i+1,1} & \mbox{ for } i \in [d]\setminus [d_0], \ j =1\\
x_{i+1,j}-x_{i,j-1} &\mbox{ for } i \in [d]\setminus [d_0], \ j \in [k+d_2]\setminus \{1\}\\
N+d_2 - x_{i,k+d_2} & \mbox{ for } i \in [d]\setminus [d_0], \ j =k+d_2+1.
\end{array}
\right.
\end{equation}
We claim that \corAB{for $\cX_k \in \gL_{{\bm \ell},k}$, with $k \ge \gd$, the set of integers $\{\delta_{i,j}(\cX_k)\}$} fixes the choices of $\{x_{d+1,j}\}_{j=1}^{\gd+d_2}$. To see this we note that for any pair of integers $k$ and $j$ such that $k \ge \gd$ and $d_2+1 \le j \le \gd + d_2$,
we have
\[
j \le k+d_2 \qquad \text{ and } \qquad d-j+1 \ge d-d_2 -\gd + 1 = d_1-\gd+1=d_0+1.
\]
Therefore
\begin{equation}\label{eq:x_d+1-telescopic}
x_{d+1,j} = \sum_{i=d-j+2}^{d} (x_{i+1,j-(d-i)}- x_{i,j-(d-i)-1}) + x_{d-j+2,1}=\sum_{i=d-j+1}^{d} \delta_{i,j-(d-i)}(\cX_k),
\end{equation}
where the last equality follows from the definition \eqref{eq:delta-dfn} of the $\{\delta_{i,j}\}$'s. This proves that $\{\delta_{i,j}(\cX_k)\}$ fixes the choices of $\{x_{d+1,j}\}_{j=d_2+1}^{d_2+\gd}$. As $\cX_k \in \gL_{{\bm \ell}, k}$ we also have that
\begin{equation}\label{eq:x_d+1-contraint}
x_{d+1, j} = j, \quad j=1,2,\ldots, d_2.
\end{equation}
The last two observations prove the claim.

To complete the proof of the bound on $\mathfrak{N}_{{\bm \ell},k}$, for $\gd \ge 0$, we note that the remaining indices of $X_{d+1}$, i.e.~$\{x_{d+1,j}\}_{j=\gd+d_2+1}^{k+d_2}$ can be chosen in $\binom{N+d_2}{k- \gd}$ ways.
The fact that $\cX_k \in \gL_{{\bm \ell}, k} \subset L_{{\bm \ell}, k}$ also implies that $\{\delta_{i,j}(\cX_k)\}$ can be chosen in 
\begin{equation}\label{eq:choice-delta-bd}
\prod_{i=1}^{d_0} \binom{\hat \ell_i-1}{k+d_2-1} \cdot \prod_{i=d_0+1}^{d} \binom{\hat \ell_i+k+d_2}{k+d_2}
\end{equation}
ways. 

From \eqref{eq:delta-dfn} it is immediate that choosing $\{\delta_{i,j}(\cX_k)\}$ and $X_{d+1}$ fixes $\cX_k$. So, to find the bound on $\mathfrak{N}_{{\bm \ell},k}$ we then need to find the number of choices $\cX_k' \in \gL_{{\bm \ell}, k} \subset L_{{\bm \ell},k}$ such that $X_{d+1}'=X_{d+1}$. This amounts to choosing only $\{\delta_{i,j}(\cX_k')\}$, and the number of such choices, as already seen above, in bounded by \eqref{eq:choice-delta-bd}. Therefore, combining the above bounds we arrive at the desired bound for $\mathfrak{N}_{{\bm \ell},k}$, when $\gd \ge 0$. 

It remains to prove \eqref{eq:N-l-k-bd} for $\gd <0$. To this end, we claim that choosing $\{\delta_{i,j}(\cX_k)\}_{i,j}$ automatically fixes $\{x_{d+1,j}\}_{j=k+d_2 +\gd+1}^{k+d_2}$ for any $\cX_k \in \gL_{{\bm \ell},k}$ and $k \ge |\gd|$.

Indeed, for any $\cX_k \in \gL_{{\bm \ell},k}$ the indices $\{x_{1,k+\ell}\}_{\ell=1}^{d_2}$ are fixed. Therefore, choosing $\{\delta_{i,j}(\cX_k)\}$ fixes the indices $\{x_{d_0+1,k+\ell}\}_{\ell=1}^{d_2}$ (recall \eqref{eq:delta-dfn}). Now similar to \eqref{eq:x_d+1-telescopic} we observe that for any $j$ such that $ d-d_0+1 \le j \le d_2$
\[
x_{d+1,k+j} = \sum_{i=d_0+1}^{d} \delta_{i,k+j-(d-i)} + x_{d_0+1,k+j - (d-d_0)}.
\]
Therefore choosing $\{\delta_{i,j}(\cX_k)\}$ also fixes $\{x_{d+1, k+j}\}_{j=d-d_0+1}^{d_2}$ and hence the claim. On the other hand $\{x_{d+1,j}\}_{j=1}^{d_2}$ are fixed by the definition of $\gL_{{\bm \ell}, k}$. Now repeating the same argument as in the case $\gd\geq 0$, we arrive at the bound \eqref{eq:N-l-k-bd} for $\gd <0$. We omit further details.
\end{proof}

Finally we proceed to the proof of Lemma \ref{lem:rouche-multi-le-d_0}. 
We begin with  the following lemma that shows that if $k < |\gd|$ then $\gL_{{\bm \ell}, k} =\emptyset$ unless the sum of the $\hat \ell_i$'s is close to $N$.
The proof appears in  \cite[Proof of Lemma 4.3]{BPZ-non-triang}, see (4.24) there.
\begin{lemma}[{\cite[Lemma 4.3]{BPZ-non-triang}}]
\corAB{Fix an integer $k < |\gd|$ where $\gd$ as in Lemma \ref{lem:rouche-multi-gr-d_0}}.
 Then for any ${\bm \ell}:=(\ell_1,\ell_2,\ldots,\ell_d)$ such that $0 \le \ell_i \le N+d_2$, for all $i \in [d]$,
\[
\gL_{{\bm \ell}, k} \ne \emptyset \Rightarrow \sum_{i=1}^d \hat \ell_i \ge N - d^2,
\]
where $\gL_{{\bm \ell}, k}$ and $\{\hat \ell_i\}_{i=1}^d$ are as in \eqref{eq:gL-ell-k} and \eqref{eq:hat-ell}, respectively.
\label{lem:combinatorial-2}
\end{lemma}

We now prove Lemma \ref{lem:rouche-multi-le-d_0}.
\begin{proof}[Proof of Lemma \ref{lem:rouche-multi-le-d_0}]
As in the proofs 
of Lemmas \ref{lem:rouche-multi-gr-d_0} and \ref{lem:rouche-multi-eq-d_0}, 
a key step is a bound on $\mathfrak{N}_{{\bm \ell},k}$  of \eqref{eq:N-ell-k}. Since $k < |\gd|$ we cannot use the bound derived in Lemma \ref{lem:combinatorial-1}. Instead, we argue as follows.

We noted in the proof of Lemma \ref{lem:combinatorial-1} that
 choosing $\{\delta_{i,j}(\cX_k)\}$ and $X_{d+1}$ fixes the choice of $\cX_k$. Therefore it follows that for any $k \le d$, 
\begin{align}\label{eq:gn-simple-bd}
\mathfrak{N}_{{\bm \ell}, k}  & \le (N+d_2)^{k+d_2} \cdot \prod_{i=1}^{d_0} \binom{\hat \ell_i-1}{k+d_2-1}^2 \cdot \prod_{i=d_0+1}^{d} \binom{\hat \ell_i+k+d_2}{k+d_2}^2\\
& \le (N+d_2)^{k+d_2} \cdot (N+d_2)^{2(k+d_2)} \cdot (N+k+2d_2)^{2(k+d_2)} \le (2N)^{10d}, \notag
\end{align}
for all large $N$, where the second inequality  follows from the fact that $\hat \ell_i \le N+d_2$ for all $i \in [d]$. Now Lemma \ref{lem:combinatorial-2} yields that
\[
\widehat P_k(z)= a_{d_1}^{-k} N^{-\gamma (k-|\gd|)} \sum_{{\bm \ell}}^* \prod_{i=1}^{d_0} \lambda_i(z)^{-\hat \ell_i}  \prod_{i=d_0+1}^{d} \lambda_i(z)^{\hat \ell_i} \cdot \widehat Q_{{\bm \ell}, k},
\]
where the sum $\sum_{\bm \ell}^*$ is taken over all ${\bm \ell}$ such that $\sum_{i=1}^d \hat \ell_i \ge N -d^2$, and hence using the
Cauchy-Schwarz inequality and arguing similarly to \eqref{eq:P-k-cs}, we obtain that for any $k=0,1,2, \ldots, |\gd|-1$, 
\begin{equation}\label{eq:Pk-small-k}
\E\left[\sup_{z \in \cS_\gd^{-\vep}}  |\widehat P_{k}(z)|^2\right] \le \left(\sum_{\bm \ell}^* \widehat \kappa({\bm \ell}) \cdot (\E|\widehat Q_{{\bm \ell}, k}|^2)^{1/2}\right)^2.
\end{equation}
Using Lemma \ref{lem:root-cont-0} and \eqref{eq:gn-simple-bd}, 
and proceeding similarly as in \eqref{eq:indv-term}, for any ${\bm \ell}$ such that $\sum_{i=1}^d \hat \ell_i \ge N- d^2$, we derive that 
\[
 \sup_{z \in \cS_\gd^{-\vep}} |\kappa^\star({\bm \ell}, z)| \cdot (\E|\widehat Q_{{\bm \ell}, k}|^2)^{1/2} \le |a_{d_1}|^{-k} \cdot N^{\gamma d} \cdot \sqrt{d!} \cdot (2N)^{5d} \cdot (1-\vep_0)^{N/2},
\]
for all large $N$. 
Hence, from \eqref{eq:Pk-small-k} it is now immediate that
\[
\sup_{k=0}^{|\gd|-1}\E\left[\sup_{z \in \cS_\gd^{-\vep}} |\widehat P_k(z)|^2\right] \le (1-\vep_0)^{N/4},
\]
for all large $N$. This completes the proof of the lemma. 
\end{proof}

We end this section with the proof of Corollary \ref{cor:non-dominant}. 

\begin{proof}[Proof of Corollary \ref{cor:non-dominant}]
As 
$
\widehat \det_N(z)= \sum_{k=0}^N \widehat P_k(z),
$
see \eqref{eq;spec-rad-new}-\eqref{eq:hat-det},
applying the Cauchy-Schwarz inequality we find that
\begin{multline}\label{eq:hat-det-decompose}
\E\left[ \sup_{z \in \cS_\gd^{-\vep}}  \left|\widehat \det_N(z) - \widehat P_{|\gd|}(z)\right|^2 \right] \le \E \left[ \sup_{z \in \cS_\gd^{-\vep}}  \sum_{k, k' \ne |\gd|} |\widehat P_k(z)| \cdot |\widehat P_{k'}(z)| \right] \\
\!\!\!\!\!\!
\!\!\!\!\!\!
\le \sum_{k, k' \ne |\gd|} \left\{\E\left[\sup_{z \in \cS_\gd^{-\vep}} |\widehat P_k(z)|^2\right]\right\}^{1/2} \cdot \left\{\E\left[\sup_{z \in \cS_\gd^{-\vep}} |\widehat P_{k'}(z)|^2\right]\right\}^{1/2} 
= \left(\sum_{k \ne |\gd|} \left\{\E\left[\sup_{z \in \cS_\gd^{-\vep}} |\widehat P_k(z)|^2\right]\right\}^{1/2}\right)^2.
\end{multline}
The proof of the corollary now completes upon applying Lemmas \ref{lem:rouche-multi-gr-d_0} and \ref{lem:rouche-multi-le-d_0}. Further details are omitted. 
\end{proof}

\section{Tightness of the limiting random field}\label{sec:limit-tight}
Our main 
goal in this section 
is to derive the tightness of the limiting random field $\gP_\gd^\infty(\cdot)$.  Recall that 
Lemma \ref{lem:limit} states 
that the random fields $\{\gP_\gd^L(\cdot)\}_{L \in \N}$ 
approximate $\gP_\gd^\infty(\cdot)$. 
Thus to show the tightness of $\gP_\gd^\infty(\cdot)$,
it will suffice to show the uniform 
tightness of the random fields $\{\gP_\gd^L(\cdot)\}_{L \in \N}$, and to
control the convergence. 
We will further derive an exponential decay of 
the tail of $\gP_\gd^\infty(\cdot)$. 
These results 
will eventually be 
used in the proofs of the main result Theorem \ref{thm:outlier-law} and Lemma \ref{lem:limit}. 

We introduce the following notation.
For any $L \in \N \cup \{\infty\}$, and $\underline \gc:=(\gc_1,\gc_2, \ldots, \gc_d)$, where $\{\gc_i\}_{i=1}^d$ are non-negative integers, we set 
\begin{equation}\label{eq:random-fld-abs}
|\gP|_\gd^L(z) := \sum_{\underline \gc: \max_i \gc_i \le L} \prod_{i=1}^{d_0} |\lambda_i(z)|^{-\gc_i} \prod_{i=d_0+1}^d   |\lambda_i(z)|^{\gc_i} \cdot \left|\sum_{\gx \in \gL_1(\gd)} \sum_{\gy \in \gL_2(\gd)} (-1)^{\gz(\gx,\gy)}\det(E_\infty[\widehat \gX; \widehat \gY])\cdot  \prod_{i=1}^d {\bf 1}_{\{ \gc_i(\gx, \gy) =\gc_i\}}\right|,
\end{equation}
where we refer the reader to Definitions \ref{dfn:random fieldnot} and \ref{dfn:random field} to recall the definitions of the notation $\gL_1(\gd)$, $\gL_2(\gd)$, $\{\gc_i(\gx, \gy)\}_{i=1}^d$, $\widehat \gX$, $\widehat \gY$, and $E_\infty$. For any $L_1, L_2 \in \N$ such that $L_1 < L_2$ we set
\[
\gP_\gd^{\wc{L}_1, L_2}(z):= \gP_\gd^{L_2}(z) - \gP_\gd^{L_1}(z) \qquad \text{ and } \qquad |\gP|_\gd^{\wc{L}_1, L_2}(z):= |\gP|_\gd^{L_2}(z) - |\gP|_\gd^{L_1}(z). 
\] 

\begin{lemma}\label{lem:random-field-abs}
Fix $\vep >0$ and $\gd \ne 0$ an integer such that $-d_2 \le \gd \le d_1$. Let $E_\infty$ be a semi-infinte array of i.i.d.~random variables with zero mean and unit variance. Then we have the following: 
\begin{enumerate}
\item[(i)] For any $L \in \N$,
\[
\E \left[ \sup_{z \in \cS_\gd^{-\vep}} \left(|\gP|_\gd^L(z) \right)^2 \right] \le \widehat C_L,
\]
where $\widehat C_L$ is as in Lemma \ref{lem:dominant-tail}.

\item[(ii)] Fix $L_1\in \N$. Then 
\[
\sup_{L_2: L_2 >L_1}  \E \left[ \sup_{z \in \cS_\gd^{-\vep}}  \left(|\gP|_\gd^{\wc{ L}_1, L_2}(z) \right)^2 \right] \le \exp(- \widehat c L_1),
\]
where $\widehat c$ is as in Lemma \ref{lem:dominant-tail}.
\end{enumerate}
\end{lemma}

Building on Lemma \ref{lem:random-field-abs} we have the next result.

\begin{corollary}\label{cor:random-field}
Under the same setup as in Lemma \ref{lem:random-field-abs} we have the following:
\begin{enumerate}
\item[(i)] For any $L \in \N$,
\[
 \E \left[ \sup_{z \in \cS_\gd^{-\vep}} \left|\gP_\gd^L(z) \right|^2 \right] \le \widehat C_L,
\]
where $\widehat C_L$ is as in Lemma \ref{lem:dominant-tail}.

\item[(ii)] Fix $L_1\in \N$. Then 
\[
\sup_{L_2: L_2 >L_1}  \E \left[ \sup_{z \in \cS_\gd^{-\vep}} \left|\gP_\gd^{\wc{ L}_1, L_2}(z) \right|^2 \right] \le \exp(- \widehat c L_1),
\]
where $\widehat c$ is as in Lemma \ref{lem:dominant-tail}.

\item[(iii)] Consequently, $\gP_\gd^\infty(\cdot)$ is well defined and $\gP_\gd^L(\cdot) \to \gP_\gd^\infty(\cdot)$ uniformly on $\cS_\gd^{-\vep}$, as $L \to \infty$, on a set of probability one. Moreover, 
\begin{equation}\label{eq:limit-tail}
\lim_{L \to \infty}  \E \left[ \sup_{z \in \cS_\gd^{-\vep}} \left|\gP_\gd^\infty(z) - \gP_\gd^L(z) \right|^2 \right] =0, 
\end{equation}
and
\begin{equation}\label{eq:limit-tight}
 \E \left[ \sup_{z \in \cS_\gd^{-\vep}} \left|\gP_\gd^\infty(z) \right|^2 \right] <\infty. 
\end{equation}
\end{enumerate}
\end{corollary}

Using Lemma \ref{lem:random-field-abs} let us first provide a proof of Corollary \ref{cor:random-field}. 

\begin{proof}[Proof of Corollary \ref{cor:random-field} {(assuming
  Lemma \ref{lem:random-field-abs})}]
{By the
triangle inequality we have that}
\[\left|\gP_\gd^L(z) \right| \le |\gP|_\gd^L(z) \quad \text{ and } \quad \left|\gP_\gd^{\wc{ L}_1, L_2}(z) \right| \le |\gP|_\gd^{\wc{ L}_1, L_2}(z).
\]
{Thus (i) and (ii) follow
  from Lemma \ref{lem:random-field-abs}. We proceed to prove (iii). For any $L \in \N$, write}
\[
\mathscr{P}_L:=  \max_{L': L < L' \le  2L} \sup_{z \in \cS_\gd^{-\vep}} \left|\gP_\gd^{\wc{ L}, L'}(z) \right|^2.
\]
Note that
\[
\mathscr{P}_L \le \sum_{L'=L+1}^{2L} \sup_{z \in \cS_\gd^{-\vep}} \left|\gP_\gd^{\wc{ L}, L'}(z)\right|^2.
\]
Thus, by part (ii)
\[
\E[\mathscr{P}_L] \le L \cdot \exp(-\widehat c L) \le \exp\left(-\frac{\widehat c L}{2}\right),
\]
for all large $L$. This, upon using the
triangle inequality, further yields that 
\begin{equation}\label{eq:second-mom-sup}
\E\left[\sup_{L' \ge L} \sup_{z \in \cS_\gd^{-\vep}} \left|\gP_\gd^{\wc{ L}, L'}(z)\right|^2\right] \le \sum_{i=0}^\infty \E\left[\mathscr{P}_{2^i L} \right] \le \exp\left(-\frac{\widehat c L}{4}\right),
\end{equation}
for all large $L$. Now the bound \eqref{eq:second-mom-sup} together with Markov's inequality and the
Borel-Cantelli Lemma show
that on a set of probability one the random functions $\{\gP_\gd^L(\cdot)\}$ are uniformly Cauchy on $\cS_\gd^{-\vep}$. So, the limit $\gP_\gd^\infty(\cdot)$, as $L \to \infty$, exists and is 
well defined on that set of probability one. Furthermore, {the convergence
is  uniform} on $\cS_\gd^{-\vep}$.  

Finally, to prove \eqref{eq:limit-tail} we simply note that 
\[
\sup_{z \in \cS_\gd^{-\vep}} |\gP_\gd^\infty(z) - \gP_\gd^L(z)|^2 \le  \sup_{L' \ge L} \sup_{z \in \cS_\gd^{-\vep}} \left|\gP_\gd^{\wc{ L}, L'}(z)\right|^2,
\]
whereas \eqref{eq:limit-tight} follows from \eqref{eq:limit-tail} and part (i) of this corollary. 
This completes the proof of the corollary.  
\end{proof}

The rest of this section will be devoted to the proof of Lemma \ref{lem:random-field-abs}. Recall from Section \ref{sec:identify-non-dom-term} that analogous bounds were proved for the dominant term in the expansion of 
$\widehat \det_N(z)$, see Lemma \ref{lem:dominant-tail}. We will prove Lemma \ref{lem:random-field-abs} by establishing an equality of the laws of the random fields $\{|\widehat P|_\gd^L (z)\}_{z \in \cS_\gd}$ and $\{|\gP|_\gd^L(z)\}_{z \in \cS_\gd}$, for any $L \in \N$ and all large $N$. This is the content of our next result. 

\begin{lemma}\label{lem:zeta-N-gd-L}
Fix $L, L_1, L_2 \in \N$ such that $L_1 <L_2$, and an integer $\gd \ne 0$ such that $-d_2 \le \gd \le d_1$. Then for all large $N$ (depending only on $L$ and $L_2$), we have the following:

\begin{enumerate}

\item[(i)] The joint laws of the random fields $\{(|\widehat P|_{|\gd|}^{L_1}(z), |\widehat P|_{|\gd|}^{L_2}(z))\}_{z\in \cS_\gd}$ and $\{(|\gP|_\gd^{L_1}(z), |\gP|_\gd^{L_2}(z))\}_{z \in \cS_\gd}$ coincide,

\end{enumerate}
and
\begin{enumerate}

\item[(ii)] The laws of the random fields $\{\widehat P_{|\gd|}^L(z)\}_{z\in \cS_\gd}$ and $\{\gP_\gd^L(z)\}_{z \in \cS_\gd}$ coincide.

\end{enumerate}

\end{lemma}

We note that equipped with Lemma \ref{lem:zeta-N-gd-L}(i) the proof of Lemma \ref{lem:random-field-abs} is immediate from Lemma \ref{lem:dominant-tail}. Further details are omitted. Thus, it only remains to prove Lemma \ref{lem:zeta-N-gd-L}.

Before presenting the proof,  we recall all relevant notation and provide a sketch. For brevity we sketch the proof of Lemma \ref{lem:zeta-N-gd-L}(ii). The idea behind the proof of the first part is the same. 

From \eqref{eq:dominant-L} we have that
\begin{equation}\label{eq:dom-L-redfn}
\widehat P_{|\gd|}^L(z) := \sum_{{\bm \ell}: \, \max \hat \ell_i \le L} \prod_{i=1}^{d_1-\gd} \lambda_i(z)^{-\hat \ell_i} \cdot \prod_{i=d_1-\gd+1}^{d} \lambda_i(z)^{\hat \ell_i} 
\cdot \sum_{\cX_{|\gd|} \in \gL_{{\bm \ell},|\gd|}} (-1)^{\sgn(\sigma_{\mathbb{X}}) \sgn(\sigma_{\mathbb{Y}})}\det(E_N[\mathbb{X}; \mathbb{Y}]),
\end{equation}
where  $\cX_{d}:= (X_1,X_2,\ldots, X_{d+1})$, {$d_0=d_1-\gd$,}
\[
X_i:=\{x_{i,1} < x_{i,2} < \cdots < x_{i, |\gd|+d_2}\}, \qquad i \in [d+1],
\]
\begin{equation}\label{eq:gL-redfn}
\gL_{{\bm \ell},|\gd|}:= \{\cX_{|\gd|} \in L_{{\bm \ell}, |\gd|}: x_{1,|\gd|+j}=N+j; \ j  \in [d_2] \quad \text{ and } \quad x_{d+1,j}=j; \ j \in [d_2]\},
\end{equation}
\begin{multline}
L_{{\bm \ell},|\gd|}:= \{ \cX_{|\gd|}:  \, 1 \le x_{i+1,1} \le x_{i,1} < x_{i+1,2} \le x_{i,2} < \cdots < x_{i+1,|\gd|+d_2} \le x_{i,|\gd|+d_2} \le N+d_2,\\
 \, \sum_{j=1}^{|\gd|+d_2}(x_{i,j}-x_{i+1,j}+1)   =\hat \ell_i; \ i=1,+2,\ldots,d_0,\\
\text{ and }  \, x_{i+1,1}+ \sum_{j=2}^{|\gd|+d_2}(x_{i+1,j}-x_{i,j-1})+ (N+d_2-x_{i,|\gd|+d_2}) =\hat \ell_i+|\gd|+d_2; \ i=d_0+1,d_0+2,\ldots,d\}, \label{eq:L-redfn}
\end{multline}

\begin{equation}
\hat \ell_i := \left\{\begin{array}{ll} \ell_i & \mbox{ if } i > d_0\\
N+d_2- \ell_i & \mbox{ if } i \le d_0
\end{array}
\right.,\notag
\end{equation}
 and
\begin{equation}\label{eq:redef-bX}
 \mathbb{X}:=\mathbb{X}(X_1):= X_1 \cap [N] \qquad \text{ and } \qquad \mathbb{Y}:= \mathbb{Y}(X_{d+1}) := (X_{d+1} - d_2) \cap [N], 
\end{equation}
see Definitions \ref{dfn:young-restrict}, \ref{dfn:random fieldnot} and Equations
 \eqref{eq-Lnew}, \eqref{eq:L-ell-k}
and \eqref{eq:hat-ell}.
Since $\widehat P_{|\gd|}^L(z)$ involves entries of the noise matrix $E_N$, it is not apriori clear that the distribution of the random field $\{\widehat P_{|\gd|}^L(z)\}_{z \in \cS_\gd}$ is free of $N$, for all large $N$. We find affine maps
that map bijectively the relevant subset of $\cX_{|\gd|}$ 
to that of $(\gx, \gy)$, for all large $N$, where $\gx \in \gL_1(\gd)$ and $\gy \in \gL_2(\gd)$ 
with $\gL_1(\gd)$ and $\gL_2(\gd)$ as in Definition \ref{dfn:random fieldnot}. The rational behind the existence of 
such affine maps is that as $\cX_{|\gd|} \in \gL_{{\bm \ell}, \gd}$, the restriction $\max \hat \ell_i \le L$ ensures that for all large $N$, a sub-collection of the array of integers $\{x_{i,j}\}_{i \in [d+1], j \in [d_2 +|\gd|]}$ is $O(L)$, whereas the rest are $N -O(L)$. This induces the affine transformations. This observation further leads to a partition of $\cX_{|\gd|}$ which then gives the shapes of the tableaux appearing in Definition \ref{dfn:random fieldnot}. See Figure \ref{fig:3} for a pictorial description of these observations.

To complete the argument we then confirm that 
\begin{equation}\label{eq:sign-c}
\gc_i(\gx,\gy)=\hat \ell_i \text{ for all } i \in [d] \quad \text{ and } \quad (-1)^{\sgn(\sigma_{\mathbb{X}}) \sgn(\sigma_{\mathbb{Y}})} = (-1)^{\gz(\gx,\gy)},
\end{equation}
under those maps, where $\{\gc_i(\gx,\gy)\}_{i \in [d]}$ and $\gz(\gx,\gy)$ are as in Definition \ref{dfn:random fieldnot}. 

The above mentioned maps also induce mappings between the entries of $E_N$ 
and those of $E_\infty$. Since all maps are bijections, using the fact that the entries of $E_N$ are i.i.d., it follows that joint law of the random variables under the summation in the \abbr{RHS} of \eqref{eq:dom-L-redfn} 
is the same as that of \eqref{eq:random-fld}. This establishes the equality in the distributions of the random fields $\{\widehat P_{|\gd|}^L(z)\}_{z \in \cS_\gd}$ and $\{\gP_\gd^L(z)\}_{z \in \cS_\gd}$. Below we carry out in detail these steps.

\begin{figure}[h]
 \captionsetup{singlelinecheck=false}
\centering	
\ytableausetup{centertableaux, boxsize=2em}
\begin{ytableau}
   \none & \none & x_{1,1} & x_{1,2} & x_{1,3} & x_{1,4} \\
 \none & \none& x_{2,1} & x_{2,2} & x_{2,3} & x_{2,4} \\
 \none & \none & x_{3,1} & x_{3,2} & x_{3,3} & x_{3,4} \\
 \none & x_{4,1} & \phantom{} & x_{4,2} & x_{4,3} & x_{4,4} \\
 \none & x_{5,1} & x_{5,2} & \phantom{} & x_{5,3} & x_{5,4} \\
  \none & x_{6,1} & x_{6,2} & x_{6,3} & \phantom{} & x_{6,4}\\
   \none & x_{7,1} & x_{7,2} & x_{7,3} & x_{7,4} & \phantom{}
\end{ytableau}
\qquad\qquad\qquad\qquad\qquad
\begin{ytableau}
         \none & x_{1,1} & \phantom{} & x_{1,2} & x_{1,3} & x_{1,4} \\
 \none & x_{2,1} & \phantom{} & x_{2,2}  & x_{2,3} & x_{2,4}  \\
\none & x_{3,1} & \phantom{} & x_{3,2}  & x_{3,3} & x_{3,4}  \\
\none & x_{4,1} & \phantom{} & x_{4,2}  & x_{4,3} & x_{4,4}  \\
\none & x_{5,1} & \phantom{} & x_{5,2}  & x_{5,3} & x_{5,4}  \\
 \none & x_{6,1} & x_{6,2} & \phantom{} & x_{6,3} & x_{6,4} \\
 \none & x_{7,1} & x_{7,2} & x_{7,3} & \phantom{} & x_{7,4} 
\end{ytableau}
\caption[]{A schematic representation of the entries of the set $\cX_{|\gd|} \in \gL_{{\bm \ell}, |\gd|}$, for $d_1=d_2=3$ and $|\gd|=1$. The condition $\max \hat \ell_i \le L$ induces a partition illustrated by the empty boxes, for $\gd =1$ (\texttt{left panel}) and $\gd =-1$ (\texttt{right panel}). For all large $N$, in both panels the entries in the left block (demarcated by the empty boxes) are $O(L)$, whereas the entries in the other block are $N - O(L)$. Furthermore, rotating the left blocks in both panels clockwise and the right blocks anti-clockwise we note that the shapes of tableaux thus produced matches with those appearing in Figure \ref{fig:2}.
}
\label{fig:3}
\end{figure}

\begin{proof}[Proof of Lemma \ref{lem:zeta-N-gd-L}]
  We will only prove (i). The proof of (ii), {being similar},
  will be omitted. Fix $L_1, L_2 \in \N$ such that $L_1 <L_2$.

As indicated by Figure \ref{fig:3} the forms of the maps differ for $\gd >0$ and $\gd <0$. For $\gd >0$,  denote
\[
(\gG_N^+(\cX_{|\gd|}))_{i,j}:= N+d_2+1 - x_{j,\gd+d_2-i+1}, \quad j \le d+1 -i \text{ and } i = 1,2, \ldots, \gd + d_2,
\]
and
\[
(\gH_N^+(\cX_{|\gd|}))_{i,j}:= x_{d+2-j,i}, \quad j \le d-d_0+1-i \text{ and } i = 1,2, \ldots,  \gd+d_2.
\]
For $\gd <0$, denote
\[
(\gG_N^-(\cX_{|\gd|}))_{i,j}:= N+d_2+1 - x_{j,d_2-\gd-i+1}, \quad j \le \min\{ d+ 1 - \gd -i, d+1\} \text{ and } i = 1,2, \ldots,  d_2
\]
and
\[
(\gH_N^-(\cX_{|\gd|}))_{i,j}:= x_{d+2-j,i},
\]
for  $j \le d+1, \,  i = 1,2, \ldots,  -\gd$ and $j \le d_2+1 - i$, $i= -\gd+1, -\gd+2,\ldots, d_2$.

Consider the case $\gd >0$. It is clear from their definitions that the shapes of the tableaux induced by $\gG^+(\cX_{|\gd|})$ and $\gH_N^+(\cX_{|\gd|})$ are given by ${\bm \mu}_1$ and ${\bm \mu}_2$, where ${\bm \mu}_1$ and ${\bm \mu}_2$ are as in Definition \ref{dfn:random fieldnot}. Using \eqref{eq:gL-redfn} it is immediate that if $\cX_{|\gd|} \in \gL_{{\bm \ell}, |\gd|}$ then 
\[
(\gG_N^+(\cX_{|\gd|}))_{i,1}= (\gH_N^+(\cX_{|\gd|}))_{i,1}=i, \quad \text{ for } i \in [d_2].
\]
To show that $\gG_N^+(\cX_{|\gd|}) \in \gL_1(\gd)$ and $\gH_N^+(\cX_{|\gd|}) \in \gL_2(\gd)$ we need to prove that they are weakly increasing in every row, and strictly increasing in every column and along the southwest diagonals. As $\cX_{|\gd|} \in L_{{\bm \ell}, \gd}$, upon recalling the definition of $L_{{\bm \ell}, |\gd|}$ from \eqref{eq:L-redfn} these are also immediate. Now we check \eqref{eq:sign-c}. Recalling the definitions of $\{\gc_i(\gx,\gy)\}_{i=1}^d$ from Definition \ref{dfn:random fieldnot} we find that for $i \in [d_0]$
\begin{align*}
\gc_i(\gG_N^+(\cX_{|\gd|}, \gH_N^+(\cX_{|\gd|})& = \sum_{j=1}^{d-d_0} \left( (\gG_N^+(\cX_{|\gd|}))_{j, i+1} - (\gG_N^+(\cX_{|\gd|}))_{j, i}+1\right) = \sum_{j=1}^{d-d_0} (x_{i, \gd+d_2-j+1}- x_{i+1, \gd+d_2-j+1} +1) \\
&=\sum_{j=1}^{\gd+d_2} (x_{i, j}- x_{i+1, j} +1) =\hat \ell_i,
\end{align*}
where we have used the fact that $\gd + d_2 = d_1 -d_0$. Similarly for $i \in [d]\setminus [d_0]$, recalling that
$d_0=d_1 -\gd$ we obtain
\begin{multline*}
\gc_i(\gG_N^+(\cX_{|\gd|}, \gH_N^+(\cX_{|\gd|}) = (\gH_N^+(\cX_{|\gd|}))_{1,d+1-i} + (\gG_N^+(\cX_{|\gd|}))_{1,i} -1 + \sum_{j=2}^{i-d_0} \left( (\gH_N^+(\cX_{|\gd|}))_{j,d+1-i} - (\gH_N^+(\cX_{|\gd|}))_{j-1,d+2-i}\right) \\
+ \sum_{j=2}^{d + 1 -i} \left( (\gG_N^+(\cX_{|\gd|}))_{j,i} - (\gG_N^+(\cX_{|\gd|}))_{j-1,i+1}\right) \corAB{-(\gd+d_2)}\\
= x_{i+1,1} + (N+d_2 - x_{i, \gd +d_2}) + \sum_{j=2}^{i-d_0} \left( x_{i+1,j} - x_{i,j-1}\right) + \sum_{j=2}^{d + 1 -i} \left( x_{i+1, \gd+d_2 -j +2} - x_{i, \gd+d_2- j+1}\right)\corAB{-(\gd+d_2)}\\
= x_{i+1,1} + \sum_{j=2}^{\gd+d_2}  \left( x_{i+1,j} - x_{i,j-1}\right) + (N+d_2 - x_{i, \gd +d_2}) \corAB{-(\gd+d_2)}= \hat \ell_i.
\end{multline*}
Now we proceed to show that $\gz(\gG_N^+(\cX_{|\gd|}), \gH^+(\cX_{|\gd|})) =\sgn(\sigma_{\mathbb{X}}) \sgn(\sigma_{\mathbb{Y}})$. To this end, recalling Definition \ref{dfn:random fieldnot} again, from \eqref{eq:redef-bX}, we find that
\begin{equation}\label{eq:redef-bX-1}
\widehat \gX(\gG_N^+(\cX_{|\gd|}), \gH_N^+(\cX_{|\gd|})) = N+d_2 - \mathbb X \qquad \text{ and } \qquad \widehat \gY(\gG_N^+(\cX_{|\gd|}), \gH_N^+(\cX_{|\gd|})) = \mathbb Y  +d_2.
\end{equation}
Since $\cX_{|\gd|} \in L_{{\bm \ell}, |\gd|}$ (recall \eqref{eq:L-redfn}), we find that for $j \in [\gd]$,
\[
N+d_2 - x_{1,j} \le N+d_2 -  x_{d_0+1,j} =  \sum_{i=d_0+1}^{d-j} (x_{i+1,j+d_0+2-i} - x_{i,j+d_0+1-i}) + (N+d_2 - x_{,\gd+d_2}) \le d L_2,
\]
where the last step follows under the assumption that $\max \hat \ell_i \le L$. Thus {$\mathbb X \subset [N]\setminus [N - d L_2]$.} Therefore, while computing $\sgn(\sigma_{\mathbb X})$ one can view $\mathbb X$ as a subset of 
$[N] \setminus [N- dL_2]$ for all large $N$, which in particular shows that $\sgn(\sigma_{\mathbb X})$ is free of $N$. Therefore, together with \eqref{eq:redef-bX-1} we derive that 
$\sgn(\sigma_{\mathbb X}) = \widehat \sgn (\widehat \gX(\gG_N^+(\cX_{|\gd|}), \gH_N^+(\cX_{|\gd|})))$ for all large $N$. A similar argument shows that $\sgn(\sigma_{\mathbb Y}) = \widehat \sgn (\widehat \gY(\gG_N^+(\cX_{|\gd|}), \gH_N^+(\cX_{|\gd|})))$ for all large $N$. Hence the map has all the desired properties.

Finally to complete the proof we further note that
 the map $\cX_{|\gd|} \mapsto (\gx, \gy)$  is a non-singular linear transformation
and therefore a bijection. Therefore, we deduce that the joint law of the random variables 
\begin{equation}\label{eq:jt-law-1}
\left\{\prod_{i=1}^{d_1-\gd} \lambda_i(z)^{-\hat \ell_i} \cdot \prod_{i=d_1-\gd+1}^{d} \lambda_i(z)^{\hat \ell_i} 
\cdot \sum_{\cX_{|\gd|} \in \gL_{{\bm \ell},|\gd|}} (-1)^{\sgn(\sigma_{\mathbb{X}}) \sgn(\sigma_{\mathbb{Y}})}\det(E_N[\mathbb{X}; \mathbb{Y}])\right\}_{\max \hat \ell_i \le L_2, z \in \cS_\gd}
\end{equation}
is equal to that of 
\begin{equation}\label{eq:jt-law-2}
 \left\{ \gc(\gx, \gy) \cdot (-1)^{\gz(\gx,\gy)}\det(E_\infty[\widehat \gX; \widehat \gY])\right\}_{ \gx \in \gL_1(\gd), \gy \in \gL_2(\gd), \max_i \gc_i(\gx, \gy) \le L_2, z \in \cS_\gd}. 
 \end{equation}
The proof for the case $\gd <0$ {is similar} and hence omitted. 

To finish the proof we now note that the bivariate random field $\{(|\widehat P|_\gd^{L_1}(z), |\widehat P|_\gd^{L_2}(z))\}_{z \in \cS_\gd}$ is some function of the random variables in \eqref{eq:jt-law-1}, and the bivariate random field $\{(|\gP|_\gd^{L_1}(z), |\gP|_\gd^{L_2}(z))\}_{z \in \cS_\gd}$ is the same function of the random variables in \eqref{eq:jt-law-2}. Thus these two bivariate random fields are indeed equal in distribution. 
This completes the proof of the lemma.
\end{proof}

\section{Anti-concentration bounds}\label{sec:anti-concentration}
In Section \ref{sec:limit-tight} we have already identified the limiting random field and derived its tightness. We recall from Section \ref{sec:prelim} that to prove Theorem \ref{thm:outlier-law} we need to establish that the limiting random field is not identically zero on a set of probability one. To do this we will require the anti-concentration property of the entries of the noise matrix as given by Assumption \ref{ass:entry}. The bound on the L\'{e}vy concentration function on the entries of the noise matrix will be used to derive an appropriate anti-concentration bound on the dominant term in the expansion of $\widehat \det_N(z)$ (see \eqref{eq:hat-det}) which will be later utilized to obtain the desired conclusion for the limiting random field.

We begin by providing the following general
anti-concentration bound for polynomials of independent real or complex-valued random variables, satisfying a bound on their L\'{e}vy concentration function given by Assumption \ref{ass:entry}, such that the degree of every variable is at most one. 

\begin{proposition}\label{prop:anti-conc-complex}
Fix $k, n \in \N$ and let $\{U_i\}_{i=1}^n$ be a sequence of
independent real or complex-valued random variables,
whose L\'{e}vy concentration functions satisfy the bound \eqref{eq:levy-bd}. Let $Q_k(U_1,U_2,\ldots,U_n)$ be a 
homogenous polynomial of degree $k$ such that the degree of 
each variable is at most one. That is,
\[
Q_k(U_1,U_2,\ldots,U_n) := \sum_{\cI \in \binom{[n]}{k}} b(\cI) \prod_{i \in \cI} U_{i},
\]
for some collection of complex-valued coefficients $\{b(\cI); \, \cI \in \binom{[n]}{k}\}$, where $\binom{[n]}{k}$ denotes the set of all $k$ distinct elements of $[n]$.

Assume that there exists an 
$\cI_0 \in \binom{[n]}{k}$ such that $|b(\cI_0)| \ge c_\star$ for some absolute constant $c_\star>0$. Then for any $\vep \in (0, {e^{-1}}]$ we have
\[
\P\left( |Q_k(U_1,U_2,\ldots,U_n)| \le \vep\right) \le \bar C \cdot (c_\star \wedge 1)^{-\upeta}\vep^{\upeta} \left(\log\left(\frac{1}{\vep}\right)\right)^{k-1},
\]
where $\upeta \in (0, {2}]$ is as in \eqref{eq:levy-bd} and $\bar C < \infty$ is some large absolute constant.
\end{proposition}

When $\{U_i\}_{i=1}^n$ are independent real valued random variables and have uniformly bounded densities with respect to the Lebesgue measure, an anti-concentration bound analogous to the above was obtained in \cite{BPZ-non-triang} (see Proposition 4.5 there), with $\upeta=1$. The proof of Proposition \ref{prop:anti-conc-complex} follows from a simple modification of the proof of \cite[Proposition 4.5]{BPZ-non-triang}. We include it for completeness. 

\begin{proof}
Since
\[
\cL(U_i, \vep) = \sup_{w \in \C} \P (|U_i -w| \le \vep) \le C \vep^{\upeta}, \qquad i \in [n],
\]
where $C$ and $\upeta$ are as in \eqref{eq:levy-bd}, using the joint independence of $\{U_i\}_{i \in [n]}$ the desired anti-concentration property is immediate for $k=1$. 
To prove the general case we proceed by induction. 

The idea behind the proof is that \corAB{$Q_k$ being a polynomial such that} the degree of each $U_i$ is at most one, we note that for $i_0 \in \cI_0$ one can write $Q_k= \mathcal{Q} \cdot U_{i_0} + \widetilde{\mathcal{Q}}$, for some $\mathcal{Q}, \widetilde{\mathcal{Q}}$ independent of $U_{i_0}$. Thus, the anti-concentration bound of $Q_k$ depends on \corAB{that of} $\mathcal{Q}$. \corAB{The advantage of this decomposition} is that the degree of $\mathcal{Q}$ is $(k-1)$. So one can iterate the above argument to obtain the desired anti-concentration bound for $Q_k$.

To formulate this idea we introduce 
some notation. Order the elements of $\cI_0$
and denote them by $i_1^0, i_2^0, \ldots, i_k^0$. For $j \le k$,
define $\cI^0_j:=\{i_j^0, i_{j+1}^0, \ldots, i_k^0\}$. Set 
\[
Q_k^0:=Q_k^0(U_i; i \notin \cI_k^0):=  \sum_{\cI: \cI \supset \cI_k^0} b(\cI) \prod_{\ell \in \cI \setminus \cI_k^0} U_{\ell}  \quad \text{ and } \quad 
Q_k^1:=Q_k^1(U_i; i \notin \cI_k^0):=\quad \sum_{\cI: \cI\cap \cI_k^0= \emptyset} b(\cI) \prod_{\ell \in \cI} U_{\ell}.
\]
For $1 \le j \le k-1$, we iteratively define 
\[
Q_j^0:=Q_j^0(U_i; i \notin \cI_j^0):=  \sum_{\cI: \cI \supset \cI_j^0} b(\cI) \prod_{\ell \in \cI \setminus \cI_j^0} U_{\ell}  \quad \text{ and } \quad 
Q_j^1:=Q_j^1(U_i, i \notin \cI_j^0):=\quad \sum_{\substack{\cI: \cI \supset \cI_{j+1}^0\\  i_j^0\notin \cI }} b(\cI) \prod_{\ell \in \cI \setminus \cI_{j+1}^0}U_{\ell}.
\]
Equipped with the above notation we see that
\[
Q_k(U_1,U_2, \ldots, U_n) =:Q_{k+1}^0 = U_{i^0_k} \cdot Q_k^0 + Q_k^1, \qquad 
Q_{j+1}^0 = U_{i_{j}^0} \cdot Q_{j}^0 + Q_{j}^1, \ j=1,2,\ldots, k-1,
\]
and
\(
Q_1^0= a(\cI_0) 
\). We will prove inductively that
\begin{equation}\label{eq:anti-conc-induction}
\P\left( |Q_j^0| \le \vep \right) \le {(3C e^{\upeta})}^{j-1} (c_\star \wedge 1)^{-\upeta}\vep^{\upeta} \left(\log\left(\frac{1}{\vep}\right)\right)^{j-2}, \quad j=2,3,\ldots, k+1,
\end{equation}
from which the desired anti-concentration bound follows by taking $j=k+1$. 
Hence, it only remains to prove \eqref{eq:anti-conc-induction}. 

For $j=2$, $Q_j^0$ is a homogeneous
polynomial of degree $1$ in the variables $U_i$, and
\eqref{eq:anti-conc-induction}  follows from the assumptions 
on $\{U_\ell\}_{\ell=1}^n$ and the fact that $|b(\cI_0)| \ge c_\star$. 
Assuming that \eqref{eq:anti-conc-induction} holds for $j=j_*$ and fixing $\delta \in (0,1)$, 
we have that with  $C_j:= {(3C e^{\upeta})}^{j-1} (c_\star \wedge 1)^{-\upeta}$,
\begin{align}\label{eq:anti-conc-split}
\P\left( \left|Q_{j_*+1}^0 \right| \le \vep\right)  & \le \P \left( \left| Q_{j_*}^0\right| \le \delta \right) + \E\left[ \P\left(  \left| U_{i_{j_*}^0} + \frac{Q_{j_*}^1 }{Q_{j_*}^0 }\right| \le \frac{\vep}{|Q_{j_*}^0 |} \bigg| \, U_i, i \notin \cI_{j_*}^0\right)\cdot {\bm 1} \left( \left|Q_{j_*}^0\right| \ge \delta\right) \right] \notag\\
& \le C_{j_*} \delta^{\upeta} \left(\log \left( \frac{1}{\delta}\right)\right)^{j_*-2} + C {\vep^{\upeta}} \cdot \E \left[ |Q_{j_*}^0|^{-\upeta} {\bm 1} \left( \left|Q_{j_*}^0\right| \ge \delta\right) \right],
\end{align}
where we have used the fact that $Q_{j_*}^1$ and $Q_{j_*}^0$ are independent of $U_{i_{j_*}}^0$, and the bound on the L\'{e}vy concentration function (i.e.~the bound \eqref{eq:levy-bd}) for the latter.
Using integration by parts, for any probability measure $\mu$ supported on $[0,\infty)$ we have that 
\[
\int_\delta^{{e^{-1}}} x^{-\upeta}  d\mu(x) = {e^{\upeta}} \mu([\delta, 1])  + \upeta \int_\delta^{{e^{-1}}} \frac{\mu([\delta,t])}{t^{1+\upeta}} dt.
\]
Therefore, using the induction hypothesis and the fact that $\upeta \in (0,2]$, we have
\begin{multline*}
\E \left[ |Q_{j_*}^0|^{-\upeta} {\bm 1} \left( \left|Q_{j_*}^0\right| \ge \delta\right) \right] \le {e^{\upeta}}+ \E \left[ |Q_{j_*}^0|^{-\upeta} {\bm 1} \left( \left|Q_{j_*}^0\right| \in [\delta,{e^{-1}}]\right) \right]  \\ \le 2{e^{\upeta}} + 2 \int_\delta^{{e^{-1}}} \frac{ \P(|Q_{j_*}^0| \le t)}{t^{1+\eta}} dt  \le 2 {e^{\upeta}}+  2 C_{j_*} \int_\delta^{{e^{-1}}} t^{-1} \left(\log\left(\frac{1}{t}\right)\right)^{j_*-2} dt \\
 \le 2 {e}^{\upeta}+ \frac{2 C_{j_*}}{j_*-1} \left(\log\left(\frac{1}{\delta}\right)\right)^{j_*-1}.
\end{multline*}
{Since for $\delta \le e^{-1}$ we have that $\log(1/\delta) \ge 1$}, combining the above with \eqref{eq:anti-conc-split} and setting $\delta=\vep$ we establish \eqref{eq:anti-conc-induction} for $j=j_*+1$. This completes the proof.
\end{proof}
Using Proposition \ref{prop:anti-conc-complex}, we now derive the following corollary which yields an appropriate lower bound on the dominant term.

\begin{corollary}\label{cor:anti-conc-complex-dom}
Fix $\vep>0$ and $\gd \ne 0$ an integer such that $-d_2 \le \gd \le d_1$. Let the entries of $E_N$ satisfy the bound \eqref{eq:levy-bd}. Then there exists a constant $C_\star:=C_\star(\vep,{\bm a})$  so that,
for any $z \in \cS_\gd^{-\vep}$ and $\vep_0 \in (0,e^{-1}]$,
\[
\P\left({|\widehat P_{|\gd|}(z)|}\le \vep_0 \right) \le C_\star  \vep_0^{\upeta} \left(\log\left(\frac{1}{\vep_0}\right)\right)^{|\gd|-1},
\]
where $\{-\lambda_\ell(z)\}_{\ell=1}^d$ are the roots of the equation $P_{z,{\bm a}}(\lambda)=\lambda^{d_2} ({\bm a}(\lambda) -z)=0$ arranged in the non-increasing order of their moduli.
\end{corollary}

To prove Corollary \ref{cor:anti-conc-complex-dom} we will need the following
lemma. Its proof is deferred to Section \ref{sec:proof-main-1}.

\begin{lemma}\label{lem:toep-det-gd-bd}
Fix $\vep>0$ and $\gd\neq 0$ an integer such that $-d_2 \le \gd \le d_1$. For $\gd > 0$ set $X_\star:= [N]\setminus [N-\gd]$ and
$Y_\star:=[\gd]$. For $\gd <0$ set $X_\star:=[-\gd]$ and $Y_\star:=[N]\setminus [N+\gd]$. Then, there exists a  constant $c_0':=c_0'(\vep,{\bm a})>0$ so that
\[
\inf_{z \in \cS_\gd^{-\vep}}\left|\frac{\det(T_N({\bm a}(z))[X_\star^c, Y_\star^c])}{a_{d_1}^{N-|\gd|}\prod_{\ell=1}^{d_1 -\gd} \lambda_\ell(z)^{N+d_2}}\right| \ge c_0', \qquad \text{ for all large } N,
\]
 where $\{\lambda_\ell(z)\}_{\ell=1}^d$ are as in Corollary \ref{cor:anti-conc-complex-dom}. 
\end{lemma}

\begin{proof}[Proof of Corollary \ref{cor:anti-conc-complex-dom}
 {(assuming Lemma \ref{lem:toep-det-gd-bd})}]
  We remind the reader that
\begin{equation}\label{eq:widehat-P}
\widehat P_{|\gd|}(z):= \frac{P_{|\gd|}(z)}{N^{-\gamma |\gd|} a_{d_1}^{N} \prod_{i=1}^{d_1-\gd} \lambda_i(z)^{N+d_2}},
\end{equation}
{see \eqref{fancyK} and \eqref{eq:hatP-k}.}
Recalling \eqref{eq:P_k-z} we note that $\widehat P_{|\gd|}(z)$ is a homogeneous polynomial of degree $|\gd|$ in the entries of the noise matrix $E_N$ such that the degree of each entry is one. By  Lemma \ref{lem:toep-det-gd-bd},
there exist $X, Y \subset [N]$ with $|X|=|Y|=|\gd|$ such that 
$\det (T_N(z)[X^c;Y^c])$ is uniformly bounded below for 
$z \in \cS_\gd^{-\vep}$.
Thus, using \eqref{eq:P_k-z} again,
we may apply 
Proposition \ref{prop:anti-conc-complex} to deduce that
 \[
 \P\left( \left|\widehat P_{|\gd|}(z)\right| \le \vep_0 \right) \le \bar C (c_0')^{-\upeta} \vep_0^{\upeta} \left(\log\left(\frac{1}{\vep_0}\right)\right)^{|\gd|-1}.
 \]
 This completes the proof. 
\end{proof}

\section{Proof of Theorem \ref{thm:outlier-law}}\label{sec:dom-term}
Using results from Sections \ref{sec:identify-non-dom-term}-\ref{sec:anti-concentration} in this section we finally complete the proof of our main result Theorem \ref{thm:outlier-law}. We begin with the proof of Lemma \ref{lem:limit}. Turning to do this task, in the result below we derive analyticity of the random fields $\{\gP_\gd^L(\cdot)\}$.


\begin{lemma}\label{lem:analyticity}
Fix $L \in \N$, $\gd \ne 0$ such that $-d_2 \ge \gd \le d_1$, and $\vep >0$. Then the maps $z \mapsto \widehat P_{|\gd|}^L(z)$ and $z \mapsto \gP_{\gd}^L(z)$ are analytic on $\cS_\gd^{-\vep}$.
\end{lemma}

As a first step we will argue that the map $z \mapsto \widehat P_{|\gd|}^L(z)$ is continuous and then use Riemann's removable singularity lemma to derive its analyticity. To this end, we have the following lemma.

\begin{lemma}\label{lem:mathfrak-f}
Let $\vep >0$ and $\gd$ be an integer such that $-d_2 \le \gd \le d_1$. Let $\mathfrak{f} : \C^d \mapsto \C$ be a continuous map such that 
\begin{equation}\label{eq:mathfrak-f}
\mathfrak{f}(\lambda_1,\lambda_2,\ldots,\lambda_d)=\mathfrak{f}(\lambda_{\pi(1)}, \lambda_{\pi(2)},\ldots,\lambda_{\pi(d)}),
\end{equation}
for all permutations $\pi$ on $[d]$ for which $\pi([d_1-\gd])=[d_1-\gd]$. Then the map $z \mapsto \mathfrak{f}(\lambda_1(z),\lambda_2(z),\ldots,\lambda_d(z))$ is continuous on $\cS_\gd^{-\vep}$.  
\end{lemma}

\begin{proof}
To prove the lemma we need to use continuity properties of the roots of the equation $P_{z,{\bm a}}(\lambda)=0$. This requires some notation. Let $\C_{\rm sym}^d$, the symmetric $d$-th power of $\C$, denote the set of equivalent classes in $\C^d$, where two points in $\C^d$ are set to be equivalent if one can be obtained by permuting the coordinates of the other. Given any two points $\langle \lambda_1,\lambda_2,\ldots, \lambda_d\rangle, \langle \mu_1,\mu_2,\ldots, \mu_d\rangle \in \C^d_{\rm sym}$  we define
\[
{\rm dist}(\langle \lambda_1,\lambda_2,\ldots, \lambda_d\rangle, \langle \mu_1,\mu_2,\ldots, \mu_d\rangle):= \inf_\pi \sup_\ell |\lambda_\ell - \mu_{\pi(\ell)}|,
\]
where the infimum is taken over all permutations $\pi$ of $[d]$. This induces a metric on $\C_{\rm sym}^d$. 

Let $\uptau: \C \mapsto \C_{\rm sym}^d$ be the map given by $\uptau(z)= \langle \lambda_1(z),\lambda_2(z),\ldots, \lambda_d(z)\rangle$, where $\{-\lambda_i(z)\}_{i=1}^d$ are the roots of the equation $P_{z,{\bm a}}(\lambda)=0$. It is well known that the map $\uptau(\cdot)$ is continuous (see \cite[Appendix 5, Theorem 4A]{whitney}).   

Using this we now establish the continuity of the function $\mathfrak{f}(\lambda_1(z),\ldots,\lambda_d(z))$.   
Consider any sequence $\{z_n\}$ such that $z_n \to z \in \cS_\gd^{-\vep}$, as $n \to \infty$. Let $\pi_n^\star$ be the permutation such that 
\[
{\rm dist}(\uptau(z), \uptau(z_n)) = \sup_{\ell} |\lambda_\ell(z) - \lambda_{\pi_n^\star(\ell)}(z_n)|.
\]
We claim that $\pi_n^\star([d_1-\gd]) =[d_1-\gd]$ for all large $n$, i.e. $\pi_n^\star$ maps $[d_1-\gd]$ to $[d_1-\gd]$. If not, then there exists $\ell_n \in [d_1-\gd]$ and $\ell_n' \in [d]\setminus[d_1-\gd]$ such that $\pi_n^\star(\ell_n)=\ell'_n$. On the other hand $z \in \cS_\gd^{-\vep}$ implies that $z_n \in \cS_{\gd}^{-\vep/2}$, for all large $n$. Therefore, the last two observations together with Lemma \ref{lem:root-cont-0} imply that
\begin{equation}\label{eq:dist-sym-lbd}
{\rm dist}(\uptau(z), \uptau(z_n)) \ge  |\lambda_{\ell_n}(z) - \lambda_{\ell_n'}(z_n)| \ge   |\lambda_{\ell_n}(z)| - |\lambda_{\ell_n'}(z_n)|  \ge  |\lambda_{d_1-\gd}(z)| - |\lambda_{d_1-\gd+1}(z_n)| \ge \vep_0',
\end{equation}
for some $\vep_0'>0$. Since $\uptau(z_n) \to \uptau(z)$ the inequality \eqref{eq:dist-sym-lbd} yields a contradiction. Hence, $\pi_n^\star([d_1-\gd]) =[d_1-\gd]$ for all large $n$, as claimed above. As $\mathfrak{f}(\cdot)$ satisfies \eqref{eq:mathfrak-f} this further implies that
\[
\mathfrak{f}(\lambda_1(z_n),\lambda_2(z_n),\ldots,\lambda_d(z_n)) = \mathfrak{f}(\lambda_{\pi_n^\star(1)}(z_n),\lambda_{\pi_n^\star(2)}(z),\ldots,\lambda_{\pi_n^\star(d)}(z)),
\]
for all large $n$. Since ${\rm dist}(\uptau(z_n),\uptau(z))\to 0$, as $n \to \infty$, the desired continuity of the map $z \mapsto \mathfrak{f}(\lambda_1(z),\ldots,\lambda_d(z))$ is immediate. This completes the proof of the lemma. 
\end{proof}

Building on Lemma \ref{lem:mathfrak-f} we now prove that the maps $z \mapsto \widehat P_{|\gd|}^L(z)$ and $z \mapsto  \gP_{\gd}^L(z)$ are analytic on $\cS_\gd^{-\vep}$.

\begin{proof}[Proof of Lemma \ref{lem:analyticity}]
Denote
\[
P_{|\gd|}^L(z) := \widehat{P}_{|\gd|}^L(z) \cdot a_{d_1}^{N-|\gd|} \cdot N^{-\gamma |\gd|} \cdot \prod_{i=1}^{d_1-\gd} \lambda_i(z)^{N+d_2}.
\]
Recalling \eqref{eq:det-decompose-1}, \eqref{eq:det-prod-decompose}, and  \eqref{eq:P-k-decompose}, and the definition of  $\hat \ell_i$ from \eqref{eq:hat-ell} we therefore note that $P_{|\gd|}^L(z)$ is the sum of the sets $(X_1,X_2,\ldots, X_{d+1})$ such that $\ell_i \ge N+d_2 - L$ for all $i =1,2,\ldots, d_1 -\gd$, and $\ell_j \le L$ for all $j=d_1 -\gd+1, \ldots, d$. 
Since for any permutation $\pi$ on $[d]$ 
\[
\prod_{i=1}^d (J_{N+d_2} + \lambda_i(z) \Id_{N+d_2}) = \prod_{i=1}^d (J_{N+d_2} + \lambda_{\pi_i}(z) )\Id_{N+d_2}),
\]
the representation \eqref{eq:det-decompose-1} of $P_{|\gd|}(z)$ further implies that $P_{|\gd|}^L(z)$ is invariant under any permutation $\pi$ on $[d]$ for which $\pi([d_1-\gd])=[d_1-\gd]$. Hence, by Lemma \ref{lem:mathfrak-f} the map $z \mapsto  P_{|\gd|}^L(z)$ is continuous on $\cS_\gd^{-\vep}$. 
The same lemma shows that the map $z \mapsto \prod_{i=1}^{d_1-\gd} \lambda_i(z)^N$ is continuous on $\cS_\gd^{-\vep}$. Hence, so is the map $z \mapsto \widehat P_{|\gd|}^L(z)$.

Next to show the analyticity of $\widehat P_{|\gd|}^L(\cdot)$ we apply Riemann's removable singularity theorem. For that it needs to be shown that except on a collection of isolated points $\widehat P_{|\gd|}^L(\cdot)$ is a holomorphic function. 

Recall \eqref{eq:dom-L-redfn}:
\[
\widehat P_{|\gd|}^L(z) =  \sum_{{\bm \ell}: \max_i \hat \ell_i \le L} \prod_{i=1}^{d_0}  \lambda_i(z)^{-\hat \ell_i} \cdot \prod_{i=d_0+1}^{d}  \lambda_i(z)^{\hat \ell_i} 
\cdot \sum_{\cX_{|\gd|} \in \gL_{{\bm \ell},|\gd|}} (-1)^{\sgn(\sigma_{\mathbb{X}}) \sgn(\sigma_{\mathbb{Y}})}\det(E_N[\mathbb{X}; \mathbb{Y}]).
\]
Let $\bar {\mathcal N}$ is the collection of $z$’s for which $P_{z,{\bm a}}(\cdot)=0$ has double roots. By \cite[Lemma 11.4]{bottcher-finite-band}, the cardinality of $\bar {\mathcal N}$ is finite, and thus all its elements are isolated. Using the implicit function theorem it follows that for $z \in \cS_\gd^{-\vep}\setminus \bar\cN$ the roots of $P_{z,{\bm a}}(\cdot)=0$ are analytic in $z$ (for a proof the reader is referred to \cite{brillinger}). Therefore there exists a reordering of the indices of the roots $\{\lambda_i(z)\}_{i=1}^d$, denoted hereafter by $\{\widehat \lambda_i(z)\}_{i=1}^d$, such that the maps $z \mapsto \widehat \lambda_i(z)$ are holomorphic on $\cS_\gd^{-\vep}\setminus \bar\cN$. From its definition it further follows that, for all $z \in \cS_\gd^{-\vep}$, among $\{\widehat \lambda_i(z)\}_{i=1}^d$ there are exactly $(d_1-\gd)$ roots that are strictly greater than one in moduli. So reusing the fact that $\widehat P_{|\gd|}^L(z)$ is invariant under any permutation of $\{ \lambda_i(z)\}_{i=1}^{d_1-\gd}$ and any permutation of the rest of the $\lambda_i(z)$'s, without loss of generality we may write 
\begin{equation}\label{eq:widehat-P-holom}
\widehat P_{|\gd|}(z) =  \sum_{{\bm \ell}: \max_i \hat \ell_i \le L} \prod_{i=1}^{d_0} \widehat \lambda_i(z)^{-\hat \ell_i} \cdot \prod_{i=d_0+1}^{d} \widehat \lambda_i(z)^{\hat \ell_i} 
\cdot \sum_{\cX_{|\gd|} \in \gL_{{\bm \ell},|\gd|}} (-1)^{\sgn(\sigma_{\mathbb{X}}) \sgn(\sigma_{\mathbb{Y}})}\det(E_N[\mathbb{X}; \mathbb{Y}]), \quad z \in \cS_\gd^{-\vep}\setminus \bar\cN.
\end{equation}
This indeed shows that $\widehat P_{|\gd|}^L(z)$ is a holomorphic function on $\cS_\gd^{-\vep}\setminus \bar\cN$. To apply Riemann's removable singularity theorem we need to show that it is bounded in a neighborhood of $\bar\cN$. This is immediate, as from the definition of the polynomial $P_{z,{\bm a}}(\lambda)=0$ we have that for any $R<\infty$,
\[
\sup_{z \in B_\C(0,R)} \max_{i=1}^d |\lambda_i(z)| =O(1).
\]
This yields that $\widehat P_{|\gd|}^L(z)$ is analytically extendable to the whole of $\cS_\gd^{-\vep}$. Finally, since  the function $\widehat P_{|\gd|}^L(z)$ is
 continuous at $z \in \bar\cN \cap \cS_\gd^{-\vep}$, we conclude that   its analytic extension to $\cS_\gd^{-\vep}$ is the  function itself. That is, the equality \eqref{eq:widehat-P-holom} continues to hold for $z \in \bar\cN \cap \cS_\gd^{-\vep}$, where by a slight abuse of notation we use $\{-\widehat \lambda_i(z)\}_{i=1}^d$ to denote the analytic extensions of the analytic parametrization of the roots of $P_{z,{\bm a}}(\lambda)=0$. Thus, the map $z \mapsto \widehat P_{|\gd|}^L(z)$ is indeed an analytic function. 
 
 Turning to prove the analyticity of the map $z \mapsto \gP_\gd^L(z)$ we recall that the proof of Lemma \ref{lem:zeta-N-gd-L} shows that the maps and $\widehat P_{|\gd|}^L(z)$ and $\gP_\gd^L(z)$, when viewed as functions of $z$ are the same map, albeit the entries of $E_N$ gets replaced by that of $E_\infty$ by the affine function defined there. This shows that $\{\gP_\gd^L\}_{L \in \N}$ are random analytic functions as well, thereby completing the proof of this lemma. 
 \end{proof}


Equipped with the
necessary ingredients, we now proceed to the proof of Lemma \ref{lem:limit}. 

\begin{proof}[Proof of Lemma \ref{lem:limit}]
We begin by noting that part (i) is a consequence of Lemma \ref{lem:analyticity}, whereas part (ii) is immediate from Corollary \ref{cor:random-field}(iii). 

So it only remains to prove that $\gP^\infty_\gd(\cdot)$ is not identically zero on a set of probability one. Without loss of generality we may assume that $\cS_\gd$ is non-empty. Fix any $\vep >0$ and pick any $z_\star \in \cS_\gd^{-\vep}$. Note that, as $\cS_\gd^{-\vep} \uparrow \cS_\gd \ne \emptyset$, as $\vep \downarrow 0$, for $\vep$ sufficiently small the sets $\cS_\gd^{-\vep}$ are non-empty, and hence a choice of $z_\star$ is feasible for any sufficiently small $\vep$. Fix this choice of $\vep$ for the remainder of the proof. 

Now fix any $\delta >0$. From Corollaries \ref{cor:dominant-tail} and \ref{cor:random-field}(iii), upon using Markov's inequality it follows that there exists an $L$ such that 
\[
\max\left\{\P\left(  |\widehat P_{|\gd|}^{\wc{ L}}(z_\star)| \ge \delta \right),  \P\left(  | \gP_{\gd}^{\wc{ L}}(z_\star)| \ge \delta \right)\right\} \le \delta,
\]
where for brevity we set
\[
\gP_\gd^{\wc{ L}}(\cdot):= \gP_\gd^\infty(\cdot) - \gP_\gd^L(\cdot).
\]
Thus, using the triangle inequality we find that
\begin{align*}
\P(|\gP_\gd^\infty(z_\star)| \le \delta) \le \P(|\gP_\gd^L(z_\star)| \le \delta +|\gP_\gd^{\wc{ L}}(z_\star)| ) & \le \P(|\gP_\gd^L(z_\star)| \le 2\delta) + \delta \\
 & =  \P(|\widehat{P}_{|\gd|}^L(z_\star)| \le 2\delta) + \delta  \le \P(|\widehat{P}_{|\gd|}(z_\star)| \le 3\delta) + 2\delta,
\end{align*} 
where the equality above follows from Lemma \ref{lem:zeta-N-gd-L}. Now, by Corollary \ref{cor:anti-conc-complex-dom}, we further deduce from above that
\[
\P(|\gP_\gd^\infty(z_\star)| \le \delta) \le C_\star (3\delta)^\upeta \left(\log \left(\frac{1}{3\delta}\right)\right)^{|\gd|-1}+ 2\delta. 
\]
As $\delta >0$ is arbitrary sending $\delta \downarrow 0$ we derive that 
\begin{equation}\label{eq:limit-non-zero}
\P(\gP_\gd^\infty(\cdot) \equiv 0 \text{ on } \cS_\gd) \le \P\left(\sup_{z \in \cS_\gd^{-\vep}} |\gP_\gd^\infty(z)|  =0\right) \le \limsup_{\delta \downarrow 0} \P(|\gP_\gd^\infty(z_\star)| \le \delta) = 0. \notag
\end{equation}
The proof of the lemma is now complete. 
\end{proof}

\begin{proof}[Proof of Theorem \ref{thm:outlier-law}]
We will use the following proposition from \cite{Sh}.
\begin{proposition}\cite[Proposition 2.3]{Sh}
  \label{prop-shirai}
  Suppose that a sequence of random analytic functions $\{X_N\}$
converges in law to (an analytic random function) X. Then, the zero process of $X_N$  converges in law to 
the zero process of $X$  provided that
$X\not\equiv 0$ almost surely.
\end{proposition}
In our case, we will use  
$\widehat \det_N(z)$, see
\eqref{eq:hat-det}, for $X_N$, and $\gP_\gd^\infty(z)$, see Definition
\ref{dfn:random field}
for $X$. (Note that by Lemma \ref{lem:limit}, 
the latter is almost surely analytic.)
Thus, the proof of Theorem \ref{thm:outlier-law} boils down to checking the conditions of Proposition \ref{prop-shirai}, that is to checking the following.\\
(i) $\widehat \det_N(\cdot)$ converges in law to $\gP_\gd^\infty(\cdot)$ in $\cS_\gd^{-\vep}$ .\\
(ii) $\gP_\gd^\infty(\cdot)\not \equiv 0$ in $\cS_\gd^{-\vep}$.

To see (i), we recall the random functions
$\widehat P_{|\gd|}(z)$ and 
$\widehat P_{|\gd|}^L(z)$,
see \eqref{eq:hatP-k} and \eqref{eq:dominant-L}.
By Corollary \ref{cor:non-dominant}, it is enough to prove (i) with
$\widehat \det_N(z)$ replaced by $\widehat P_{|\gd|}(z)$. 
By \eqref{eq:dominant-L-bar} and
Corollary \ref{cor:dominant-tail}, for (i) it is then enough to
prove that the law of $\widehat P_{|\gd|}^L(\cdot)$ converges, 
as first $N\to\infty$ and then $L\to\infty$, to the law of 
$\gP_\gd^\infty(\cdot)$. By Lemma  \ref{lem:zeta-N-gd-L}(iii),
the law of $\widehat P_{|\gd|}^L(\cdot)$ 
coincides, for $N$ large,  with that of $\gP_\gd^L(\cdot)$, 
and the latter law is independent
of $N$. One now concludes (i) by noting that
by Lemma \ref{lem:limit}(ii),  $\gP_\gd^L(\cdot)$
converges uniformly in $\cS_\gd^{-\vep}$, as $L\to\infty$, 
to $\gP_{\gd}^\infty(\cdot)$ 

The point (ii) is a consequence of part (iii) of Lemma \ref{lem:limit}.
This completes the proof of the theorem.
\end{proof}

\section{Proof of Theorem \ref{thm:no-outlier}}\label{sec:proof-main-1}

Next we proceed to the proof of Theorem \ref{thm:no-outlier}, i.e.~we aim to show that there are no outliers outside the spectrum of the Toeplitz operator $T({\bm a})$. In the set-up of Theorem \ref{thm:no-outlier},
Lemma \ref{lem:rouche-multi-gr-d_0} yields the desired upper bound on the non-dominant terms. In this set-up the dominant term is the non-random unperturbed Toeplitz matrix $T_N({\bm a}(z))$. Hence to complete the proof we need a uniform lower bound on the {latter}. 

\begin{lemma}\label{lem:toep-det-unif-lbd}
Let ${\bm a}$ be a Laurent polynomial given by 
\[
{\bm a}(\lambda):= \sum_{\ell=-d_2}^{d_1} a_\ell \lambda^\ell, \qquad \lambda \in \C,
\]
for some $d_1,d_2 \in \N$. Fix $\vep>0$. Then, there exists a positive constant $c_0>0$ such that
\[
\inf_{z \in \cS_0^{-\vep}\cap B_\C(0,N^{1/2})}\left|\frac{\det(T_N({\bm a}(z)))}{a_{d_1}^N\prod_{\ell=1}^{d_1} \lambda_\ell(z)^N}\right| \ge c_0, \qquad \text{ for all large } N,
\]
where $\{-\lambda_\ell(z)\}_{\ell=1}^d$ are the roots of the equation $P_{z,{\bm a}}(\lambda)=\lambda^{d_2} ({\bm a}(\lambda) -z)=0$ arranged in the non-increasing order of their moduli and $d=d_1+d_2$.
\end{lemma}

{We will later check, see \eqref{eq:no-outlier-22}, that
all eigenvalues of $T_N({\bm a})+\Delta_N$ are}
contained in $B_\C(0,N^{1/2})$ with high probability. Thus the uniform lower bound of Lemma \ref{lem:toep-det-unif-lbd} is sufficient to complete the proof of Theorem \ref{thm:no-outlier}. For any $z \in \C$ the expression for the determinant of $T_N({\bm a}(z))$ is well known:~it follows from Widom's formula (see \cite[Theorem 2.8]{bottcher-finite-band}) when the roots of roots of $P_{z,{\bm a}}(\cdot)=0$ are distinct, while in the other case one can use Trench's 
formula \cite[Theorem 2.10]{bottcher-finite-band}. As Lemma \ref{lem:toep-det-unif-lbd} requires a uniform bound on the determinant for $z \in \cS_0^{-\vep}\cap B_\C(0,N^{1/2})$, we refrain from using \cite[Theorems 2.8 and 2.10]{bottcher-finite-band} and instead we use the observation by Bump and Diaconis \cite{BD}, where they noted that 
irrespective of whether $P_{z,{\bm a}}(\cdot)=0$ has double roots or not, the determinant of a finitely banded Toeplitz determinant can be expressed as a certain Schur polynomial. 

Before proceeding to the proof of Lemma \ref{lem:toep-det-unif-lbd} we recall
the definition of the
Schur polynomials. Given any partition ${\bm \nu}:= (\nu_1,\nu_2,\ldots,\nu_d)$ with $\nu_1\ge \nu_2\ge \cdots \ge \nu_d \ge 0$ we define Schur polynomial $S_{{\bm \nu}}$ by
\begin{equation}\label{eq:schur}
S_{\bm \nu}(\lambda_1,\lambda_2, \ldots,\lambda_d):= \frac{\det V_{\bm \nu}(\lambda_1,\lambda_2, \ldots,\lambda_d)}{\det V_{\bm 0}(\lambda_1,\lambda_2, \ldots,\lambda_d)},
\end{equation}
where for any partition ${\bm \alpha}:=(\alpha_1,\alpha_2,\ldots,\alpha_d)$
\[
V_{\bm \alpha}(\lambda_1,\lambda_2, \ldots,\lambda_d) = \begin{bmatrix} \lambda_1^{\alpha_1+d-1} & \lambda_2^{\alpha_1+d-1} & \cdots & \lambda_d^{\alpha_1+d-1}\\ \lambda_1^{\alpha_2+d-2} & \lambda_2^{\alpha_2+d-2} & \cdots & \lambda_d^{\alpha_2+d-2}\\ \vdots & \vdots & \ddots & \vdots \\\lambda_1^{\alpha_d} & \lambda_2^{\alpha_d} & \cdots & \lambda_d^{\alpha_d} \end{bmatrix},
\] and ${\bm 0}:=(0,0,\ldots,0)$ denotes the zero partition. If $\{\lambda_\ell\}_{\ell=1}^d$ are not all distinct then both the numerator and the denominator of \eqref{eq:schur} are zero. In that case, the quotient needs to be evaluated using L'H\^opital's rule. Therefore the proof of Lemma \ref{lem:toep-det-unif-lbd} also splits into two parts: $z \notin \bar\cN$ and $z \in \bar\cN$, where we 
remind the reader that $\bar\cN$ is the collection of $z$'s for which $\{\lambda_\ell(z)\}_{\ell=1}^d$ are not all distinct and it is a set of finite cardinality. The first case is handled in the following lemma. 

\begin{lemma}\label{lem:toep-det-unif-lbd-1}
Under the same set-up as in Lemma \ref{lem:toep-det-unif-lbd},
{and in particular with the same $c_0$,}
\[
\inf_{z \in \cS_0^{-\vep}\cap B_\C(0,N^{1/2}) \setminus \bar\cN}\left|\frac{\det(T_N({\bm a}(z)))}{a_{d_1}^N\prod_{\ell=1}^{d_1} \lambda_\ell(z)^N}\right| \ge c_0, \qquad \text{ for all large } N.
\]
\end{lemma}

We use the following continuity properties to derive Lemma \ref{lem:toep-det-unif-lbd} from Lemma \ref{lem:toep-det-unif-lbd-1}.

\begin{lemma}\label{lem:root-cont}
Fix $\vep >0$. For any $z \in  \cS_0^{-\vep}$, the maps $z \mapsto \det(T_N({\bm a}(z)))$ and $z \mapsto \prod_{\ell=1}^{d_1} |\lambda_\ell(z)|^N$ are continuous, where $\{\lambda_\ell(z)\}_{\ell=1}^d$ are as in Lemma \ref{lem:toep-det-unif-lbd}.  
\end{lemma}

It is obvious that Lemma \ref{lem:toep-det-unif-lbd} follows from Lemmas \ref{lem:toep-det-unif-lbd-1} and \ref{lem:root-cont}. To prove Lemma \ref{lem:root-cont} we see that the continuity of the map $z \mapsto \det(T_N({\bm a}(z)))$ is obvious. The continuity of the other map is a consequence of Lemma \ref{lem:mathfrak-f} (applied with $\gd=0$). We next prove Lemma \ref{lem:toep-det-unif-lbd-1}.

\begin{proof}[Proof of Lemma \ref{lem:toep-det-unif-lbd-1}]
It was noted in \cite[proof of Theorem 1]{BD} that 
\begin{equation}\label{eq:bd-schur}
\det (T_N({\bm a}(z))) = (-1)^{Nd_1} a_{d_1}^N \cdot S_{\mathfrak{m}}(\lambda_1(z), \lambda_2(z),\ldots,\lambda_d(z)),
\end{equation}
where the partition $\mathfrak{m}$ is given by
\[
\mathfrak{m}:= (\underbrace{N,N,\ldots,N}_{d_1}, \underbrace{0,0,\ldots,0}_{d_2}).
\]
To evaluate the \abbr{RHS} of \eqref{eq:bd-schur} we use the representation \eqref{eq:schur}. The denominator of \eqref{eq:schur} is the determinant of the standard Vandermonde matrix. Hence, to complete the proof we expand the determinant in the numerator using Laplace's expansion, find the dominant term, and show that the sum of the 
{other}  terms is of smaller order. 

Fix $z \notin \bar\cN$, implying that
the roots $\{\lambda_\ell(z)\}_{\ell=1}^d$ of the polynomial equation $P_{z,{\bm a}}(\cdot)=0$ are all distinct. Now applying Laplace's expansion of the determinant we find that 
\begin{equation}
\det V_{\mathfrak{m}} (z) = (-1)^{\frac{d_1(d_1+1)}{2}}\sum_{M \in \binom{[d]}{d_1}} (-1)^{(\sum_{i \in M} i )}\det (V_{\mathfrak m}(z)[[d_1]; M]) \cdot \det (V_{\mathfrak m}(z)[[d]\setminus [d_1]; \bar M]),
\end{equation}
where we denote $V_{\mathfrak m}(z):= V_{\mathfrak m}(\lambda_1(z),\lambda_2(z),\ldots, \lambda_d(z))$ and $\bar M:= [d]\setminus M$. Recalling the definition of $V_{\mathfrak m}(z)$ it follows 
that for any $M \in \binom{[d]}{d_1}$,
\begin{align}\label{eq:vandermonde-1}
{\mathfrak V}_{\mathfrak{m}}(z,M)  := \, & \det (V_{\mathfrak m}(z)[[d_1]; M]) \cdot \det (V_{\mathfrak m}(z)[[d]\setminus [d_1]; \bar M]) \notag\\
  = \, & \prod_{\ell\in M} \lambda_\ell(z)^{N+d_2} \cdot \prod_{\ell < \ell' \in M} (\lambda_\ell(z) - \lambda_{\ell'}(z)) \cdot  \prod_{\ell < \ell' \in \bar M} (\lambda_\ell(z) - \lambda_{\ell'}(z)).
\end{align}
Moreover,
\begin{equation}\label{eq:vandermonde-2}
\det V_{\bm 0}(\lambda_1(z),\lambda_2(z),\ldots,\lambda_d(z)) = \prod_{\ell < \ell' \in [d]} (\lambda_\ell(z) - \lambda_{\ell'}(z)).
\end{equation}

Setting now $M=[d_1]$ and using \eqref{eq:vandermonde-1}-\eqref{eq:vandermonde-2} together with
\eqref{eq:root-bd-in-z}, we get
\begin{multline}\label{eq:dom-term}
 \left|\frac{\det (V_{\mathfrak m}(z)[[d_1]; [d_1]]) \cdot \det (V_{\mathfrak m}(z)[[d]\setminus [d_1]; [d]\setminus [d_1]])}{V_{\bm 0}(\lambda_1(z),\lambda_2(z),\ldots,\lambda_d(z))} \right| \\  =   \prod_{\ell \in [d_1]} |\lambda_\ell(z)|^{N+d_2} \cdot \prod_{\ell \in [d_1]} \prod_{\ell' \in [d]\setminus[d_1]} {|\lambda_\ell(z) - \lambda_{\ell'}(z)|}^{-1} 
 \\
\ge  \prod_{\ell \in [d_1]} |\lambda_\ell(z)|^{N} \cdot \prod_{\ell \in [d_1]} \prod_{\ell' \in [d]\setminus[d_1]} {\left|1 - \frac{\lambda_{\ell'}(z)}{ \lambda_\ell(z)}\right|}^{-1} \ge 2^{-d^2} \cdot \prod_{\ell \in [d_1]} |\lambda_\ell(z)|^N,
\end{multline}
uniformly on $z \in \cS_0^{-\vep}\setminus \bar\cN$, for some $c>0$. Hence, in light of \eqref{eq:schur} and \eqref{eq:bd-schur}, to obtain a uniform lower bound on $\det T_N({\bm a}(z))$, for $z \in \cS_0^{-\vep}\cap B_\C(0,N^{1/2})\setminus \bar\cN$, it suffices to show that  
\begin{equation}\label{eq:schur-rest}
\frac{\sum_{M \in \binom{[d]}{d_1}\setminus \{[d_1]\}} (-1)^{(\sum_{i \in M} i)} {\mathfrak V}_{\mathfrak m}(z,M)}{\det V_{\bm 0}(\lambda_1(z),\lambda_2(z),\ldots,\lambda_d(z))} = o \left( \prod_{\ell=1}^{d_1} |\lambda_\ell(z)|^N\right),
\end{equation}
uniformly over $z \in \cS_0^{-\vep}\cap B_\C(0,N^{1/2})\setminus \bar\cN$. Using \eqref{eq:vandermonde-1}-\eqref{eq:vandermonde-2} one may try to individually bound each of the terms in the \abbr{LHS} of \eqref{eq:schur-rest}. However, it can be seen that because of the division by the determinant of the Vandermonde matrix and due to the presence of double roots for $z \in \bar\cN$ some of the terms in \abbr{LHS} of \eqref{eq:schur-rest} blow up as $z$ approaches the set $\bar\cN$. 

To overcome this issue we claim that for any $z \notin \bar\cN$, the numerator of the \abbr{LHS} of \eqref{eq:schur-rest} contains a factor 
\begin{equation}\label{eq:factorize}
\prod_{\ell < \ell' \in [d_1]} (\lambda_\ell(z) - \lambda_{\ell'}(z)) \cdot \prod_{\ell < \ell' \in [d]\setminus [d_1]} (\lambda_\ell(z) - \lambda_{\ell'}(z)).
\end{equation}
Turning to prove this claim, fixing any $\ell_0 < \ell_1 \in [d_1]$ we show that $(\lambda_{\ell_0}(z) - \lambda_{\ell_1}(z))$ is a factor of the numerator of the \abbr{LHS} of \eqref{eq:schur-rest}. Then repeating the same argument one can show that the same holds for $\ell_0 < \ell_1 \in [d]\setminus [d_1]$. This gives the claim \eqref{eq:factorize}. 

To show that $(\lambda_{\ell_0}(z)-\lambda_{\ell_1}(z))$ is a factor we fix any $M \in \binom{[d]}{d_1}\setminus \{[d_1]\}$. If $M$ is such that $\ell_0, \ell_1 \in M$ then by \eqref{eq:vandermonde-1} it follows that $(\lambda_{\ell_0}(z) - \lambda_{\ell_1}(z))$ is indeed a factor. Same holds if $\ell_0, \ell_1 \notin M$. So it boils down to showing that $(\lambda_{\ell_0}(z)-\lambda_{\ell_1}(z))$ is a factor of the sum
\[
\sum (-1)^{(\sum_{i \in M} i)} {\mathfrak V}_{\mathfrak m}(z,M),
\]
where the sum is taken over all $M \in \binom{[d]}{d_1}\setminus \{[d_1]\}$ such that exactly one of $\ell_0$ and $\ell_1$ are in $M$. Pick any such $M$ and without loss of generality assume $\ell_0 \in M$. Then we define $M_1$ by replacing $\ell_0$ by $\ell_1$ and keeping the other elements as is. Note that this map is a bijection and moreover $M_1 \ne [d_1]$. We will show that for any $M$ and $M_1$ chosen and defined as above $(\lambda_{\ell_0}(z) - \lambda_{\ell_1}(z))$ is a factor of 
\[
{\mathfrak f}(\lambda_1(z),\lambda_2(z),\ldots,\lambda_d(z)):=(-1)^{(\sum_{i \in M} i)} {\mathfrak V}_{\mathfrak m}(z,M) + (-1)^{(\sum_{i \in M_1} i)} {\mathfrak V}_{\mathfrak m}(z,M_1).
\]
This will prove that $(\lambda_{\ell_0}(z) - \lambda_{\ell_1}(z))$ is a factor of the \abbr{LHS} of \eqref{eq:schur-rest}.  

Now to prove the last claim it suffices to show that if $\lambda_{\ell_0}(z) =\lambda_{\ell_1}(z)$ then ${\mathfrak f}=0$. This follows from the following two observations. First, from the definition of $M_1$ and \eqref{eq:vandermonde-1} we observe that ${\mathfrak V}(z,M)$ equals ${\mathfrak V}(z,M_1)$, upto a change in sign, when $\lambda_{\ell_0}(z)$ is replaced by $\lambda_{\ell_1}(z)$. Second, from \eqref{eq:vandermonde-1} we further note that the change in sign is $\ell_1 - \ell_0-1$ which is the sum total of the number of elements in $M$ and $\bar M$ between $\ell_0$ and $\ell_1$. Thus we have that ${\mathfrak f}=0$ when $\lambda_{\ell_0}(z) = \lambda_{\ell_1}(z)$. So we now have \eqref{eq:factorize}.

From \eqref{eq:vandermonde-1} we have that for every $M \in \binom{[d]}{d_1}$ the determinant ${\mathfrak V}_{\mathfrak m}(z,M)$ is a multivariate polynomial in $\{\lambda_\ell(z)\}_{\ell=1}^d$ with coefficients free of $N$ (only the exponents may depend on $N$). Thus, equipped with \eqref{eq:factorize} and using \eqref{eq:vandermonde-2} we next obtain that
\begin{equation}\label{eq:schur-rest-1}
\frac{\sum_{M \in \binom{[d]}{d_1}\setminus \{[d_1]\}} (-1)^{(\sum_{i \in M} i)} {\mathfrak V}_{\mathfrak m}(z,M)}{\det V_{\bm 0}(\lambda_1(z),\lambda_2(z),\ldots,\lambda_d(z))} = {\mathfrak P}(\lambda_1(z),\lambda_2(z), \ldots, \lambda_d(z)) \cdot \prod_{\ell \in [d_1]} \prod_{\ell \in [d]\setminus [d_1]} (\lambda_\ell(z) - \lambda_{\ell'}(z))^{-1},
\end{equation}
for some multivariate polynomial ${\mathfrak P}(\cdot)$ with coefficients free of $N$. Since the sum in \eqref{eq:schur-rest-1} is taken over $M \ne [d_1]$, using \eqref{eq:vandermonde-1} once more we find that each of the terms in the polynomial ${\mathfrak P}(\cdot)$ is bounded by
\begin{equation}\label{eq:coeff-bd-1}
C \prod_{\ell=1}^{d_1-1} |\lambda_\ell(z)|^{N+d_2} \cdot |\lambda_{d_1+1}(z)|^{N+d_2} \cdot \prod_{\ell=1}^d |\lambda_\ell(z)|^d \le C (1-\vep_0)^{2N} \prod_{\ell=1}^{d_1} |\lambda_\ell(z)|^N\cdot \prod_{\ell=1}^d |\lambda_\ell(z)|^{2d},
\end{equation}
for some constant $C< \infty$ and $\vep_0>0$, where the last step follows from Lemma \ref{lem:root-cont-0}. 

Next we claim that, for some constant $\widetilde C < \infty$ and $|z| >1$,
\begin{equation}\label{eq:root-bd-in-z}
\max_\ell |\lambda_\ell(z)| \le \widetilde C |z|.
\end{equation}
Indeed, the claim is immediate if $d_1=0$, for then $\max_\ell |\lambda_\ell(z)|=O(1/z)$.
Assume therefore $d_1>0$.
Consider any root $\lambda$ of the polynomial equation $P_{z, {\bm a}}(\lambda)=0$. Then,
\[
a_{d_1} \lambda^d = - \corOZ{\sum_{i=-d_2}^{d_1 - 1}} (a_i -z \delta_{i,0}) \lambda^i. 
\]
Therefore, assuming without loss of generality that $|\lambda| \ge 1$ and
using the triangle inequality, we find that 
\[
|a_{d_1}| |\lambda|^d \le C' |z| \cdot |\lambda|^{d-1},
\]
for some $C' < \infty$, yielding the claim 
\eqref{eq:root-bd-in-z} \corOZ{also in case $d_1>0$}. This, in particular implies that the \abbr{RHS} of \eqref{eq:coeff-bd-1} is upper bounded by 
\begin{equation}\label{eq:coeff-bd-2}
C (1-\vep_0)^{2N} \prod_{\ell=1}^{d_1} |\lambda_\ell(z)|^N\cdot (\widetilde C |z|)^{2d^2}.
\end{equation}
Since there are at most $O(N^d)$ terms in the polynomial ${\mathfrak P}(\cdot)$ using Lemma \ref{lem:root-cont-0} again, from \eqref{eq:schur-rest-1} and the bound in \eqref{eq:coeff-bd-2} we deduce that
\[
\left|\frac{\sum_{M \in \binom{[d]}{d_1}\setminus \{[d_1]\}} (-1)^{(\sum_{i \in M} i)} {\mathfrak V}_{\mathfrak m}(z,M)}{\det V_{\bm 0}(\lambda_1(z),\lambda_2(z),\ldots,\lambda_d(z))} \right| \le (1-\vep_0)^N \cdot \prod_{\ell=1}^{d_1} |\lambda_\ell(z)|^N, 
\]
uniformly over $z \in \cS_0^{-\vep}\cap B_\C(0,N^{1/2}) \setminus \bar\cN$, for all large $N$. This yields \eqref{eq:schur-rest} and thus the proof is complete.
\end{proof}
Before proving Theorem \ref{thm:no-outlier}, 
{we sketch}
the proof of Lemma \ref{lem:toep-det-gd-bd}, 
{which is similar to that}
of Lemma \ref{lem:toep-det-unif-lbd}. 

\begin{proof}[Proof of Lemma \ref{lem:toep-det-gd-bd}]
Consider the case $\gd >0$. Recalling Definition \ref{dfn:toep-shifted} we find that 
\[
T_N({\bm a }(z))[X_\star^c;Y_\star^c]= T_{N-\gd}({\bm a}, z; d_1-\gd).
\]
It also follows from there that the polynomial associated with the symbol of the Toeplitz matrix $T_{N-\gd}({\bm a}, z; d_1-\gd)$ is $P_{z,{\bm a}}(\cdot)$. Therefore, for $z \notin \bar\cN$ it does not have any double roots. Moreover, for $z \in \cS_\gd^{-\vep}$ the number of roots of $P_{z,{\bm a}}(\lambda)=0$ that are greater than one in moduli is $d_1 -\gd$ and it equals the maximal positive degree of the Laurent polynomial associated with the Toeplitz matrix $T_{N-\gd}({\bm a}, z; d_1-\gd)$. So we are in the set up of Lemma \ref{lem:toep-det-unif-lbd}.
Therefore proceeding as in the proof of Lemma \ref{lem:toep-det-unif-lbd} we deduce that
\[
\inf_{z \in \cS_\gd^{-\vep}\cap B_\C(0,R) \setminus \bar\cN}\left|\frac{\det(T_N({\bm a}(z))[X_\star^c;Y_\star^c])}{a_{d_1}^{N-|\gd|}\prod_{\ell=1}^{d_1-\gd} \lambda_\ell(z)^{N-|\gd|}}\right| \ge c_0^\star, \qquad \text{ for all large } N,
\]
and some constant $c_0^\star>0$. 

{We claim that for
$\gd \ne 0$, the set $\cS_\gd^{-\vep}$ is a bounded set.
Indeed,
from the definition of $\{\cS_\gd\}_{\gd=-d_2}^{d_1}$ we have that $\cup_{\gd=-d_2}^{d_1}\cS_\gd = \C \setminus {\bm a}(\mathbb{S}^1)$. 
Hence, recalling \eqref{eq:cS_0-spec}, we deduce that 
$\cS_\gd \subset {\rm spec} \, T({\bm a})$, for any $\gd \ne 0$. As the spectrum of $T({\bm a})$ is contained in a disk of radius at most $\|T({\bm a})\|$ centered at zero, the claim follows.}

The last claim
together with \eqref{eq:root-bd-in-z} we therefore have that 
\[
\inf_{z \in \cS_\gd^{-\vep} \setminus \bar\cN}\left|\frac{\det(T_N({\bm a}(z))[X_\star^c;Y_\star^c])}{a_{d_1}^{N-|\gd|}\prod_{\ell=1}^{d_1 -\gd} \lambda_\ell(z)^{N+d_2}}\right| \ge \bar c_0^\star, \qquad \text{ for all large } N,
\]
and some other constant $\bar c_0^\star >0$. 

A similar reasoning as in the proof of Lemma \ref{lem:root-cont} shows that the map $z \mapsto \prod_{\ell=1}^{d_1-\gd} |\lambda_\ell(z)|^{N-|\gd|}$ is continuous on $\cS_\gd^{-\vep}$. Hence combining the last two observations we derive the desired uniform lower bound for $\gd >0$. The proof for the 
case $\gd <0$ is similar.
\end{proof}

We finally prove Theorem \ref{thm:no-outlier}.

\begin{proof}[Proof of Theorem \ref{thm:no-outlier}]
We begin by recalling \eqref{eq:cS_0-spec}, which implies that 
\[
(({\rm spec} \ T({\bm a}))^c)^{-\vep}  = \cS_0^{-\vep}.
\] 
We write
\[
\widetilde P_k(z):= \frac{P_k(z)}{a_{d_1}^N \prod_{\ell=1}^{d_1}\lambda_\ell(z)^N},
\]
for $k=1,2,\ldots, N$, where we recall the definition of $P_k(z)$ from \eqref{eq:P_k-z}. Recalling \eqref{eq:hatP-k} we note that 
\[
|\widetilde P_k(z)| \le |\widehat P_k(z)| \cdot \prod_{i=1}^{d_0} |\lambda_i(z)|^{k+d_2},
\]
for any $k=1,2,\ldots, N$ and $z \in \cS_0$. Thus, by Lemma \ref{lem:rouche-multi-gr-d_0} and the Cauchy-Schwarz inequality, and upon proceeding similarly to \eqref{eq:hat-det-decompose} it follows that 
\[
\E \left[ \sup_{z \in \cS_0^{-\vep}} \left|\sum_{k=1}^N \widetilde P_k(z)\right|^2 \right] \le \left( \sum_{k=1}^N \left\{ \E\left[\sup_{z \in \cS_0^{-\vep}} |\widetilde P_k(z) |^2 \right]\right\}^{1/2} \right)^2 \le N^{-\eta_0/2},
\]
for all large $N$, where $\eta_0 >0$ is as in Lemma \ref{lem:rouche-multi-gr-d_0}. Hence, by Markov's inequality
\[
\P\left( \sup_{z \in \cS_0^{-\vep}} \left|\frac{\sum_{k=1}^N P_k(z)}{a_{d_1}^N \prod_{\ell=1}^{d_1} \lambda_\ell(z)^N}\right| \le N^{-\eta_0/8} \right) \ge 1 - N^{-\eta_0/4},
\]
for all large $N$. Therefore, recalling $P_0(z)= \det (T_N({\bm a}(z)))$ and using Lemma \ref{lem:toep-det-unif-lbd} we derive that on an event of probability at least $1- N^{-\eta_0/4}$, we have
\[
|P_0(z)| \ge 2 \left|\sum_{k=1}^N P_k(z)\right|, \qquad \text{ for all } z  \in (({\rm spec} \ T({\bm a}))^c)^{-\vep}\cap B_\C(0,N^{1/2}),
\]
for all large $N$. As the maps $z \mapsto \det(T_N({\bm a}(z))+\Delta_N)$ and $P_0(z)$ \corAB{are analytic}, an application of Rouch\'e's theorem further yields that on the same event, the number of roots of $\det (T_N({\bm a}(z))+\Delta_N)=0$ and $P_0(z)=\det (T_N({\bm a}(z)))=0$ are same on the interior of the bounded set $(({\rm spec} \ T({\bm a}))^c)^{-\vep} \cap B_\C(0,N^{1/2})$. Since $|\lambda_1(z)| \ge |\lambda_2(z)| \ge \cdots \ge |\lambda_{d_1}(z)| >1$ on $\cS_0$ Lemma \ref{lem:toep-det-unif-lbd} implies that there are no roots of the equation $P_0(z)=0$ in $\cS_0$. So 
\begin{equation}\label{eq:no-outlier-21}
\P\left( L_N\left(\left( \left({\rm spec} \ T({\bm a})\right)^c\right)^{-\vep}\cap (B_\C(0,N^{1/2}))^0 \right) =0\right) =1 -o(1).
\end{equation}
To complete the proof we recall the spectral radius (i.e.~the maximum modulus eigenvalue) of a matrix bounded by it operator norm. Therefore using the triangle inequality we find that
\[
\E \| T_N({\bm a}) + \Delta_N\| \le \|T_N({\bm a})\| + N^{-\gamma} \E \|E_N\| \le \|T_N({\bm a})\| + N^{-\gamma} \E \|E_N\|_{{\rm HS}} \le C'(1+N^{-\gamma +1}) \le N^{1/2 -\vep'},
\]
for some $\vep' >0$, all large $N$, and any $\gamma >\frac12$, where $\| \cdot \|$ and $\| \cdot\|_{{\rm HS}}$ denotes the operator norm and the Hilbert-Schmidt norm respectively. So by Markov's inequality
\begin{equation}\label{eq:no-outlier-22}
\P\left(L_N\left(B_\C(0, N^{\frac12-\frac{\vep'}{2}})^c\right) >0\right) \le \P\left(\|T_N({\bm a}) +\Delta_N\| \ge N^{\frac12-\frac{\vep'}{2}} \right) \le N^{-\frac{\vep'}{2}}.
\end{equation}
Combining \eqref{eq:no-outlier-21}-\eqref{eq:no-outlier-22} the proof is now complete.
\end{proof}

\appendix

\section{The spectral radius of $E_N$}\label{sec:spec-rad}
In this short section we show that the decomposition \eqref{eq;spec-rad-new} used in the proofs of Theorems \ref{thm:no-outlier} and \ref{thm:outlier-law} can be adapted to prove the following. 
\begin{proposition}
\label{prop-specrad}
Let $\{E_N\}_{N \in \N}$ be a sequence of $N \times N$ random matrices with independent complex-valued entries of mean zero and unit variance. Denote $\varrho_N$ to be the spectral radius of $N^{-1/2} E_N$, i.e.~the maximum modulus eigenvalue of $N^{-1/2} E_N$. Then the sequence $\{\varrho_N\}_{N \in \N}$ is tight.
\end{proposition} 
We remark that Proposition \ref{prop-specrad} seems to be contained in Theorem \ref{thm:no-outlier}. However, formally the
latter cannot be applied since it would require one to take $\bm a \equiv 0$, while throughout the paper (and in particular, in the proof of Theorem
\ref{thm:no-outlier}), we assume that $\bm a$ is a nontrivial Laurent polynomial.

If the entries of $E_N$ are i.i.d.~having a finite $(2+\delta)$-th moment and possessing a symmetric law then it is known that $\varrho_N\to 1$ in probability, see \cite{bordenave}, while the operator norm of $N^{-1/2} E_N$
blows up as soon as the fourth moment of the entries is infinite. It is conjectured in \cite{bordenave}
 that in the critical case of finiteness of second moments, the convergence in probability to one still
holds. Proposition \ref{prop-specrad} is a weak form of the conjecture with elementary proof.
\begin{proof} Set $\Delta_N= N^{-1/2} E_N$. We decompose 
$$ \det(\Delta_N-zI_N)=(-z)^N+\sum_{k=1}^{N}  (-z)^{N-k}P_k$$
where
 \begin{equation}\label{eq:P_k-z-new}
P_k:=\sum_{\substack{X\subset [N]\\ |X|=k}} N^{- k/2} \det (E_N[X; X]),
\end{equation}
compare with \eqref{eq:P_k-z}. Note that $\Var(P_k)\leq 1$, while $\E P_k=0$. Therefore, for a fixed constant $\bar C$,
we have that $\P(|P_k|>\bar C^k)\leq \bar C^{-2k}$. So, setting $\mathcal{A}_0=\cup_{k=1}^\infty \{ |P_k|>\bar C^k\}$, it yields that
$$\P(\mathcal{A}_0)\leq \frac{2}{ \bar C^2}.$$
Note that on $\mathcal {A}_0^c$ we have that for $z$ with $|z|>4 \bar C$,
$$ \frac{|z^N|}{|\sum_{k=1}^{N}  (-z)^{N-k}P_k|}\geq \frac{1}{\sum_{k=1}^N 4^{-k}}\geq 3.$$
This in particular implies that there can be no zero of  $\det(\Delta_N-zI_N)$ with modulus larger than $4\bar C$. Thus the  claim follows.
\end{proof}


\begin{thebibliography}{99}



\bibitem{Al}Per~Alexandersson.
\newblock Schur Polynomials, Banded Toeplitz Matrices and Widom's Formula.
\newblock {\em The Electronic Journal of Combinatorics}, {\bf 22(4)}, paper no.~22, 2012.

\bibitem{BPZ-twisted}Anirban~Basak, Elliot~Paquette, and Ofer~Zeitouni.
\newblock Regularization of non-normal matrices by Gaussian noise - the banded Toeplitz and twisted Toeplitz cases.
\newblock {\em Forum of Mathematics, Sigma}, {\bf 7}, E3, 2019.

\bibitem{BPZ-non-triang}Anirban~Basak, Elliot~Paquette, and Ofer~Zeitouni.
\newblock Spectrum of random perturbations of Toeplitz matrices with finite symbols.
\newblock {\em Transactions of American Mathematical Society}, to appear.

\bibitem{Bax-Sch}Glenn~Baxter and Palle~Schmidt. 
\newblock Determinants of a certain class of non-Hermitian Toeplitz matrices. 
\newblock {\em Mathematica Scandinavica}, {\bf 9}, 122--128, 1961.



\bibitem{BC}Charles~Bordenave and Mireille~Capitaine.
\newblock Outlier Eigenvalues for Deformed I.I.D.~Random Matrices.
\newblock {\em Communications on Pure and Applied Mathematics}, {\bf 69(11)}, 2131--2194, 2016.

\bibitem{bordenave}
  Charles~Bordenave, Pietro~Caputo, Djalil~Chafa\"{i}, and Konstantin~Tikhomirov.
\newblock On the spectral radius of a random matrix: An upper bound without fourth moment.
\newblock {\em Annals Probab.} {\bf 46}, 2268--2286, 2018.
  
\bibitem{bottcher-finite-band}Albrecht~B\"{o}ttcher and Sergei~M.~Grudsky.
\newblock {\em Spectral Properties of Banded Toeplitz Matrices}.
\newblock Vol.~96, Siam, 2005.  

\bibitem{brillinger}David~R.~Brillinger.
\newblock The Analyticity of the Roots of a Polynomial as Functions of the Coefficients.
\newblock {\em Mathematics Magazine}, {\bf 39(3)}, 145--147, 1966.

\bibitem{BD}Daniel~Bump and Persi~Diaconis.
\newblock Toeplitz minors.
\newblock {\em Journal of Combinatorial Theory, Series A}, {\bf 97}, 252--271, 2002.  

\bibitem{DV}D.~J.~Daley and D.~Vere-Jones. 
\newblock {\em An introduction to the theory of point processes: volume II:~general theory and structure}. 
\newblock Springer Science \& Business Media, 2007.

\bibitem{DH}
  Edward~B.~Davies and Mildred~Hager.
  \newblock Perturbations of Jordan matrices.
  \newblock {\em J. Approx. Theory} {\bf 156}, 82--94, 2009.
  
\bibitem{EK}
Alan Edelman and Eric Kostlan.
\newblock  How many zeros of a random polynomial
are real?
\newblock {\em  Bulletin of the American Mathematical Society} {\bf 32}, 1--37, 1995.


\bibitem{HKPV}John~B.~Hough, Manjunath~Krishnapur, Yuval~Peres, and Balint~Vir\'{a}g.
\newblock {\em Zeros of Gaussian analytic functions and determinantal point processes}.  University Lecture Series, {\bf 51}, American Mathematical Society, Providence, RI, 2009.

\bibitem{NV}
  St\'{e}phane Nonnenmacher and Martin Vogel.
  \newblock{Local eigenvalue statistics of one dimensional random nonseladjoint pseudodifferential operators}.
 \newblock {\em ArXiv preprint}, arXiv:1711.05850 (2017).
  \bibitem{Rataj}
  Jan~Rataj and Ludek~Zaj\'{i}\v{c}ek.
  \newblock On the structure of sets with positive reach.
  \newblock {\em Mathematische Nachrchten} {\bf 290(11--12)}, 1806--1829, 2017.
  
  \bibitem{Sh}Tomoyuki~Shirai. 
  \newblock Limit theorems for random analytic functions and their zeros.
  \newblock {\em Functions in number theory and their probabilistic aspects}, 335--359, RIMS K\^{o}ky\^{u}roku Bessatsu, B34, Res.~Inst.~Math.~Sci.~(RIMS), Kyoto, 2012.

\bibitem{SV}
  Johannes~Sj\"{o}strand and Martin~Vogel.
  \newblock Large bi-diagonal matrices and random perturbations.
  \newblock {\em Journal of spectral theory} {\bf 6}, 977--1020, 2016.
\bibitem{SV1}
  Johannes~Sj\"{o}strand and Martin~Vogel.
  \newblock Interior eigenvalue density of large bi-diagonal matrices subject to random perturbations.
  \newblock {\em ArXiv preprint}, arXiv:1604.05558.
  
 \bibitem{SV2}Johannes~Sj\"{o}strand and Martin~Vogel.
 \newblock Interior eigenvalue density of Jordan matrices with random perturbations.
 \newblock In {\em Analysis Meets Geometry}, pp.~439--466, Birkh\"{a}user, Cham, 2017. 
  
  \bibitem{SV3}Johannes~Sj\"{o}strand and M.~Vogel.
\newblock Toeplitz band matrices with small random perturbations.
\newblock {\em ArXiv preprint}, arXiv:1901.08982.

\bibitem{Sn}Piotr~\'{S}niady.
\newblock Random regularization of Brown spectral measure.
\newblock {\em Journal of Functional Analysis} {\bf 193}, 291--313, 2002.

\bibitem{Stanley}Richard~P.~Stanley.
\newblock {\em Enumerative Combinatorics}, Volume 2.
\newblock Cambridge University Press, 1999.




\bibitem{trefethen}Lloyd~M.~Trefethen and Mark~Embree.
\newblock {\em Spectra and pseudospectra:~the behavior of nonnormal matrices and operators}.
\newblock Princeton University Press, 2005.

\bibitem{Trefethen1991} Lloyd~N.~Trefethen.
  \newblock {Pseudospectra of matrices},
  \newblock {in Numerical Analysis 1991},
  \newblock {D.~F.~Griffiths and G.~A.~Watson, eds.},
  \newblock {Longman Sci. Tech., Harlow, UK} {\bf 260}, 234--266, 1992.
  

\bibitem{trench}William~F.~Trench.
\newblock On the eigenvalue problem for Toeplitz band matrices.
\newblock {\em Linear Algebra and its Applications}, {\bf 64}, 199--214, 1985.

\bibitem{W}Phillip~M.~Wood.
\newblock Universality of the ESD for a fixed matrix plus small random noise: A stability approach.
newblock {\em Annales de l'Institut Henri Poincar\'{e}, Probabilit\'{e}s et Statistiques} 
{\bf 52}, 1877--1896, 2016.





\bibitem{whitney}Hassler~Whitney.
\newblock {\em Complex Analytic Varieties}.
\newblock Addison-Wesley Pub.~Co. Reading, Massachusetts, 1972.
	
\end{thebibliography}
\end{document}